\documentclass[a4paper,twoside,10pt]{article}

\usepackage[a4paper,left=3cm,right=3cm, top=3cm, bottom=3cm]{geometry}
\usepackage[latin1]{inputenc}
\usepackage{mathrsfs}
\usepackage{graphicx}
\usepackage{epstopdf}
\usepackage{appendix}
\usepackage{amsmath}
\usepackage{amsthm}
\usepackage{amssymb}
\usepackage{hyperref}
\usepackage{stmaryrd}
\usepackage{float}
\usepackage{calrsfs}
\DeclareMathAlphabet{\pazocal}{OMS}{zplm}{m}{n}
\usepackage{bigints}
\usepackage{cite}
\usepackage{color}
\usepackage[abs]{overpic}
\usepackage[font=footnotesize,labelfont=bf]{caption}
\usepackage{cases}
\usepackage{tikz}
\usepackage{algorithmic}
\usepackage{rotating}
\usepackage{blkarray}
\usetikzlibrary{matrix,calc,arrows,shapes,snakes, shapes.geometric}
\usepackage[author={Lorenzo}]{pdfcomment}
\usepackage{soul,xcolor}
\usepackage{verbatim}
\usepackage{graphicx}
\usepackage{arydshln}
\usepackage{subcaption}
\usepackage{algorithm}

\DeclareFontFamily{OT1}{pzc}{}
\DeclareFontShape{OT1}{pzc}{m}{it}{<-> s * [1.10] pzcmi7t}{}
\DeclareMathAlphabet{\mathpzc}{OT1}{pzc}{m}{it}
\DeclareMathOperator{\diam}{diam}

\restylefloat{table}
\theoremstyle{plain}
\newtheorem{thm}{Theorem}[section]
\newtheorem{cor}[thm]{Corollary}
\newtheorem{lem}[thm]{Lemma}
\newtheorem{prop}[thm]{Proposition}
\theoremstyle{definition}

\newtheorem{exmp}{Example}[section]
\theoremstyle{remark}
\newtheorem{remark}{Remark}

\newcommand{\lm}[1]{{\color{red}{#1}}} %colors
\newcommand{\eremk}{\hbox{}\hfill\rule{0.8ex}{0.8ex}}

\newcommand{\GammaD}{\Gamma_D}
\newcommand{\GammaN}{\Gamma_N}
\newcommand{\f}{f}

\newcommand{\n}{\mathbf n}
\renewcommand{\ne}{\n_\e}
\newcommand{\nei}{\n_{\ei}}
\newcommand{\nE}{\n_\E}
\newcommand{\nOmega}{\n_\Omega}
\newcommand{\V}{V}
\newcommand{\VgD}{\V_{\gD}}
\newcommand{\Vz}{\V_0}
\newcommand{\Vn}{\V_n}

\newcommand{\VnE}{\Vn(\E)}
\renewcommand{\a}{a}
\newcommand{\aE}{\a^\E}
\newcommand{\an}{\a_n}
\newcommand{\anE}{\an^\E}
\newcommand{\SE}{S^\E}
\newcommand{\Abf}{\mathbf A}
\newcommand{\NOmega}{N_\Omega}
\newcommand{\robrack}[1]{{\left({#1} \right)}} % round brackets trick
\newcommand{\cbrack}[1]{{\left\{{#1} \right\}}} % curly brackets trick
\newcommand{\E}{K}

\newcommand{\e}{e}
\newcommand{\ei}{\e(i)}
\newcommand{\g}{g}
\renewcommand{\ij}{_{i,j}} % shortcut i,j
\renewcommand{\S}{\mathcal S}
\newcommand{\SA}{\S_{\Abf}}
\newcommand{\SAE}{\SA^\E}
\newcommand{\SAe}{\SA^\e}
\newcommand{\SAei}{\SA^{\ei}}
\newcommand{\En}{\mathcal E_n}
\newcommand{\EnI}{\En^I}
\newcommand{\EnB}{\En^B}
\newcommand{\EnE}{\En^\E}
\newcommand{\Eno}{\mathcal E_n^1}
\newcommand{\Entw}{\mathcal E_n^2}
\newcommand{\taun}{\mathcal T_n}
\newcommand{\tauno}{\mathcal T_n^1}
\newcommand{\tauntw}{\mathcal T_n^2}
\newcommand{\taunth}{\mathcal T_n^3}
\newcommand{\un}{u_n}
\newcommand{\vn}{v_n}
\newcommand{\EE}{\mathcal E^\E}
\newcommand{\p}{p}
\newcommand{\h}{h}
\newcommand{\he}{\h_\e}
\newcommand{\hei}{\h_{\e(i)}}
\newcommand{\hE}{\h_\E}
\newcommand{\Pinabla}{\widetilde\Pi^{\nabla,\E}_\p}
\newcommand{\Piboldnabla}{\boldsymbol \Pi^{\nabla,\E}_\p}
\newcommand{\Piboldnablapmo}{\boldsymbol \Pi^{\nabla,\E}_{\p-1}}
\newcommand{\Pinablaglob}{\widetilde \Pi^{\nabla}_\p}
\newcommand{\Pize}{\widetilde \Pi^{0,\e}_{\p-1}}
\newcommand{\Piboldze}{\boldsymbol \Pi^{0,\e}_{\p-1}}

\newcommand{\PizE}{\Pi^{0,\E}_{\p-2}}
\newcommand{\boldalpha}{\boldsymbol \alpha}

\newcommand{\qpmt}{q_{\p-2}}

\newcommand{\malpha}{m_{\alpha}}

\newcommand{\mbeta}{m_{\beta}}
\newcommand{\mbetae}{\mbeta^\e}
\newcommand{\malphae}{m_{\alpha}^\e}
\newcommand{\mtildealpha}{\widetilde m_{\alpha}}
\newcommand{\mtildealphaE}{\widetilde m_{\alpha}^\E}
\newcommand{\mtildebeta}{\widetilde m_{\beta}}
\newcommand{\mtildealphae}{\widetilde m_{\alpha}^\e}
\newcommand{\mtildebetae}{\widetilde m_{\beta}^\e}
\newcommand{\mbar}{\overline m}
\newcommand{\mbare}{\overline m^\e}
\newcommand{\mbaralphae}{\overline m_{\alpha}^\e}
\newcommand{\mbarbetae}{\overline m_{\beta}^\e}
\newcommand{\mbargammae}{\overline m_{\gamma}^\e}

\newcommand{\Lcal}{\mathcal L}
\newcommand{\Pbbtilde}{\widetilde{\mathbb P}}
\renewcommand{\g}{g}
\newcommand{\gD}{\g_D}
\newcommand{\gN}{\g_N}
\newcommand{\Ncaln}{\mathcal N_n}
\newcommand{\upi}{u_\pi}
\newcommand{\uI}{u_I}
\newcommand{\deltan}{\delta_n}
\newcommand{\uz}{u_0}
\newcommand{\Pibfstar}{\mathbf \Pi^* }
\newcommand{\Gbf}{\mathbf G}
\newcommand{\Gbarbf}{\overline{\mathbf G}}
\newcommand{\GSbf}{\mathbf {GS}}
\newcommand{\Gbftilde}{\widetilde {\Gbf}}
\newcommand{\GbftildeA}{\widetilde {\Gbf}^A}
\newcommand{\GbftildeB}{\widetilde {\Gbf}^B}
\newcommand{\Pibf}{\mathbf \Pi}
\newcommand{\Sbf}{\mathbf S}
\newcommand{\Ibf}{\mathbf I}
\newcommand{\Bbf}{\mathbf B}
\newcommand{\Bbarbf}{\overline{\mathbf B}}
\newcommand{\BbfA}{\mathbf B^A}
\newcommand{\BbfB}{\mathbf B^B}
\newcommand{\BbfC}{\mathbf B^C}
\newcommand{\BbfD}{\mathbf B^D}
\newcommand{\BbfE}{\mathbf B^E}
\newcommand{\BbfF}{\mathbf B^F}
\newcommand{\BbfG}{\mathbf B^G}
\newcommand{\BbfH}{\mathbf B^H}
\newcommand{\BbfI}{\mathbf B^I}
\newcommand{\BbarbfAB}{\overline{\mathbf B}^{A,B}}
\newcommand{\BbarbfC}{\overline{\mathbf B}^C}
\newcommand{\BbarbfDEGH}{\overline{\mathbf B}^{D,E,G,H}}
\newcommand{\BbarbfFI}{\overline{\mathbf B}^{F,I}}
\newcommand{\Dbf}{\mathbf D}
\newcommand{\Dbarbf}{\overline{\mathbf D}}
\newcommand{\DbfA}{\mathbf D^A}
\newcommand{\DbfB}{\mathbf D^B}
\newcommand{\DbfC}{\mathbf D^C}
\newcommand{\DbarbfAB}{\overline{\mathbf D}^{A,B}}
\newcommand{\DbarbfC}{\overline{\mathbf D}^C}
\newcommand{\dof}{\text{dof}}
\newcommand{\Nun}{\mathcal V_n}
\newcommand{\nPE}{n_\p ^\E}

\newcommand{\nPEmt}{n_{\p-2} ^\E}
\newcommand{\nPEpmt}{n_{\p-2}^\E}
\newcommand{\nPe}{n_{\p-1}^\e}
\newcommand{\ntildeE}{\widetilde n ^\E_{\p}}
\newcommand{\ntildee}{\widetilde n ^\e_{\p-1}}
\newcommand{\NtildeE}{\widetilde N ^\E}
\newcommand{\alphabold}{\boldsymbol \alpha}
\newcommand{\xbold}{\mathbf x}
\newcommand{\card}{\text{card}}
\newcommand{\NV}{N_V}
\newcommand{\Lbbe}{\mathbb L^\e}
\newcommand{\Lbbei}{\mathbb L^{\e(i)}}
\newcommand{\alphatilde}{\widetilde \alpha}
\newcommand{\xbfE}{\mathbf x_\E}
\newcommand{\gammatilde}{\widetilde \gamma}
\newcommand{\qp}{q_{\p}}
\newcommand{\qppt}{q_{\p+2}}
\newcommand{\Mbb}{\mathbb M}
\newcommand{\Nbb}{\mathbb N}
\newcommand{\Pbb}{\mathbb P}
\newcommand{\Rbb}{\mathbb R}
\renewcommand{\mp}{m^{[\p]}}
\newcommand{\mtilde}{\widetilde m}

\newcommand{\ntilde}{\widetilde n}
\newcommand{\Ks}{K_s}
\newcommand{\Ksmt}{K_{s-2}}
\newcommand{\lambdaz}{\lambda^0}
\newcommand{\Gbfz}{\mathbf G ^0}
\newcommand{\Gbfze}{\mathbf G ^{0,\e}}

\newcommand{\Gbarbfze}{\overline{\mathbf G} ^{0,\e}}

\newcommand{\bbfze}{\mathbf b ^{0,\e}}
\newcommand{\Bbfze}{\mathbf B ^{0,\e}}

\newcommand{\Bbarbfze}{\overline{\mathbf B}^{0,\e}}
\newcommand{\Lambdabfz}{\boldsymbol \Lambda ^{0,\e}}

\newcommand{\lambdabar}{\overline \lambda}
\newcommand{\phibar}{\overline \varphi}

\begin{tiny}
\author{
\normalsize{
}}
\end{tiny}

\title{\Large{Enrichment of the nonconforming virtual element method with singular functions}}
\date{}
\author{E. Artioli\thanks{Dipartimento di Ingegneria Civile e Ingegneria Informatica, Universit\`a di Roma Tor Vergata, 00133 Rome, Italy (artioli@ing.uniroma2.it)}, \quad
L. Mascotto\thanks{Fakult\"at f\"ur Mathematik, Universit\"at Wien, 1090 Vienna, Austria (lorenzo.mascotto@univie.ac.at)}}

\begin{document}
%%%%%%%%%%%%%%%%%%%%%%%%%%%%%%%%%%%%%
\maketitle

\begin{abstract}
\noindent
We construct a nonconforming virtual element method (ncVEM) based on approximation spaces that are enriched with special singular functions.
This enriched ncVEM is tailored for the approximation of solutions to elliptic problems, which have singularities due to the geometry of the domain.
Differently from the traditional extended Galerkin method approach, based on the enrichment of local spaces with singular functions, no partition of unity is employed.
Rather, the design of the method hinges upon the special structure of the nonconforming virtual element spaces.
We discuss the theoretical analysis of the method and support it with several numerical experiments.
We also present an orthonormalization procedure drastically trimming the ill-conditioning of the final system.

\medskip\noindent
\textbf{AMS subject classification}: 65N12, 65N15, 65N30

\medskip\noindent
\textbf{Keywords}: virtual element method, extended Galerkin method, singular function, enrichment, optimal convergence, polygonal mesh
\end{abstract}

%%%%%%%%%%%%%%%%%%%%%%%%%%%%%%%%%%%%%%%%%%%%%%%%%%%%%%%%%%%%%%%%%%%%%%%%%%%
\section{Introduction} \label{section introduction}
%%%%%%%%%%%%%%%%%%%%%%%%%%%%%%%%%%%%%%%%%%%%%%%%%%%%%%%%%%%%%%%%%%%%%%%%%%%
The virtual element method (VEM) is a recent generalization of the finite element method (FEM) to very general polygonal/polyhedral meshes; see~\cite{VEMvolley} and~\cite{nonconformingVEMbasic},
as for the original references of conforming and nonconforming VEM for elliptic problems in primal formulation, respectively.
In this paper, we shall design a modification of the nonconforming VEM.
Differently from several other polytopal methods, virtual element spaces are designed to mimic properties of the solution to the problem under consideration.
This is very much in the spirit of Trefftz methods and renders the VEM extremely similar to the boundary element method-based FEM~\cite{Weisser_book}.
Even on standard triangular and tetrahedral meshes, new elements can be constructed.

The design of special virtual element spaces mimicking the continuous problem has been exploited in various occasions.
Amongst them, we mention the approximation of solutions to
the Stokes equation~\cite{BLV_StokesVEMdivergencefree} with divergence free spaces;
polyharmonic problems~\cite{polyharmonic:VEM:LongChen, polyharmonic:VEM:AMV} with polyharmonic virtual element spaces;
problems with zero right-hand side tackled with the Trefftz VEM~\cite{ncHVEM};
elasticity problems~\cite{artioli2017stress, dassi2019three} with symmetric stresses inserted in the virtual element spaces;
see also~\cite{BLM_VEMsmalldeformation, artioli2017arbitrary, wriggers2017efficient}.

In this paper, we construct special virtual element spaces for the approximation of solutions to elliptic problems, which have singularities due to the geometry of the domain.
Our approach falls within the broad family of extended Galerkin methods based on the enrichment of local spaces with singular functions, such as the extended finite element method (XFEM), see, e.g., \cite{moes1999finite, moes2002extended},
and the generalized finite element method (GFEM); see, e.g., \cite{strouboulis2000design}.
These methods work as follows. Consider an elliptic problem on a polygonal domain with smooth data.
The solution to this problem has an a priori known singular behaviour at the vertices of the domain; see, e.g., \cite{grisvard} and the references therein.
For this reason, a standard FEM converges to the exact solution suboptimally.
In order to cope with this suboptimality, in the extended Galerkin methods, the approximation space is enriched with special singular functions.
A partition of unity is employed in order to patch the local approximation spaces seamlessly.
This enrichment permits to recover an optimal convergence rate of the error of the method; see also~\cite{MR1426012}.

In~\cite{XVEM_2019}, the extended finite element setting of~\cite{moes1999finite} is translated into the virtual element one:
local spaces consisting of polynomials plus singular functions are patched with the aid of a virtual partition of unity, in the spirit of~\cite{Helmholtz-VEM}.

Our approach is different and exploits the structure of virtual element spaces.
Instead of inserting the singular functions in the approximation spaces explicitly and patching the local spaces with a partition of unity, we proceed as follows.
The singular functions, which are typically added to the approximation spaces in the extended Galerkin methods, belong to the kernel of the differential operator appearing in the problem under consideration.
Such singular functions can be inserted into local virtual element spaces, using the fact that they are defined as solutions to local problems with data in the finite dimensional spaces.
By suitably tuning the boundary conditions in such local spaces, we include the singular functions implicitly.
Eventually, the local spaces are patched in a nonconforming fashion.

A first advantage of our approach resides in the flexibility of using polygonal meshes.
On the other hand, the analysis and the implementation of the method hinge upon a minor modification of what is done in the nonenriched nonconforming VEM; see, e.g., \cite{nonconformingVEMbasic}.
Furthermore, the structure of nonconforming spaces allows for the use of techniques suited to damp the ill-conditioning, which typically arises in the extended Galerkin methods.

The method presented in this paper can be extended to more general problems, such as linear and nonlinear elasticity problems.
The extension to the three dimension VEM is straightforward, thanks to the nonconforming structure of the space; see~\cite{cangianimanzinisutton_VEMconformingandnonconforming}.

The enriched virtual element method is based on two main ingredients: local stabilizations and projections onto bulk and face enriched polynomial spaces.
We develop the analysis of the method for arbitrary polynomial order.
We point put that the analysis of the stabilization of the method is still at an embryonic stage.
In fact, we are able to provide an explicit stabilization assuming that the method is enriched with functions that are not ``too singular'';
for instance, we are still not able to provide an explicit stabilization for singular functions arising from, e.g., slit domains.
Moreover, we prove the lower stability bound under a strong assumption related to inverse estimates in enriched polynomial spaces.
Importantly, we provide practical stabilizations, which lead to optimal convergence rate in the numerical experiments.

Importantly, for the approximation of solutions to elliptic problems in primal formulation, we mention that there are two main families of the VEM: conforming and nonconforming VEMs.
For technical reasons, it appears that the latter family is more suited to the enrichment we are going to present.
This is a relevant fact, for the nonconforming VEM has strong links with the Hybrid-High order (HHO) methods and Hybridizable Discontinuous Galerkin method;
see, e.g., \cite{bridging_HHOandHDG, di2019hybrid, di2018discontinuous},
and~\cite{yemm2021design} where the ideas underlying enriched non-conforming methods, like the ncVEM, have been adapted to the HHO setting.
Thus, the proposed enrichment goes beyond the scope of the nonconforming VEM and could be investigated in other settings as well.

\paragraph*{Structure of the paper.}
In Section~\ref{section:problem}, we present the model problem and recall regularity results for elliptic partial differential equations on polygonal domain: we focus on the case, where the singularities attain at the corner of the domain.
We devote Section~\ref{section:VEM} to the design of the enriched virtual element method.
Its theoretical analysis is the topic of Section~\ref{section:error_analysis}. Here, we also discuss some generalizations of the method.
The theoretical results are validated by several numerical experiments in Section~\ref{section:NR},
including an orthonormalization procedure dramatically trimming the ill-conditioning of the final system.
We draw some conclusions in Section~\ref{section:conclusions} and provide the implementation details in Appendices~\ref{appendix:implementation}, \ref{appendix:implementation2}, and~\ref{appendix:Laplacian-polynomial}.

\paragraph*{Notation.}
We employ a standard notation for Sobolev spaces.
Given~$D \subset \mathbb R^2$ a domain and~$s~\in~\mathbb N$, we denote the standard Sobolev space of integer order~$s$ over~$D$ by~$H^s(D)$.
The case~$s=0$ is special: the Sobolev space~$H^0(D)$ is the Lebesgue space~$L^2(D)$.
We endow the Sobolev spaces with the standard inner products and seminorms~$(\cdot, \cdot)_{s,D}$ and~$\vert \cdot \vert_{s,D}$,
and denote the Sobolev norm of order~$s$ by
\[
\Vert \cdot \Vert_{s,D}^2 := \sum_{\ell=0}^s \Vert \cdot \Vert^2_{\ell,D}.
\]
For~$s=1$, it is convenient to write
\[
\a^D(\cdot,\cdot) := (\cdot,\cdot)_{1,D}.
\]
Fractional Sobolev spaces can be defined in several ways. We use the definition of finiteness of the Aronszajin-Gagliardo-Slobodeckij norm; see, e.g., \cite{fractional-guide} and the references therein.
In particular, for any sufficiently smooth~$v$ on~$\partial D$, set
\[
\vert v \vert_{\frac{1}{2},\partial D}^2 := \int_{\partial D} \int_{\partial D} \frac{\vert v(\xi) - v(\eta) \vert^2}{\vert \xi -\eta \vert^{2}} d\xi d\eta, \quad \quad \Vert v\Vert^2_{\frac{1}{2},\partial D} :=\Vert v\Vert^2_{0,\partial D}+ \vert v\vert^2_{\frac{1}{2},\partial D}.
\]
We define the fractional Sobolev space~$H^{\frac{1}{2}}(\partial D)$ as
\[
H^{\frac{1}{2}}(\partial D):= \left\{  v\in L^2(\partial D) \text{ such that } \Vert v \Vert_{\frac{1}{2},\partial D} \text{ is finite}    \right\}.
\]
Negative Sobolev spaces are defined via duality. In particular, we introduce~$H^{-\frac{1}{2}}(\partial D)$ as the dual space of~$H^{\frac{1}{2}}(\partial D)$. This space and its norm read
\[
H^{-\frac{1}{2}} (\partial D) := (H^{\frac{1}{2}} (\partial D))^*,\quad \quad \Vert  v \Vert_{-\frac{1}{2},\partial D} := \sup_{\Vert w \Vert_{\frac{1}{2},\partial D}  =  1} (v,w)_{0,\partial D}.
\]

%%%%%%%%%%%%%%%%%%%%%%%%%%%%%%%%%%%%%%%%%%%%%%%%%%%%%%%%%%%%%%%%%%%%%%%%%%%
\section{Model problem and regularity of the solution} \label{section:problem}
Let~$\Omega \subset \mathbb R^2$ be a polygonal domain with boundary~$\Gamma = \overline{\GammaN} \cup \overline{\GammaD}$,
where~$\GammaD$ is a closed set in the topology of~$\Gamma$ and~$\GammaD\cap \GammaN =\emptyset$.
Denote the outward normal unit vector of~$\Gamma$ by~$\nOmega$.
Let~$\f$ be an analytic source term on~$\Omega$, and ~$\gD$ and~$\gN$ be piecewise smooth functions on~$\GammaD$ and~$\GammaN$.
We allow for slit domains, see Figure~\ref{figure:domains} (left), and domains with internal cuts (or cracks), see Figure~\ref{figure:domains} (right).
\begin{figure}[h]
\begin{center}
\begin{tikzpicture}[scale=4.5] 
\draw[black, very thick, -] (0,0) -- (1,0) -- (1,1) -- (0,1) -- (0,0);
\draw[black, very thick, -] (0.5, 0) -- (0.5, 0.5);
\draw(1.05,-.05) node[black] {{$\Abf_1$}}; \draw(1.05,1.05) node[black] {{$\Abf_2$}};
\draw(-.05,1.05) node[black] {{$\Abf_3$}}; \draw(-.05,-.05) node[black] {{$\Abf_4$}};
\draw(.43,-.05) node[black] {{$\Abf_5=$}};
\draw(.5,.55) node[black] {{$\Abf_6$}};
\draw(.57,-.05) node[black] {{$\,\,\Abf_7$}};
\end{tikzpicture}
\quad\quad\quad\quad\quad\quad\quad\quad
\begin{tikzpicture}[scale=4.5] 
\draw(1.05,-.05) node[black] {{$\Abf_1$}}; \draw(1.05,1.05) node[black] {{$\Abf_2$}};
\draw(-.05,1.05) node[black] {{$\Abf_3$}}; \draw(-.05,-.05) node[black] {{$\Abf_4$}};
\draw(.21,.21) node[black] {{$\Abf_5$}};
\draw(.78,.78) node[black] {{$\Abf_6$}};
\draw[black, very thick, -] (0,0) -- (1,0) -- (1,1) -- (0,1) -- (0,0);
\draw[black, very thick, -] (0.25, 0.25) -- (0.75, 0.75);
\end{tikzpicture}
\end{center}
\caption{\emph{Left panel}: a slit square domain. \emph{Right panel}: a domain with an internal crack. We highlight the vertices and tips of the domain~$\Omega$ in bold letters.
If a cut starts from the boundary of~$\Omega$, then two vertices share the same coordinates.}
\label{figure:domains}
\end{figure}
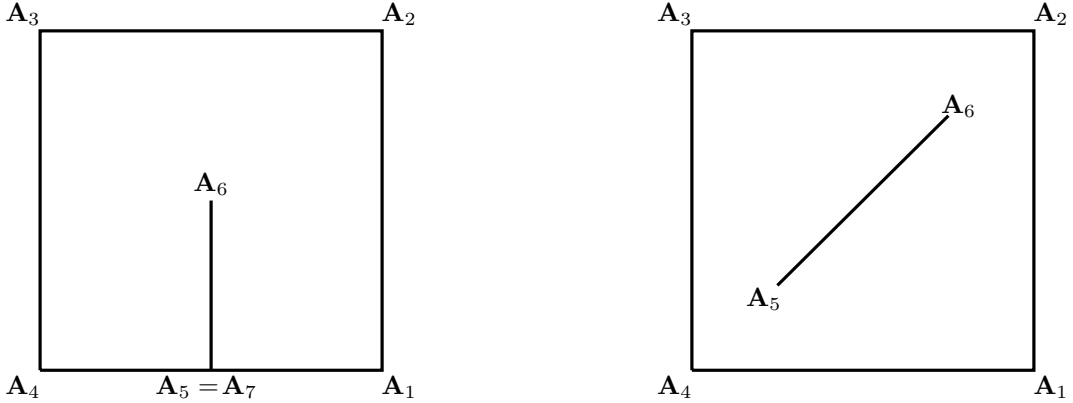

\noindent Consider the following 2D Poisson problem on~$\Omega$:
\begin{equation} \label{strong:form}
\begin{cases}
\text{find } u \text{ such that}\\
-\Delta u = \f 		& \text{in } \Omega\\
\nOmega \cdot \nabla u = \gN 	& \text{on } \GammaN\\
u = \gD 			& \text{on } \GammaD.\\
\end{cases}
\end{equation}
Define \begin{equation} \label{basic:notation}
\begin{split}
& \VgD := H^1_{\gD}(\Omega) := \{ v\in H^1(\Omega)  \mid v=\gD \text{ on } \GammaD\}, \\
& \Vz := H^1_{0}(\Omega) := \{ v\in H^1(\Omega)  \mid v=0 \text{ on } \GammaD \}, \\
& \a(u,v)=\int_\Omega \nabla u \cdot \nabla v \quad \forall u,\, v \in H^1(\Omega).
\end{split}
\end{equation}
In weak formulation, problem~\eqref{strong:form} reads
\begin{equation} \label{weak:form}
\begin{cases}
\text{find } u \in \VgD \text{ such that}\\
\a(u,v) = (\f,v)_{0,\Omega} + (\gN, v)_{0,\GammaN} \quad \forall v\in \Vz.\\
\end{cases}
\end{equation}
Even if the right-hand side and the boundary conditions~$\gD$ and~$\gN$ are (piecewise) analytic, the solution~$u$ to problem~\eqref{weak:form} is not analytic over~$\overline \Omega$ in general.
More precisely, $u$ is the combination of an analytic function and a series of singular terms associated with the corners of the domain and the tips of the cracks;
see, e.g., \cite{SchwabpandhpFEM, babuskaguo_curvilinearhpFEM,grisvard, costabel2002crack} and the references therein.

We recall such an expansion.
Let~$\NOmega$ be the number of vertices and tips of the cracks of~$\Omega$. Denote the set of such vertices and tips by~$\{ \Abf_i \}_{i=1}^{\NOmega}$ and the associated angles by~$\{  \omega_i \}_{i=1}^{\NOmega}$.
When no confusion occurs, we call vertex both a vertex and a tip.
If a crack has one of the two tips on the boundary of the domain~$\Omega$, then two vertices share the same coordinates;
see, e.g., the vertices~$\Abf_5$ and~$\Abf _7$ in Figure~\ref{figure:domains} (left).
The boundary conditions are imposed on the two lips of the cut separately.

We say that the vertex~$\Abf_i$ is a D (N) vertex if $\Abf_i$ is at the interface between two edges in~$\GammaD$ ($\GammaN$).
Otherwise, we say that~$\Abf_i$ is a D-N vertex.
Introduce the singular exponents
\begin{equation} \label{singular exponents}
\alpha\ij =
\begin{cases}
j \frac{\pi}{\omega_i} & \text{if } \Abf_i \text{ is either D or N}\\
\robrack{j-\frac{1}{2}} \frac{\pi}{\omega_i} & \text{if } \Abf_i \text{ is D-N}\\
\end{cases} \quad \quad \forall \, i=1,\dots, \NOmega,\quad j \in \mathbb N.
\end{equation}
To each vertex~$\Abf_i$, $i=1,\dots, \NOmega$, we associate the two (oriented counterclockwise) adjacent edges~$\Gamma_{i(1)}$ and~$\Gamma_{i(2)}$,
and the local set of polar coordinates
\begin{equation} \label{local polar coordinates}
\Abf_i  \quad \xrightarrow{\hspace*{1cm}} \quad \robrack{r_i,\, \theta_i}.
\end{equation}
Next, we introduce the so-called singular functions.
For all~$i=1,\dots, \NOmega$ and~$j \in \mathbb N$, if the singular exponent~$\alpha\ij$ in~\eqref{singular exponents} does not belong to~$\mathbb N$, then we set
\begin{equation} \label{singular:functions:1}
S\ij (r_i,\theta_i) = 
\begin{cases}
r_i^{\alpha\ij} \sin \robrack{\alpha\ij \theta_i} & \text{if } \Gamma _{i(2)}  \subset \GammaD\\
r_i^{\alpha\ij} \cos \robrack{\alpha\ij \theta_i} & \text{otherwise} .\\
\end{cases}
\end{equation}
Instead, if~$\alpha\ij \in \mathbb N$, $i=1,\dots, \NOmega$, $j\in \mathbb N$, then we set
\begin{equation} \label{singular:functions:2}
S\ij (r_i,\theta_i) = 
\begin{cases}
r_i^{\alpha\ij} \robrack{\log(r_i) \sin \robrack{\alpha\ij \theta_i} + \theta_i\cos(\alpha\ij \theta_i)} & \text{if } \Gamma _{i(2)}  \subset \GammaD \\
r_i^{\alpha\ij} \robrack{\log(r_i) \cos \robrack{\alpha\ij \theta_i} + \theta_i\sin(\alpha\ij \theta_i)} & \text{otherwise}. \\
\end{cases}
\end{equation}
If~$\Abf_i$ is either a D or an N vertex, we can easily check that
\begin{equation} \label{regularity singular functions}
S\ij \in H^{1+j\frac{\pi}{\omega_i} - \varepsilon} (\Omega) \quad \quad \forall \varepsilon > 0 \quad \text{arbitrarily small}.
\end{equation}
Moreover, we have
\[
\Delta S\ij = 0 \quad \quad \text{pointwise in $\Omega$}.
\]
\begin{thm} \label{theorem:decomposition:solution}
Given~$s>0$, assume that~$\f \in H^{s-1}(\Omega)$ and~$\gD$ and~$\gN$ are piecewise in~$H^{s+\frac{1}{2}}(\GammaD)$ and in~$H^{s-\frac{1}{2}}(\GammaN)$.
Then, the following decomposition of the solution to problem~\eqref{weak:form} is valid:
\begin{equation} \label{decomposition:solution}
u = \uz + \sum_{i=1}^{\NOmega} \sum_{\alpha\ij < s} c\ij S\ij (r_i,\theta_i),
\end{equation}
where~$\uz \in H^{1+s}(\Omega)$ and~$c\ij \in \mathbb R$;
\end{thm}
\begin{proof}
see, e.g., \cite{babuvska1988regularity, babuvska1989regularity} and the references therein.
\end{proof}
Theorem \ref{theorem:decomposition:solution} states that the solution to problem~\eqref{weak:form} is not analytic in general, but rather has a known singular behaviour at the vertices of the domain and at the tips of the crack.
Such a singular behaviour depends on the magnitude of the angles associated with the vertices of the domain, regardless of the smoothness of the data.
\medskip

For ease of presentation, in the remainder of the paper, we assume that~$u$, the solution to problem~\eqref{weak:form}, is such that
the series of singular functions in~\eqref{decomposition:solution} reduces to a single term associated with a single vertex~$\Abf$.
In other words, we assume that~$u$ decomposes into
\begin{equation} \label{assumption:solution}
u = \uz + \SA.
\end{equation}
In~\eqref{assumption:solution}, $\uz$ denotes an analytic function over~$\overline \Omega$,
whereas~$\SA$ is a singular function of the form either~\eqref{singular:functions:1} or~\eqref{singular:functions:2} with singularity centred at the corner/tip~$\Abf$.
From~\eqref{singular:functions:1} and~\eqref{singular:functions:2}, we have that~$\SA(\lambda r, \theta) = \lambda^\alpha \SA(r,\theta)$, where~$\alpha>0$ depends on~$\SA$ and, consequently, on the geometry of~$\Omega$.
Additionally, we assume that~$\GammaN = \emptyset$  and~$\gD = 0$ as well; see Remark~\ref{remark:Neumann} for further comments on more general cases.

Thus, in weak formulation, the problem we aim to solve reads
\begin{equation} \label{weak:form:simple}
\begin{cases}
\text{find } u \in \V := H^1_0(\Omega) \text{ such that}\\
\a(u,v) = (\f, v) _{0,\Omega} \quad \forall v \in H^1_0(\Omega).\\
\end{cases}
\end{equation}
The analysis of this paper can be generalized:
in Section~\ref{subsection:generalizations}, we discuss how to cope with nonhomogeneous boundary conditions; multiple singularities; 3D problems; general elliptic operators.

We exhibit a couple of examples falling in the setting of assumption~\eqref{assumption:solution}.
\begin{exmp}
Let~$\Omega$ be the L-shaped domain, see Figure~\ref{figure:singular_domains} (left),
\begin{equation} \label{L-shaped}
\Omega = (-1,1)^2 \setminus \left( [0,1) \times (-1,0]    \right).
\end{equation}
According to~\eqref{singular:functions:1}-\eqref{singular:functions:2}, the expected strongest singularity is located at the re-entrant corner~$\Abf=(0,0)$.
Assuming that only Dirichlet boundary conditions are imposed, it is of our interest to consider the case when the singular function is
\begin{equation} \label{13.5}
\SA (r,\theta) = r^{\frac{2}{3}} \sin \left( \frac{2}{3} \theta  \right),
\end{equation}
where~$(r, \theta)$ are the polar coordinates at~$\Abf$.
\end{exmp}
\begin{exmp} 
Let~$\Omega$ be the unit square with an internal crack, see Figure~\ref{figure:singular_domains} (right),
\begin{equation} \label{crack-domain}
\Omega = (0,1)^2 \setminus \{  (x,y) \in \mathbb R^2 \mid x\in [1/4, \,3/4] ,\, y = x   \}.
\end{equation}
According to~\eqref{singular:functions:1}-\eqref{singular:functions:2}, the expected strongest singularities are located at the two tips of the internal crack.
Denote one of the two tips by~$\Abf$ and its polar coordinates by~$(r, \theta)$.
It is of our interest to consider the case when~$u$ is singular at~$\Abf$ only and the singular function is given by
\[
\SA (r,\theta) = r^{\frac{1}{2}} \sin \left( \frac{1}{2} \theta  \right).
\]
\end{exmp}

\begin{figure}[h]
\begin{center}
\begin{tikzpicture}[scale=2] 
\draw[black, very thick, -] (0,0) -- (1,0) -- (1,1) -- (-1,1) -- (-1,-1) -- (0,-1) -- (0,0) ;
\draw(.15,-.15) node[black] {{$\Abf$}};
\end{tikzpicture}
\quad \quad \quad \quad \quad \quad \quad \quad 
\begin{tikzpicture}[scale=3.99] 
\draw(.79,.79) node[black] {{$\Abf$}};
\draw[black, very thick, -] (0,0) -- (1,0) -- (1,1) -- (0,1) -- (0,0);
\draw[black, very thick, -] (0.25, 0.25) -- (0.75, 0.75);
\end{tikzpicture}
\end{center}
\caption{\emph{Left panel}: the L-shaped domain in~\eqref{L-shaped}. \emph{Right panel}: the unit square with an internal crack~\eqref{crack-domain}.
We denote the vertex/tip, where we assume the singularity takes place, by~$\Abf$.}
\label{figure:singular_domains}
\end{figure}
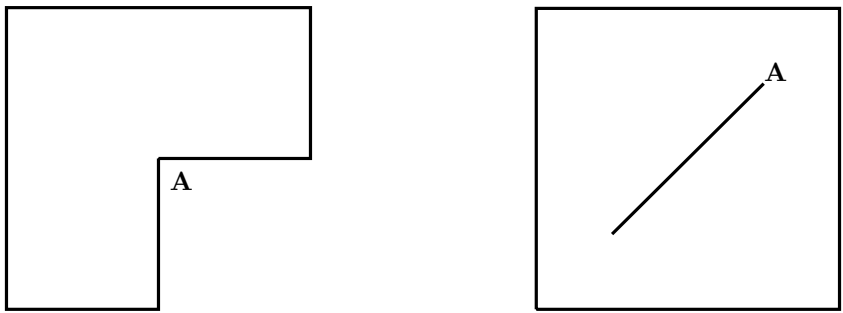

%%%%%%%%%%%%%%%%%%%%%%%%%%%%%%%%%%%%%%%%%%%%%%%%%%%%%%%%%%%%%%%%%%%%%%%%%%%%%%%%%%%%%%%%%%%%
\section{The enriched virtual element method} \label{section:VEM}
%%%%%%%%%%%%%%%%%%%%%%%%%%%%%%%%%%%%%%%%%%%%%%%%%%%%%%%%%%%%%%%%%%%%%%%%%%%%%%%%%%%%%%%%%%%%
We devote this section to the design of the novel enriched virtual element method.
To this aim, we introduce sequences of (regular) polygons in Section~\ref{subsection:decomposition}.
We define the enriched virtual element spaces in Section~\ref{subsection:EVES}, and the discrete bilinear forms and right-hand side in Section~\ref{subsection:discreteBFandRHS}.
Eventually, we exhibit the enriched VEM in Section~\ref{subsection:method}.

%%%%%%%%%%%%%%%%%%%%%%%
\subsection{Regular polygonal decomposition} \label{subsection:decomposition}
Here, we introduce sequences of (regular) polygons and some geometric assumptions.

Consider a sequence~$\{\taun\}_n$ of nonoverlapping polygons partitioning~$\Omega$.
For all~$n \in \mathbb N$ and~$\E\in \taun$, denote the size of~$\E$ by~$\hE$ and the mesh size function of~$\taun$, i.e., the maximum of such local diameters, by~$\h$.
For each~$n \in \mathbb N$, denote the set of vertices and edges of~$\taun$ by~$\Nun$ and~$\En$, and the set of internal and boundary edges by~$\EnI$ and~$\EnB$.
For all~$\E \in \taun$, let~$\EnE$ be the set of edges of element~$\E$ and~$\xbfE$ its barycenter.
Given~$\e \in \En$ an edge, denote its length by~$\he$.

Henceforth, we demand that the following assumptions on~$\taun$, for all~$n \in \mathbb N$, are valid: there exists a positive constant~$\gamma\in (0,1)$, such that
\begin{itemize}
\item[(\textbf{A0})] for all couples~$\E_1$ and~$\E_2 \in \taun$, $\gamma \h_{\E_1} \le \h_{\E_2} \le \gamma^{-1} \h_{\E_1}$;
\item[(\textbf{A1})] for all~$\E \in \taun$, $\E$ is star-shaped with respect to a ball of radius larger than or equal to~$\gamma\,\hE $;
\item[(\textbf{A2})] for all~$\e \in \EE$, the length~$\he$ of~$\e$ is larger than or equal to~$\gamma \,\hE$.
\end{itemize}
We employ assumptions (\textbf{A0})-(\textbf{A2}) in the analysis of the method; see Section~\ref{section:error_analysis} below.
Following, e.g., \cite{beiraolovadinarusso_stabilityVEM,brennerVEMsmall, CaoChen2018}, they could be weakened.
For the sake of simplicity, we stick here to standard geometric assumptions.

We allow for elements with a crack having (at most) one endpoint on the boundary and do not consider the case of completely cracked elements.

\paragraph*{Notation on normal unit vectors.}
Given~$\E \in \taun$, we denote its outward normal unit vector by~$\nE$.
Besides, with each~$\e\in \En$, we associate the normal unit vector~$\ne$ once and for all.
In general, $\ne$ does not coincide with~$\nE{}_{|\e}$ necessarily. However, $\ne \cdot \nE{}_{|\e} = \pm 1$.
If~$\e\in \EnB$, then we fix~$\ne:= \nOmega{}_{|\e}$.

%%%%%%%%%%%%%%%%%%%%%%%
\subsection{Nonconforming enriched virtual element spaces} \label{subsection:EVES}
The core idea behind the design of the enriched VEM (EVEM) is that the singular functions appearing in the expansions~\eqref{decomposition:solution} and~\eqref{assumption:solution} belong to the kernel of the Laplace operator.

Subdivide~$\taun$ into three layers.
The first layer $\tauno$ consists of the polygons abutting or containing the singular vertex~$\Abf$, as well as polygons that are sufficiently close to~$\Abf$: given~$\gammatilde >0$,
\begin{equation} \label{first-layer:definition}
\tauno := \left\{ \E \in \taun \text{ such that } \text{dist}(\Abf, \xbfE) \le \gammatilde  \diam(\Omega)  \right\}.
\end{equation}
Further comments on this definition will be provided in Remark~\ref{remark:first-layer} below.

The second layer~$\tauntw$ consists of the polygons sharing at least one edge with elements in the first layer, i.e.,
\[
\tauntw :=  \left\{ \E \in \taun \setminus \tauno \text{ such that there exists~$\widetilde \E \in \tauno$ with } \card(\EE \cap \mathcal E^{\widetilde \E} ) >0  \right\}.
\]
We set the third layer~$\taunth$ as the remainder of the elements in~$\taun$:
\begin{equation} \label{third:layer}
\taunth :=  \left\{ \E \in \taun \setminus (\tauno \cup \tauntw) \right\}.
\end{equation}
We consider an analogous splitting for the set of edges~$\En$. In particular, we subdivide~$\En$ into two layers of edges.
The first one, $\Eno$, is the set of all the edges belonging to the boundary of elements in~$\tauno$:
\[
\Eno = \left\{  \e \in \En \text{ such that there exists~$\E\in \tauno$ with } \e \subset \partial \E \right\},
\]
whereas second layer consists of the remainder of the edges, i.e.,
\[
\Entw = \left\{  \e \in \En \setminus \Eno \right\}.
\]
We exhibit a couple of graphical examples of such layers.

\begin{exmp}
Let~$\Omega$ be the L-shaped domain split into a Cartesian mesh of~$48$ elements; see Figure~\ref{figure:singular_domains} (left).
The re-entrant corner is the singular vertex.
We show the distributions of element and edge layers in Figure~\ref{figure:layersL-shaped} (left) and (right).
We pick~$\gammatilde = 1/10$ in~\eqref{first-layer:definition}.
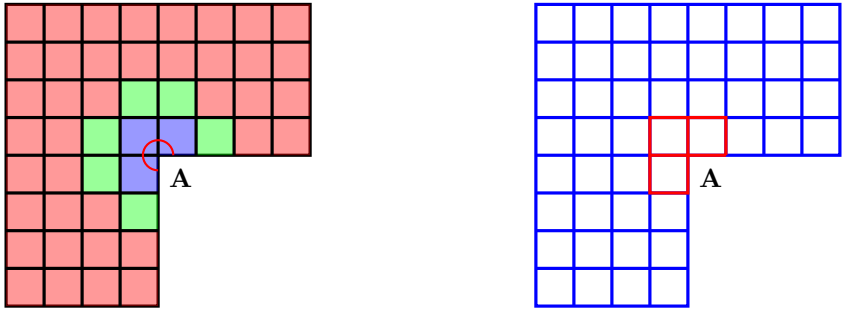
\begin{figure}[h]
\begin{center}
\begin{tikzpicture}[scale=2] 
\draw[black, very thick, -] (0,0) -- (1,0) -- (1,1) -- (-1,1) -- (-1,-1) -- (0,-1) -- (0,0) ;
\fill[blue, opacity=0.4] (0.25,0) -- (0.25, 0.25) -- (-0.25, 0.25) -- (-0.25,-0.25) -- (0,-0.25) -- (0,0) -- (0.25,0);
\fill[green, opacity=0.4] (0.25, 0) -- (0.5,0) -- (.5,.5) -- (-.5,.5) -- (-.5,-.5) -- (0,-.5) -- (0,-.25) -- (-.25,-.25) -- (-.25,.25) -- (.25,.25) -- (.25,0);
\fill[red, opacity=0.4] (0.5, 0) -- (1,0) -- (1,1) -- (-1,1) -- (-1,-1) -- (0,-1) -- (0,-.5) -- (-.5,-.5) -- (-.5,.5) -- (.5,.5) -- (.5,0);
\fill[white] (0.25, 0.25) -- (0.5, 0.25) -- (0.5, 0.5) -- (0.25, 0.5) -- (0.25, 0.25);
\fill[red, opacity=0.4] (0.25, 0.25) -- (0.5, 0.25) -- (0.5, 0.5) -- (0.25, 0.5) -- (0.25, 0.25);
\fill[white] (-0.25, 0.25) -- (-0.5, 0.25) -- (-0.5, 0.5) -- (-0.25, 0.5) -- (-0.25, 0.25);
\fill[red, opacity=0.4] (-0.25, 0.25) -- (-0.5, 0.25) -- (-0.5, 0.5) -- (-0.25, 0.5) -- (-0.25, 0.25);
\fill[white] (-0.25, -0.25) -- (-0.5, -0.25) -- (-0.5, -0.5) -- (-0.25, -0.5) -- (-0.25, -0.25);
\fill[red, opacity=0.4] (-0.25, -0.25) -- (-0.5, -0.25) -- (-0.5, -0.5) -- (-0.25, -0.5) -- (-0.25, -0.25);
\draw[black, very thick, -] (-.75,-1) -- (-.75,1); \draw[black, very thick, -] (-.5,-1) -- (-.5,1);  \draw[black, very thick, -] (-.25,-1) -- (-.25,1); 
\draw[black, very thick, -] (0,0) -- (0,1);  \draw[black, very thick, -] (0.25,0) -- (0.25,1);  \draw[black, very thick, -] (0.5,0) -- (0.5,1);  \draw[black, very thick, -] (0.75,0) -- (0.75,1); 
\draw[black, very thick, -] (-1,-.75) -- (0,-.75); \draw[black, very thick, -] (-1,-.5) -- (0,-.5);  \draw[black, very thick, -] (-1,-.25) -- (0,-.25);  \draw[black, very thick, -] (-1,-0) -- (0,-.0); 
\draw[black, very thick, -] (-1,.25) -- (1,.25); \draw[black, very thick, -] (-1,.5) -- (1,.5); \draw[black, very thick, -] (-1,.75) -- (1,.75);
\draw[red,thick,domain=0:270] plot ({1/10*cos(\x)}, {1/10*sin(\x)});
\draw(.15,-.15) node[black] {{$\Abf$}};
\end{tikzpicture}
\quad \quad \quad \quad \quad \quad \quad \quad 
\begin{tikzpicture}[scale=2] 
\draw[blue, very thick, -] (0,0) -- (1,0) -- (1,1) -- (-1,1) -- (-1,-1) -- (0,-1) -- (0,0) ;
\draw[blue, very thick, -] (-.75,-1) -- (-.75,1); \draw[blue, very thick, -] (-.5,-1) -- (-.5,1);  \draw[blue, very thick, -] (-.25,-1) -- (-.25,1); 
\draw[blue, very thick, -] (0,0) -- (0,1);  \draw[blue, very thick, -] (0.25,0) -- (0.25,1);  \draw[blue, very thick, -] (0.5,0) -- (0.5,1);  \draw[blue, very thick, -] (0.75,0) -- (0.75,1); 
\draw[blue, very thick, -] (-1,-.75) -- (0,-.75); \draw[blue, very thick, -] (-1,-.5) -- (0,-.5);  \draw[blue, very thick, -] (-1,-.25) -- (0,-.25);  \draw[blue, very thick, -] (-1,-0) -- (0,-.0); 
\draw[blue, very thick, -] (-1,.25) -- (1,.25); \draw[blue, very thick, -] (-1,.5) -- (1,.5); \draw[blue, very thick, -] (-1,.75) -- (1,.75);
\draw[red, very thick, -] (0.25,0) -- (0.25, 0.25) -- (-0.25, 0.25) -- (-0.25,-0.25) -- (0,-0.25) -- (0,0) -- (0.25,0);
\draw[red, very thick, -] (0,0) -- (0,.25); \draw[red, very thick, -] (0,0) -- (-.25, 0);
\draw(.15,-.15) node[black] {{$\Abf$}};
\end{tikzpicture}
\end{center}
\caption{\emph{Left panel}: in blue, the elements in the first element layer~$\tauno$; in green, the elements in the second element layer~$\tauntw$; in red, the elements in the third layer~$\taunth$.
\emph{Right panel}: in red, the edges in the first edge layer~$\Eno$; in blue, the edges in the second edge layer~$\Entw$.
The domain is the L-shaped domain defined in~\eqref{L-shaped}. We assume that the solution to problem~\eqref{weak:form:simple} is singular only at  the re-entrant corner~$\Abf$.
In red, we depict the circumference of radius~$1/10$ centred at~$\Abf$. The parameter~$\gammatilde$ in~\eqref{first-layer:definition} is~$1/10$.}
\label{figure:layersL-shaped}
\end{figure}
\end{exmp}

\begin{exmp}
Let~$\Omega$ be the unit square domain with an internal crack split into a Cartesian mesh of~$9$ elements; see Figure~\ref{figure:singular_domains} (right).
The top-right tip is the singular vertex.
We show the distributions of element and edge layers, in Figure~\ref{figure:layersCrack} (left) and (right).
We pick~$\gammatilde = 1/10$ in~\eqref{first-layer:definition}.
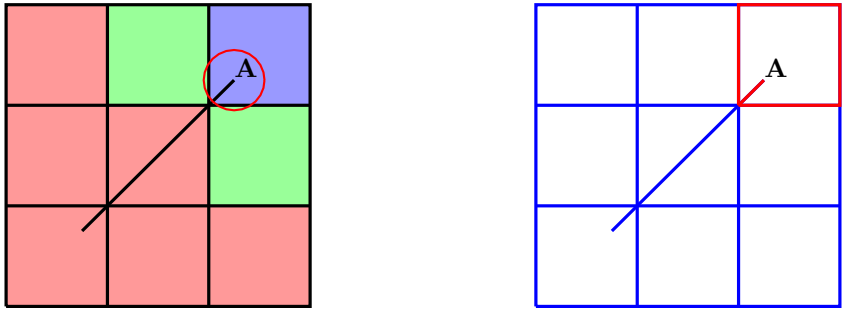
\begin{figure}[h]
\begin{center}
\begin{tikzpicture}[scale=4] 
\fill[blue, opacity=0.4] (2/3,2/3) -- (1,2/3) -- (1,1) -- (2/3,1) -- (2/3,2/3);
\fill[green, opacity=0.4] (1,2/3) -- (2/3,2/3) -- (2/3,1) -- (1/3,1) -- (1/3,1/3) -- (1,1/3) -- (1,2/3);
\fill[red, opacity=0.4] (1,1/3) -- (1/3,1/3) -- (1/3,1) -- (0,1) -- (0,0) -- (1,0) -- (1,1/3);
\fill[white] (1/3,1/3) -- (2/3,1/3) -- (2/3,2/3) -- (1/3,2/3) -- (1/3,1/3);
\fill[red, opacity=0.4] (1/3,1/3) -- (2/3,1/3) -- (2/3,2/3) -- (1/3,2/3) -- (1/3,1/3);
\draw[black, very thick, -] (0,0) -- (1,0) -- (1,1) -- (0,1) -- (0,0) ;
\draw[black, very thick, -] (.25,.25) -- (.75, .75);
\draw[black, very thick, -] (0,1/3) -- (1,1/3); \draw[black, very thick, -] (0,2/3) -- (1,2/3); 
\draw[black, very thick, -] (1/3,0) -- (1/3,1); \draw[black, very thick, -] (2/3,0) -- (2/3,1); 
\draw[red,thick,domain=0:360] plot ({0.75+1/10*cos(\x)}, {0.75+1/10*sin(\x)});
\draw(.79,.79) node[black] {{$\Abf$}};
\end{tikzpicture}
\quad \quad \quad \quad \quad \quad \quad \quad 
\begin{tikzpicture}[scale=4] 
\draw[blue, very thick, -] (0,0) -- (1,0) -- (1,1) -- (0,1) -- (0,0) ;
\draw[blue, very thick, -] (.25,.25) -- (.75, .75);
\draw[blue, very thick, -] (0,1/3) -- (1,1/3); \draw[blue, very thick, -] (0,2/3) -- (1,2/3); 
\draw[blue, very thick, -] (1/3,0) -- (1/3,1); \draw[blue, very thick, -] (2/3,0) -- (2/3,1); 
\draw[red, very thick, -] (2/3,2/3) -- (1,2/3) -- (1,1) -- (2/3,1) -- (2/3,2/3);
\draw[red, very thick, -] (2/3,2/3) -- (0.75,.75);
\draw(.79,.79) node[black] {{$\Abf$}};
\end{tikzpicture}
\end{center}
\caption{\emph{Left panel}: in blue, the elements in the first element layer~$\tauno$; in green, the elements in the second element layer~$\tauntw$; in red, the elements in the third layer~$\taunth$.
\emph{Right panel}: in red, the edges in the first edge layer~$\Eno$; in blue, the edges in the second edge layer~$\Entw$.
The domain is the unit square domain with an internal crack defined in~\eqref{crack-domain}.
We assume that the solution to problem~\eqref{weak:form:simple} is singular only at  the right-upper tip of the crack~$\Abf$.
In red, we depict the circumference of radius~$1/10$ centred at~$\Abf$. The parameter~$\gammatilde$ in~\eqref{first-layer:definition} is~$1/10$.}
\label{figure:layersCrack}
\end{figure}
The top-right corner contains a crack, which has a tip on the boundary.
\end{exmp}

Henceforth, we fix a~$\p \in\mathbb N$, which will denote the standard polynomial order of accuracy of the method, and introduce auxiliary functions and spaces.
Given~$\E \in \taun$, define the bulk-scaled enriching function
\begin{equation} \label{SAE}
\SAE(r, \theta) := \SA \left( \frac{r}{\hE}, \theta \right).
\end{equation}
For example, if~$\SA$ is given as in~\eqref{13.5}, then $\SAE(r,\theta) = \left(  \frac{r}{\hE}  \right)^{\frac{2}{3}} \sin \left( \frac{2}{3} \theta   \right)$.

Given~$\e \in \EE$, define the edge-scaled enriching function
\begin{equation} \label{SAe}
\SAe(r,\theta) :=  \SA \left( \frac{r}{\he} ,\theta \right).
\end{equation}
For example, if~$\SA$ is given as in~\eqref{13.5}, then $\SAe(r,\theta) = \left(  \frac{r}{\he}  \right)^{\frac{2}{3}} \sin \left( \frac{2}{3} \theta   \right)$.

Observe that
\begin{equation} \label{relation:special}
\SAE(r,\theta) = \left(  \frac{\he}{\hE}  \right)^{\alpha} \SAe(r, \theta),
\end{equation}
where~$\alpha >0$ depends on the definition of~$\SA$. For example,  if~$\SA$ is given as in~\eqref{13.5}, then $\alpha= 2/3$.

We define the set of enriched polynomials over an element~$\E \in \taun$ as follows:
\[
\Pbbtilde_\p(\E) = \begin{cases} \mathbb P_\p(\E)  \oplus \SAE & \text{if } \E \in \tauno\\ \mathbb P_\p(\E) & \text{otherwise}. \\ \end{cases}
\]
In other words, on the elements close to the singular vertex/tip $\Abf$, we add the singular function~$\SA$ to the nonenriched polynomial space~$\mathbb P_\p(\E)$.

Further, we define the set of enriched polynomials over edges:
\begin{equation} \label{edge:enriched:polynomials}
\Pbbtilde _{\p-1}(\e) = \begin{cases}
\mathbb P_{\p-1}(\e)  \oplus \left( \ne \cdot \nabla \SAe{}_{|\e} \right) & \text{if } \e \in \Eno\\
\mathbb P_{\p-1}(\e) & \text{otherwise}. \\
\end{cases}
\end{equation}
In other words, we consider nonenriched one dimensional polynomial spaces~$\mathbb P_{\p-1}(\e)$ on all edges except those belonging to the boundary of the elements in the first layer~$\tauno$.
Here, we consider the normal derivative of the scaled enriching function~$\SAe$ as additional special function.
\medskip

Next, we define the local enriched virtual element spaces: for all~$\E \in \taun$,
\begin{equation} \label{local:new}
\VnE := \cbrack{ \vn \in H^1(\E) \mid \Delta \vn \in \mathbb P_{\p-2}(\E),\; \ne \cdot \nabla \vn{}_{|\e} \in \Pbbtilde_{\p-1}(\e)  \quad \forall \e \in \EE}.
\end{equation}
Functions in~$\VnE$ are unknown in closed form both in the bulk and the boundary of element~$\E$.
This is the reason why the functions are referred to as virtual.
The space~$\VnE$ contains the space of polynomials of degree~$\p$. Furthermore, if~$\E \in \tauno$, then the singular function~$\SA$ belongs to~$\VnE$ as well.
This is the reason why we regard the space~$\VnE$ as enriched.

We introduce additional notation.
Let~$\cbrack{\malpha}_{\vert \boldalpha \vert = 0}^{\dim(\mathbb P_{\p-2}(\E))}$ be a basis of~$\mathbb P_{\p-2}(\E)$.
For instance, this basis consists of the monomials introduced in~\cite{VEMvolley} or some orthonormal basis as in~\cite{fetishVEM}.
We assume that the elements~$\malpha$ are centred in the barycenter of the element and scaled according to the element diameter.
It is known, that the~$\malpha$ basis can be~$L^2$-orthonormalized for stability purposes; see, e.g., \cite{fetishVEM}.
For the sake of exposition, we stick here to the monomial basis.
Besides, let~$\cbrack{\mtildealphae}_{\alpha=0}^{\p-1}(\e)$ be a basis of~$\Pbbtilde_{\p-1}(\e)$ defined in~\eqref{edge:enriched:polynomials} for all edges~$\e \in \EE$.
A possible choice of the basis is provided by the first~$\p-1$ Legendre polynomials on the local system of coordinates over~$\e$, and the global normal derivative of the singular function~$\SAe$ over edge~$\e$.
Such a basis can be orthonormalized, leading to a dramatic improvement of the performance of the method; see Sections~\ref{subsection:NR-stab} and Appendix~\ref{appendix:implementation2} below.
Here, for the sake of presentation, we stick to the former choice.
\medskip

Consider the following set of linear functionals on~$\VnE$: for all~$\vn \in \VnE$,
\begin{itemize}
\item the internal moments:
\begin{equation} \label{internal:moments}
\frac{1}{\vert \E \vert} \int_\E \vn \malpha \quad \forall \alpha = 1, \dots, \dim(\mathbb P_{\p-2}(\E));
\end{equation}
\item if~$\e \in \Eno$, the edge moments:
\begin{equation} \label{edge:moments}
\begin{cases}
\frac{1}{\he} \int_\e \vn \mtildealphae \quad \forall \alpha=0,\dots,\p-1,\, \forall \e \in \EE,  \\
\int_\e \vn (\ne \cdot \nabla \SAe){}_{|\e} \quad \alpha=\p\,,\, \forall \e \in \EE. \\
\end{cases}
\end{equation}
If~$\e \in \Entw$, then the edge moments are the same, but there is no special moment for~$\alpha=\p$.
\end{itemize}

\begin{remark} \label{remark:well-posedness-dofs}
The enriched edge functionals in~\eqref{edge:moments} are well posed for all possible singular functions~$\SA$.
In order to see this, we first observe that~$(\ne \cdot \nabla \SAe){}_{|\e} \in L^1(\e)$ for all~$\e \in \En$.
This follows based on the explicit representation of the singular function~$\SAe$.

On the other hand, given~$\E\in \taun$, any function~$\vn$ in the local virtual element space~$\VnE$ solves a local elliptic problem.
In particular, $\vn \in H^{1+ \varepsilon}(\E)$, $\varepsilon >0$.
Therefore, the Sobolev embedding theorem in two dimensions yields $\vn \in \mathcal C^{0}(\overline \E)$ and~$\vn{}_{|\e} \in \mathcal C^{0}(\e)$ for all~$\e \in \EE$.
The well posedness of~\eqref{edge:moments} follows.

This fact entails that all the forthcoming integrations by parts involving functions~$\vn \in \VnE$ are well defined; see, e.g., \eqref{IBPs} and~\eqref{IBPs2}.
\eremk
\end{remark}

\begin{lem} \label{lemma:dofs}
For all~$\E \in \taun$, the set of linear functionals in~\eqref{internal:moments}-\eqref{edge:moments} is a unisolvent set of degrees of freedom for the space~$\VnE$.
\end{lem}
\begin{proof}
The dimension of~$\VnE$ is equal to the number of linear functionals; see~\cite{nonconformingVEMbasic}.
Thence, it suffices to prove the unisolvence of such functionals.
Observe that
\begin{equation} \label{IBPs}
\vert \vn \vert^2_{1,\E} = -\int_\E \underbrace{ \Delta \vn}_{\in \mathbb P_{\p-2}(\E)} \, \vn + \sum_{\e \in \EE} \int_\e \underbrace{\nE \cdot \nabla \vn}_{\in \Pbbtilde_{\p-1}(\e)} \, \vn.
\end{equation}
The first and second terms are zero, for the internal~\eqref{internal:moments} and the edge moments~\eqref{edge:moments} are zero by assumption.
Hence, $\vn$ is constant. This and the fact that the average over~$\partial \E$ of~$\vn$ is equal to zero entail the assertion.
\end{proof}

Compared to the degrees of freedom (DOFs) in the nonenriched nonconforming VEM~\cite{nonconformingVEMbasic}, we consider the same internal DOFs~\eqref{internal:moments}.
As for the edge DOFs~\eqref{edge:moments}, we cope with additional moments related to the special functions on the edges of the elements in the first edge layer~$\Eno$.
In Figures~\ref{figure:dofs-p1} and~\ref{figure:dofs-p2}, we depict the DOFs for the nonenriched and enriched nonconforming virtual element method, with order~$\p=1$ and~$2$ on a pentagon.
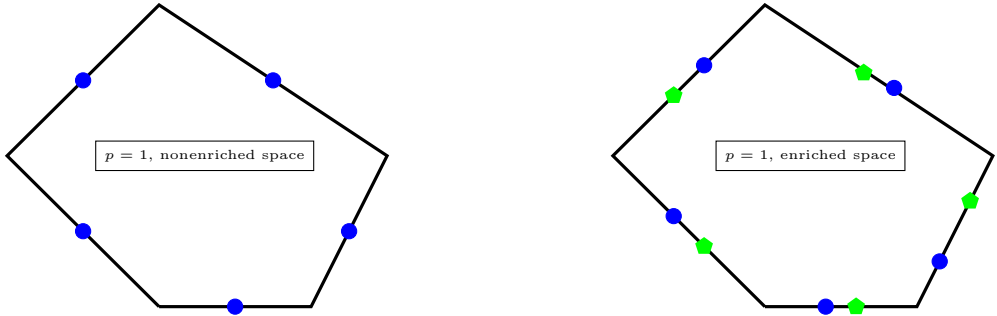
\begin{figure}[h]
\begin{center}
\begin{tikzpicture}[scale=2] 
\draw[black, very thick, -] (0,0) -- (1,0) -- (1.5,1) -- (0,2) -- (-1,1) -- (0,0);
\draw[blue, fill] (0.5,0) circle (0.5mm); \draw[blue, fill] (1.25,0.5) circle (0.5mm); \draw[blue, fill] (0.75,1.5) circle (0.5mm); \draw[blue, fill] (-0.5,1.5) circle (0.5mm); \draw[blue, fill] (-0.5,0.5) circle (0.5mm); 
\node[draw, black] at (0.3,1) {\tiny{$\p=1$, nonenriched space}};
\end{tikzpicture}
\quad\quad\quad\quad\quad\quad\quad\quad
\begin{tikzpicture}[scale=2] 
\draw[black, very thick, -] (0,0) -- (1,0) -- (1.5,1) -- (0,2) -- (-1,1) -- (0,0);
\draw[blue, fill] (0.4,0) circle (0.5mm); \draw[blue, fill] (1.15,0.3) circle (0.5mm); \draw[blue, fill] (0.85,1.45) circle (0.5mm); \draw[blue, fill] (-0.4,1.6) circle (0.5mm); \draw[blue, fill] (-0.6,0.6) circle (0.5mm); 
\node[fill=green,regular polygon, regular polygon sides=5,inner sep=2pt] at (.6,0) {};
\node[fill=green,regular polygon, regular polygon sides=5,inner sep=2pt] at (1.35,0.7) {};
\node[fill=green,regular polygon, regular polygon sides=5,inner sep=2pt] at (0.65,1.55) {};
\node[fill=green,regular polygon, regular polygon sides=5,inner sep=2pt] at (-0.6,1.4) {};
\node[fill=green,regular polygon, regular polygon sides=5,inner sep=2pt] at (-0.4,0.4) {};
\node[draw, black] at (0.3,1) {\tiny{$\p=1$, enriched space}};
\end{tikzpicture}
\end{center}
\caption{Degrees of freedom on a pentagon for~$\p=1$.
\emph{Left panel}: nonenriched nonconforming VEM. \emph{Right panel}: enriched nonconforming VEM.
The blue circles represent polynomial moments on the edges. The green pentagons represent the enriched edge moments.}
\label{figure:dofs-p1}
\end{figure}
\begin{figure}[h]
\begin{center}
\begin{tikzpicture}[scale=2] 
\draw[black, very thick, -] (0,0) -- (1,0) -- (1.5,1) -- (0,2) -- (-1,1) -- (0,0);
\draw[blue, fill] (0.4,0) circle (0.5mm); \draw[blue, fill] (1.15,0.3) circle (0.5mm); \draw[blue, fill] (0.85,1.45) circle (0.5mm); \draw[blue, fill] (-0.4,1.6) circle (0.5mm); \draw[blue, fill] (-0.6,0.6) circle (0.5mm); 
\draw[blue, fill] (0.6,0) circle (0.5mm); \draw[blue, fill] (1.35,0.7) circle (0.5mm); \draw[blue, fill] (0.65,1.55) circle (0.5mm); \draw[blue, fill] (-0.6,1.4) circle (0.5mm); \draw[blue, fill] (-0.4,0.4) circle (0.5mm); 
\node[fill=magenta,regular polygon, regular polygon sides=3,inner sep=1.5pt] at (.4,0.8) {};
\node[draw, black] at (0.3,1.2) {\tiny{$\p=2$, nonenriched space}};
\end{tikzpicture}
\quad\quad\quad\quad\quad\quad\quad\quad
\begin{tikzpicture}[scale=2] 
\draw[black, very thick, -] (0,0) -- (1,0) -- (1.5,1) -- (0,2) -- (-1,1) -- (0,0);
\draw[blue, fill] (0.3,0) circle (0.5mm); \draw[blue, fill] (1.05,0.1) circle (0.5mm); \draw[blue, fill] (0.95,1.36) circle (0.5mm); \draw[blue, fill] (-0.3,1.7) circle (0.5mm); \draw[blue, fill] (-0.7,0.7) circle (0.5mm); 
\draw[blue, fill] (0.5,0) circle (0.5mm); \draw[blue, fill] (1.25,0.5) circle (0.5mm); \draw[blue, fill] (0.75,1.5) circle (0.5mm); \draw[blue, fill] (-0.5,1.5) circle (0.5mm); \draw[blue, fill] (-0.5,0.5) circle (0.5mm); 
\node[fill=green,regular polygon, regular polygon sides=5,inner sep=2pt] at (.7,0) {};
\node[fill=green,regular polygon, regular polygon sides=5,inner sep=2pt] at (1.45,0.9) {};
\node[fill=green,regular polygon, regular polygon sides=5,inner sep=2pt] at (0.55,1.625) {};
\node[fill=green,regular polygon, regular polygon sides=5,inner sep=2pt] at (-0.7,1.3) {};
\node[fill=green,regular polygon, regular polygon sides=5,inner sep=2pt] at (-0.3,0.3) {};
\node[fill=magenta,regular polygon, regular polygon sides=3,inner sep=1.5pt] at (.4,0.8) {};
\node[draw, black] at (0.3,1.2) {\tiny{$\p=2$, enriched space}};
\end{tikzpicture}
\end{center}
\caption{Degrees of freedom on a pentagon for~$\p=2$.
\emph{Left panel}: nonenriched nonconforming VEM. \emph{Right panel}: enriched nonconforming VEM.
The blue circles represent the polynomial moments on the edges. The green pentagons represent the enriched edge moments. The magenta triangle represents the only polynomial bulk moment.}
\label{figure:dofs-p2}
\end{figure}
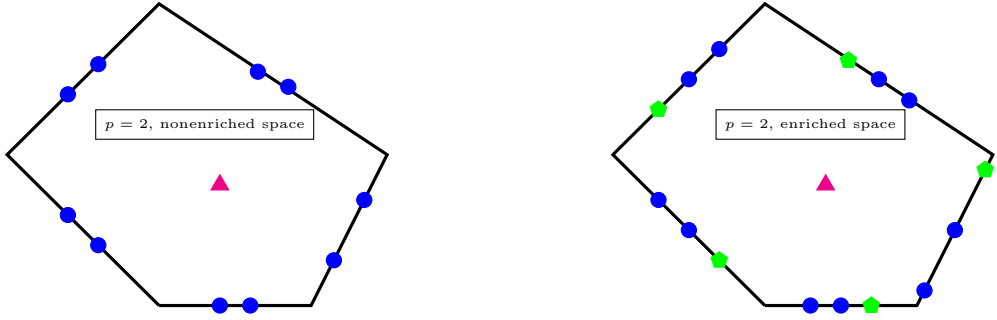

For future convenience, introduce the local canonical basis~$\{\varphi_i\}_{i=1}^{\dim(\VnE)}$ defined as
\begin{equation} \label{canonical-basis}
\dof_j(\varphi_i) = \delta_{i,j},
\end{equation}
where~$\delta_{i,j}$ denotes the Kronecker delta.

Next, split the bilinear form~$\a(\cdot, \cdot)$ defined in~\eqref{basic:notation} into local contributions:
\[
\a(u,v) = \sum_{\E\in \taun} \aE(u,v) := \sum_{\E \in \taun} \int_\E \nabla u_{|\E} \cdot \nabla v_{|\E} \quad \quad \forall u,\, v \in H^1(\Omega).
\]
The definition of the degrees of freedom allows for the computation the enriched $H^1$-orthogonal projection~$\Pinabla:\VnE \rightarrow \Pbbtilde_\p(\E)$:
\begin{equation} \label{Pinabla:proj}
\begin{cases}
\aE(\vn - \Pinabla \vn, \qp) = 0 \\
\int_{\partial \E} (\vn - \Pinabla \vn) = 0\\
\end{cases} \quad \quad \forall \vn \in \VnE, \quad \forall \qp \in \Pbbtilde_\p(\E) .
\end{equation}
In fact, an integration by parts yields
\begin{equation} \label{computability:Pinabla}
\aE(\vn,\qp) = (\vn, \underbrace{-\Delta \qp}_{\in \mathbb P_{\p-2}(\E)})_{0,\E} + \sum_{\e \in \EE} (\vn, \underbrace{\nE \cdot \nabla \qp}_{\in \Pbbtilde_{\p-1}(\e)})_{0,\e}.
\end{equation}
The first and second term on the right-hand side of~\eqref{computability:Pinabla} are computable from~\eqref{internal:moments} and~\eqref{edge:moments}.
The second term can be approximated at any precision by a one dimensional quadrature formula; see Remark~\ref{remark:Jacobi} below for more comments on this point.
For all the elements~$\E \in \tauno$, the projection~$\Pinabla$ maps functions belonging to the virtual element space into the space of bulk-enriched polynomial space~$\widetilde{\mathbb P}_{\p}(\E)$.

Further, for all~$\e \in \En$, we consider the possibly enriched~$L^2$ edge projector $\Pize : \VnE{}_{|\e} \rightarrow \Pbbtilde_{\p-1}(\e)$, defined as
\begin{equation} \label{edge:projector}
\int_\e (\vn - \Pize \vn) \mtildealphae = 0 \quad\quad \forall \vn\in \VnE,\quad \forall \mtildealphae \in \Pbbtilde_{\p-1}(\e).
\end{equation}
The computability of such a projector follows from the definition of the edge degrees of freedom~\eqref{edge:moments}.

\begin{remark} \label{remark:computability-edge-proj}
Let~$\e \in \En$ and~$\E$ be such that~$\e \in \EE$. The projector~$\Pize$ in~\eqref{edge:projector} can be computed if~$\SA \in H^{\frac{3}{2}+\varepsilon}(\E)$ with~$\varepsilon>0$.
In fact, we need~$\nE \cdot \nabla \SA{}_{|\e} \in L^2(\e)$.
This appears as a partial limitation to design and applicability of the method.
However, we need to compute the projector~$\Pize$ in the following circumstances only:
in the design of a theoretical stabilization for the method; see Section~\ref{subsection:stab} below;
in the discretization of certain nonhomogeneous Neumann boundary conditions; see Remark~\ref{remark:Neumann} and Appendix~\ref{subsection:nh-Neumann} below;
in the ``orthonormalized'' version of the method; see Appendix~\ref{appendix:implementation2} below.
\end{remark}

We also introduce the vector nonenriched polynomial projections $\Piboldnabla : [H^1(\E)]^2 \rightarrow [\mathbb P_{\p}(\E)]^2$
and~$\Piboldze: [L^2(\e)]^2 \rightarrow [\mathbb P_{\p-1}(\e)]^2$ similarly as in~\eqref{Pinabla:proj} and~\eqref{edge:projector}.

Eventually, we define the nonenriched $L^2$-bulk orthogonal projector~$\PizE : \VnE \rightarrow \mathbb P_{\p-2}(\E)$ as
\begin{equation} \label{L2projection}
(\vn - \PizE \vn, \qpmt)_{0,\E} = 0 \quad \quad \forall \qpmt \in \mathbb P_{\p-2}(\E).
\end{equation}
The projector~$\PizE$ is computable from the internal degrees of freedom~\eqref{internal:moments}
and is used for the approximation of the Neumann boundary conditions only; see Remark~\ref{remark:Neumann} below.
\medskip

Next, we define the global nonconforming virtual element space~$\Vn$.
Given an internal edge~$\e \in \EnI$, denote its two adjacent elements by~$\E^+$ and~$\E^-$. Instead, given a boundary edge~$\e \in \EnB$, denote its adjacent element by~$\E$.
Moreover, denote the space of~$L^2(\Omega)$ functions piecewise in~$H^1$ over~$\taun$ by~$H^1(\taun)$,
and define the broken Sobolev norm
\[
\vert v \vert^2_{1,\taun} := \sum_{\E \in \taun} \vert v_{\E} \vert^2_{1,\E} \quad \quad \forall v \in H^1(\taun).
\]
Introduce the jump operator across an edge~$\e \in \Eno$: given~$v \in H^1(\taun)$, set
\begin{equation} \label{jump}
\llbracket v \rrbracket _\e= \llbracket v \rrbracket :=
\begin{cases}
v_{|\E^+} \n_{\E^+} + v_{|\E^-} \n_{\E^-} 		& \text{if } \e \in \EnI\\
v \nE			 						& \text{if } \e \in \EnB.\\
\end{cases}
\end{equation}
Introduce the global nonconforming Sobolev space of order~$\p$, subordinated to the mesh~$\taun$, including homogeneous boundary conditions in a nonconforming sense:
\[
\begin{split}
H^{1,nc}_0(\taun,\p) := 	& \left\{ v\in H^1(\taun)  \middle\vert \int _\e \llbracket v \rrbracket \cdot \ne \, \mtildealphae =0 \quad \forall \mtildealphae \in \Pbbtilde_{\p-1}(\e),\; \forall \e \in \En  \right\}.\\
\end{split}
\]
We define the global test and trial nonconforming enriched virtual element spaces as
\begin{equation} \label{localEVEspace}
 \Vn := \left\{\vn \in H^{1,nc}_{0}(\taun,\p)   \mid \vn{}_{|\E} \in \VnE \;\, \forall \E \in \taun   \right\}. \\
\end{equation}
We construct the space~$\Vn$ by a nonconforming coupling of the local edge degrees of freedom~\eqref{edge:moments}.
The global canonical basis is defined  from its local counterparts~\eqref{canonical-basis} accordingly.

%%%%%%%%%%%%%%%%%
\subsection{The discrete bilinear form and right-hand side} \label{subsection:discreteBFandRHS}
Functions in the virtual element spaces are not available in closed form.
Thus, in order to design the numerical scheme for the approximation of solutions to~\eqref{weak:form:simple},
we introduce a global bilinear form and right-hand side that are computable in terms of the degrees of freedom.
To this purpose, we generalize the construction in~\cite{nonconformingVEMbasic} to the enriched setting.

\paragraph*{The discrete bilinear form.}
Using the orthogonality property of the projector~$\Pinabla$ in~\eqref{Pinabla:proj}, we apply Pythagoras' theorem in Hilbert spaces and get
\[
\aE(u, v) = \aE(\Pinabla u, \Pinabla v) + \aE( (I-\Pinabla) u, (I-\Pinabla) v ) \quad \forall u,\, v \in H^1(\E).
\]
The first term on the right-hand side is computable on~$\VnE \times \VnE$, see~\eqref{computability:Pinabla}, whereas the second one is not.
As standard in virtual elements~\cite{VEMvolley,hitchhikersguideVEM},
we introduce a symmetric bilinear form~$\SE : \ker(\Pinabla) \times \ker(\Pinabla) \rightarrow \mathbb R$ satisfying
\begin{equation} \label{property:stabilization}
c_*(\E) \vert  \vn \vert^2_{1,\E} \le \SE(\vn, \vn) \le c^*(\E) \vert \vn \vert^2 _{1,\E} \quad \quad \forall \vn \in \ker(\Pinabla),\quad \forall \E\in \taun,
\end{equation}
where $c_*(\E)$ and~$c^*(\E)$ are two positive constants, possibly depending on the polynomial degree~$\p$, the geometric properties of~$\E$, and the singular function~$\SA$.

Having this at hand, define the local discrete bilinear forms
\[
\anE (\un, \vn) = \aE(\Pinabla \un, \Pinabla \vn) + \SE( (I-\Pinabla) \un , (I-\Pinabla) \vn ) \quad \quad \forall \un,\, \vn \in \VnE,
\]
and the global discrete bilinear form
\[
\an(\un, \vn) = \sum_{\E \in \taun} \anE (\un{}_{|\E}, \vn{}_{|\E})  \quad \quad \forall \un,\vn \in \Vn.
\]
We postpone explicit choices and further considerations about the stabilization forms~$\SE$ to Section~\ref{subsection:stab} below.

For all~$\E \in \taun$, the local discrete bilinear forms~$\anE(\cdot,\cdot)$ are coercive and continuous with respect to the $H^1$ seminorm,
with coercivity and continuity constants
\[
\alpha_*(\E) = \min(1,c_*(\E)), \quad \quad \alpha^*(\E) = \max(1, c^*(\E)).
\]
In other words, for all~$\E \in \taun$, we have
\begin{equation} \label{coercivity_continuity}
\alpha_*(\E) \vert \vn \vert^2_{1,\E} \le \anE(\vn, \vn) \le \alpha^*(\E) \vert \vn \vert^2_{1,\E} \quad\quad \forall \vn \in \VnE.
\end{equation}
Moreover, the stabilization is symmetric and polynomially-enriched consistent: for all~$\E \in \taun$,
\begin{equation} \label{consistency}
\anE(\qp, \vn) = \aE(\qp, \vn) \quad \quad \forall \qp \in \Pbbtilde _\p(\E) , \quad  \forall \vn \in \Vn.
\end{equation}

\paragraph*{The discrete right-hand side.}
Denote the number and the set of vertices of element~$\E$ by~$N_V^\E$ and~$\{\nu_i\}_{i=1}^{N_V^\E}$, recall the definition of the projector~$\PizE$ in~\eqref{L2projection},
and define
\[
\langle \f, \vn \rangle_n = \sum_{\E \in \taun} \langle \f_{|\E}, \vn{}_{|\E} \rangle_{n, \E} := 
\begin{cases}
\sum_{\E \in \taun} \int_\E \f_{|\E} \PizE \vn{}_{|\E} & \text{if } \p \ge 2\\
\sum_{\E \in \taun} \frac{1}{N_V^\E} \int_\E \f_{|\E} (\sum_{i=1}^{N_V} \vn(\nu_i)  ) & \text{if } \p=1.\\
\end{cases}
\]

%%%%%%%%%%%%%%%%%%%%%%%%%%%%%%
\subsection{The method} \label{subsection:method}
The nonconforming enriched virtual element method for problem~\eqref{weak:form:simple} reads
\begin{equation} \label{ncEVEM}
\begin{cases}
\text{find } \un \in \Vn \text{ such that}\\
\an(\un,\vn) =  \langle \f, \vn \rangle_n \quad\quad \forall \vn \in \Vn.
\end{cases}
\end{equation}
Method~\eqref{ncEVEM} is well posed thanks to the continuity and coercivity properties detailed in Section~\ref{subsection:discreteBFandRHS}.
We devote Section~\ref{section:error_analysis} below to the analysis of~\eqref{ncEVEM},
whereas we provide the implementation details in Appendix~\ref{appendix:implementation}.

\begin{remark} \label{remark:Neumann}
Nonhomogeneous Neumann boundary conditions~$\gN$ in~\eqref{weak:form} are approximated as follows:
\begin{equation} \label{Neumann:approximation}
\sum_{\e \in \EnB,\, \e \subset \GammaN} \int_{\e} \gN \Pize \vn.
\end{equation}
As highlighted in Remark~\ref{remark:computability-edge-proj}, for all~$\e\in \En$, the projector~$\Pize$ can be computed if~$\SA \in H^{\frac{3}{2} + \varepsilon}(\E)$ with~$\varepsilon>0$, where~$\E \in \taun$ is such that~$\e \in \EE$.
Further comments on nonhomogenous Neumann boundary conditions and a generalization of~\eqref{Neumann:approximation} for some special classes of~$\gN$ are provided in Appendix~\ref{subsection:nh-Neumann} below.

Nonhomogeneous Dirichlet boundary conditions are enforced through the degrees of freedom on the Dirichlet edges in the trial and test spaces.
\end{remark}

\paragraph*{The role of the Heaviside function.} Our approach does not allow for the inclusion of the Heaviside in the test and trial spaces as done in the XFEM, GFEM, and XVEM.
In fact, the Laplacian of the Heaviside function is not a function, rendering the construction of~$\VnE$ in~\eqref{local:new} not feasible.
However, the flexibility of polygonal meshes renders the use of the Heaviside function practically useless.
In fact, elements cut into two by a crack can be re-meshed into distinct polygons,
whereas the setting in Section~\ref{subsection:EVES} allows for handling automatically elements  with internal cracks;
see also~\cite{artioli2020vem}.

%%%%%%%%%%%%%%%%%%%%%%%%%%%%%%%%%%%%%%%%%%%%%%%%%%%%%%%%%%%%%%%%%%%%%%%%%%%
\section{Error analysis} \label{section:error_analysis}
%%%%%%%%%%%%%%%%%%%%%%%%%%%%%%%%%%%%%%%%%%%%%%%%%%%%%%%%%%%%%%%%%%%%%%%%%%%
In this section, we analyse the rate of convergence of method~\eqref{ncEVEM}.
We present an abstract error result in Section~\ref{subsection:abstract}.
In Sections~\ref{subsection:best:polynomial} and~\ref{subsection:best:VEM}, we describe the approximation properties of enriched polynomial and enriched virtual element spaces.
On the other hand, Sections~\ref{subsection:approx:RHS} and~\ref{subsection:crime_nc} deal with the approximation of the variational crimes perpetrated in the approximation of the right-hand side and the nonconformity of the method.
After introducing and analysing possible stabilizations in Section~\ref{subsection:stab}, we collect all the above estimates in Section~\ref{subsection:convergence_method}.
We discuss some extensions and generalization in Section~\ref{subsection:generalizations}.

%%%%%%%%%%%%%%%
\subsection{Abstract error analysis} \label{subsection:abstract}
Here, we present the abstract error analysis result for method~\eqref{ncEVEM}.
Recall that jump operator~$\llbracket \cdot \rrbracket_\e$ across edge~$\e$ is defined in~\eqref{jump},
and define the bilinear form $\Ncaln:H^1(\Omega) \times H^{1,nc}_0 (\taun, \p) \rightarrow \mathbb R$ as
\begin{equation} \label{ncTerm}
\Ncaln (u,v) =\sum_{\e \in \En }\int_\e \nabla u \cdot \llbracket v \rrbracket_\e.
\end{equation}

\begin{thm} \label{theorem:abstract}
Let~$u$ and~$\un$ be the solutions to~\eqref{weak:form:simple} and~\eqref{ncEVEM}.
Recall that~$\alpha_*(\E)$ and~$\alpha^*(\E)$ are the stability constants in~\eqref{coercivity_continuity}, and~$\Ncaln$ is defined in~\eqref{ncTerm}.
Then, the following a  priori estimate is valid: for all~$\upi$ piecewise in~$\Pbbtilde_{\p}(\E)$ and for all~$\uI \in \Vn$,
\begin{equation} \label{abstract:estimate}
\begin{split}
\vert u - \un \vert_{1,\taun} \le 	& \max_{\E \in \taun} \alpha_*^{-1}(\E) \left\{  \sup_{\vn \in \Vn} \frac{\langle \f, \vn \rangle_n - (\f, \vn)_{0,\Omega} }{\vert \vn \vert_{1,\taun}}  + \sup _{\vn \in \Vn}  \frac{\Ncaln (u,\vn)}{\vert \vn \vert_{1,\taun}}        \right.\\
						&  \quad \quad \left. + \left (1+ \max_{\E\in \taun } \alpha^*(\E) \right) \left( \vert u - \upi \vert_{1,\taun} + \vert u - \uI \vert_{1,\taun}  \right)  \right\}. \\
\end{split}
\end{equation}
\end{thm}
\begin{proof}
The proof follows along the same lines as that in~\cite[Theorem 4.1]{nonconformingVEMbasic}. For the sake of completeness, we carry out all the details.

For all~$\uI \in \Vn$, the triangle inequality entails, for~$\deltan:= \un - \uI$,
\begin{equation} \label{triangle:abstract}
\vert u - \un \vert_{1,\taun } \le \vert u - \uI \vert_{1, \taun} + \vert \uI - \un \vert_{1,\taun} = \vert u - \uI \vert_{1,\taun} + \vert \deltan \vert_{1, \taun}.
\end{equation}
We estimate the second term on the right-hand side: for all~$\upi$ piecewise in~$\Pbbtilde _\p(\E)$, apply~\eqref{coercivity_continuity} and~\eqref{consistency} and get
\[
\begin{split}
\vert \deltan \vert_{1,\taun}^2 	& = \sum_{\E \in \taun} \vert \deltan \vert^2_{1,\E} \le \sum_{\E \in \taun} \alpha_*^{-1}(\E) \anE(\deltan ,\deltan)\\
						& \le \max_{\E \in \taun} \alpha_*^{-1}(\E) \sum_{\E \in \taun} \left\{ \anE(\un, \deltan) - \anE(\uI - \upi, \deltan) -\aE(\upi - u, \deltan) - \aE(u, \deltan)     \right\}.\\
\end{split}
\]
Observe that
\[
\begin{split}
\sum_{\E \in \taun} \aE(u, \deltan) 	& = \sum_{\E \in \taun} \left\{   \int_\E -\Delta u \, \deltan + \int_{\partial \E} \nE \cdot \nabla u\, \deltan      \right\}\\
							& = (\f, \deltan)_{0,\Omega} + \sum_{\e \in \En} \int_\e \nabla u \cdot \llbracket \deltan \rrbracket = : (\f, \deltan)_{0,\Omega} + \Ncaln (u,\deltan). \\
\end{split}
\]
Deduce that
\[
\begin{split}
\vert \deltan \vert^2_{1,\taun}	& \le \max_{\E \in \taun} \alpha_*^{-1}(\E) \Big{\{} \langle \f, \deltan \rangle_n - (\f, \deltan) _{0,\Omega} -\Ncaln(u, \deltan) \\
						& \quad \quad \quad \quad \quad \quad \quad \quad \left.+ \max_{\E \in \taun} \alpha^*(\E) \vert \uI - \upi \vert_{1,\taun} \vert \deltan \vert_{1,\taun} + \vert u - \upi \vert_{1,\taun} \vert \deltan \vert_{1,\taun}   \right\},\\
\end{split}
\]
which implies
\[
\begin{split}
\vert \deltan \vert_{1,\taun}	& \le \max_{\E \in \taun} \alpha_*^{-1}(\E) \Big{\{} \frac{\langle \f, \deltan \rangle_n - (\f, \deltan) _{0,\Omega}}{\vert \deltan \vert_{1,\taun}} -\frac{\Ncaln(u, \deltan)}{\vert \deltan \vert_{1,\deltan}} \\
						& \quad \quad \quad \quad \quad \quad \quad \quad \left.+ (1+\max_{\E \in \taun} \alpha^*(\E)) \vert u - \upi \vert_{1,\taun} + \max_{\E\in \taun} \alpha^*(\E)   \vert u - \uI \vert_{1,\taun}  \right\}.\\
\end{split}
\]
This, together with~\eqref{triangle:abstract}, \lm{yields} the assertion.
\end{proof}

%%%%%%%%%%%%%%%
\subsection{Best enriched polynomial approximation estimates} \label{subsection:best:polynomial}
We show how to estimate from above the term~$\vert u - \upi\vert_{1,\taun}$ on the right-hand side of~\eqref{abstract:estimate} for a specific choice of~$\upi$ piecewise in~$\Pbbtilde_p(\E)$.
\begin{lem} \label{lemma:best_Epolynomial}
Let~$u$ be the solution to problem~\eqref{weak:form:simple}, $\uz$ be as in~\eqref{assumption:solution}, and assumptions (\textbf{A0})-(\textbf{A2}) be valid.
Then, there exists~$\upi$ piecewise in~$\Pbbtilde_\p(\E)$ such that
\[
\vert u - \upi \vert_{1,\taun} \le c \h^{\p} \left\{ \left( \sum_{\E \in \tauno} \vert \uz \vert_{\p+1,\E}^2      \right)^{\frac{1}{2}}   + \left( \sum_{\E \in \tauntw \cup \taunth} \vert u \vert_{\p+1,\E}^2      \right)^{\frac{1}{2}}    \right\},
\]
where~$c$ is a positive constant depending on~$\p$ and~$\gamma$, being~$\gamma$ introduced in (\textbf{A0})-(\textbf{A2}), but is independent of~$\h$ and~$u$.
\end{lem}
\begin{proof}
For all~$\E \notin \tauno$, we have~$\Pbbtilde_\p(\E) = \mathbb P_\p(\E)$. Therefore, we pick~$\upi$ as the best piecewise~$H^1(\E)$ polynomial approximant of~$u$. Deduce that
\begin{equation} \label{standard:polynomial:approx}
\vert  u - \upi \vert_{1,\E} = \inf_{\qp \in \mathbb P_\p(\E)} \vert u - \qp \vert_{1,\E} \le c \hE^\p \vert u \vert_{\p+1,\E}.
\end{equation}
This is a consequence of the smoothness of~$u$ on all~$\E \notin \tauno$ and standard polynomial best approximation estimates; see, e.g., \cite{BrennerScott}.
The constant $c$ depends on the order of accuracy of the method~$\p$ and on the shape of element~$\E$.

If~$\E \in \tauno$, then, $\Pbbtilde_\p(\E)  \supsetneqq \mathbb P_\p(\E)$: the former space is spanned by the latter \emph{plus} the singular function~$\SAE$.
Thus, a suitable choice of~$\upi$ is given by a combination of the singular function~$\SAE$ and the best~$H^1(\E)$ polynomial approximant of~$\uz$, being~$\uz$ introduced in~\eqref{assumption:solution}.
This entails, for some constants~$\widetilde c \in \mathbb R$,
\begin{equation} \label{estimate:Lemma42}
\vert u - \upi \vert_{1,\E} = \inf_{\qp \in \mathbb P_\p(\E)} \vert \uz - \qp + \widetilde c\SAE - \widetilde c \SAE \vert_{1,\E} \le c \hE^\p \vert \uz \vert_{\p+1,\E},
\end{equation}
where~$c$ is a positive constant depending on~$\p$ and on~$\gamma$, being~$\gamma$ introduced in (\textbf{A0})-(\textbf{A2}), but is independent of~$\h$ and~$u$.

Bound~\eqref{estimate:Lemma42} is a consequence of the smoothness of~$\uz$ on all~$\E \in \tauno$
and standard polynomial best approximation estimates; see, e.g., \cite{BrennerScott}.
Collecting the local estimates~\eqref{standard:polynomial:approx} and~\eqref{estimate:Lemma42} and summing up over all the elements, we get the assertion.
\end{proof}
The name of the game in Lemma~\ref{lemma:best_Epolynomial} is that the singular part of the solution is approximated by the singular function in the virtual element spaces on the elements close or containing the singular vertex.

%%%%%%%%%%%%%%%
\subsection{Best interpolation estimates} \label{subsection:best:VEM}
Here, we show how to estimate from above the term~$\vert u -\uI \vert_{1,\taun}$ on the right-hand side of~\eqref{abstract:estimate} for a specific choice of~$\uI$ in~$\Vn$.
In particular, we prove an upper bound on the best interpolation error in nonconforming enriched virtual element spaces in terms of a constant times an enriched polynomial best approximation term.
\begin{lem} \label{lemma:bestInterpolation}
Let~$u$ be any function in~$H^1(\Omega)$. Then, there exists~$\uI \in \Vn$ such that
\[
\vert u - \uI \vert_{1,\taun} \le  \vert u - \upi \vert _{1,\taun}
\]
for all~$\upi$ piecewise in~$\Pbbtilde_\p(\E)$.
\end{lem}
\begin{proof}
The proof follows along the same lines as that of~\cite[Proposition 3.8]{ncHVEM}.
For the sake of completeness, we provide some details.

We define~$\uI \in \Vn$ by imposing the same degrees of freedom as~$u$. More precisely, set
\begin{equation} \label{fixing:dofs}
\begin{split}
& \int_\E (u-\uI) \qpmt = 0 \quad \quad \forall \qpmt \in \mathbb P_{\p-2}(\E),\quad \forall \E \in \taun  ,\\
& \int_\e (u-\uI) \widetilde q_{\p-1}^\e = 0 \quad \quad \forall \widetilde q_{\p-1}^\e \in \Pbbtilde_{\p-1}(\e),\quad \forall \e \in \En .\\
\end{split}
\end{equation}
Recall from the definition of the local virtual element spaces in~\eqref{local:new} that, for all~$\e \in \En$ and~$\E \in \taun$,
\begin{equation} \label{nice:property}
\nE \cdot \nabla \uI \in \Pbbtilde_{\p-1}(\e), \quad\quad \Delta \uI \in  \mathbb P_{\p-2}(\E).
\end{equation}
We deduce
\begin{equation} \label{IBPs2}
\begin{split}
\vert u - \uI \vert^2_{1,\E}	& = \int _\E \nabla (u - \uI) \cdot \nabla (u - \uI) \\
									& = \int_\E -\Delta (u - \uI) (u - \uI) + \int_{\partial \E} \nE \cdot \nabla (u - \uI) (u - \uI)\\
									& \overset{\eqref{fixing:dofs},\,\eqref{nice:property}}{=} \int_\E -\Delta (u - \upi) (u - \uI) + \int_{\partial \E} \nE \cdot \nabla (u - \upi) (u - \uI)\\
									& = \int _\E \nabla (u - \uI) \cdot \nabla (u - \upi) \le \vert u - \upi \vert_{1,\E}  \vert u - \uI \vert_{1,\E}.  \\ 
\end{split}
\end{equation}
The assertion follows dividing both sides by~$\vert u - \uI \vert_{1,\E}$ and summing over all the elements.
\end{proof}
As a consequence, we have the following best interpolation result in nonconforming enriched virtual element spaces.
\begin{prop} \label{proposition:bestInterpolation}
Let~$u$ be the solution to problem~\eqref{weak:form:simple}, $\uz$ be as in~\eqref{assumption:solution}, and assumptions (\textbf{A0})-(\textbf{A2}) be valid.
Then, there exists~$\uI \in \Vn$ such that
\[
\vert u - \uI \vert_{1,\taun} \le  \,c \h^{\p} \left\{ \left( \sum_{\E \in \tauno} \vert \uz \vert_{\p+1,\E}^2      \right)^{\frac{1}{2}}   + \left( \sum_{\E \in \tauntw \cup \taunth} \vert u \vert_{\p+1,\E}^2      \right)^{\frac{1}{2}}    \right\},
\]
where~$c$ is exactly the same constant appearing in the bound of Lemma~\ref{lemma:best_Epolynomial}.
\end{prop}
\begin{proof}
Combine Lemmas~\ref{lemma:best_Epolynomial} and~\ref{lemma:bestInterpolation}.
\end{proof}

%%%%%%%%%%%%%%%
\subsection{Bound on the variational crime due to the right-hand side} \label{subsection:approx:RHS}
Here, we show an upper bound on
\[
\sup_{\vn \in \Vn} \frac{\langle \f, \vn\rangle  -  (\f, \vn)_{0,\Omega}}{\vert \vn \vert_{1,\taun}},
\]
i.e., on the term representing the variational crime perpetrated in the discretization of the right-hand side in~\eqref{strong:form}.

\begin{lem} \label{lemma:RHS}
Given~$\p \in \mathbb N$, let~$\f \in H^{\p - 1}(\Omega)$. Under assumptions (\textbf{A0})-(\textbf{A2}), the following bound is valid:
\[
\sup_{\vn \in \Vn} \frac{\langle \f, \vn\rangle  -  (\f, \vn)_{0,\Omega}}{\vert \vn \vert_{1,\taun}} \le c \h^\p \Vert \f \Vert _{\p-1, \Omega},
\]
where~$c$ is a positive constant depending on~$\p$ and on~$\gamma$, being~$\gamma$ introduced in (\textbf{A0})-(\textbf{A2}).
\end{lem}
\begin{proof}
The proof is exactly the same as in the nonenriched VE conforming setting: no special functions are used in the approximation of the right-hand side; see~\cite[Section 4.7]{VEMvolley} for more details.
\end{proof}

%%%%%%%%%%%%%%%
\subsection{Bound on the variational crime due to the nonconformity} \label{subsection:crime_nc}
Here, we prove an upper bound on the term
\[
\sup_{\vn \in \Vn}  \frac{\Ncaln(u,\vn)}{\vert \vn \vert_{1,\taun}},
\]
i.e., the term representing the variational crime perpetrated when imposing the nonconformity of trial and test spaces.
\begin{lem} \label{lemma:bound_ncTerm}
Let~$u$ be the solution to problem~\eqref{weak:form:simple}, $\uz$ be as in~\eqref{assumption:solution}, $\Ncaln$ be defined in~\eqref{ncTerm}, and assumptions (\textbf{A0})-(\textbf{A2}) be valid.
Then, we have
\[
\frac{\Ncaln(u,\vn)}{\vert \vn \vert_{1,\taun}} \le c \h^\p \left\{  \left( \sum_{\E \in \tauno} \vert \uz \vert^2_{\p,\E}\right)^{\frac{1}{2}} + \left( \sum_{\E \in \tauntw \cup \taunth} \vert u \vert^2_{\p,\E}\right)^{\frac{1}{2}}  \right\} ,
\]
where~$c$ is a positive constant depending on~$\p$ and on~$\gamma$, being~$\gamma$ introduced in (\textbf{A0})-(\textbf{A2}).
\end{lem}
\begin{proof}
We prove the bound edge by edge. Without loss of generality, we assume that~$\e \in \EnI$, for the case~$\e \in \EnB$ can be treated analogously,
and~$\e \in \Eno$, for the case~$\e \in \Entw$ follows as in~\cite[Lemma~4.1]{nonconformingVEMbasic}.

Let~$\E^+$ and~$\E^-$ be the two elements sharing edge~$\e$. We write
\[
\int_\e \nabla u \cdot \llbracket \vn \rrbracket_\e = \int_\e \ne \cdot \nabla u (\vn{}_{|\E^+}- \vn{}_{|\E^-}).
\]
Denote the $L^2(\e)$ projector onto constant functions on~$\e$ by~$\Pi^{0,\e}_0$.
Assumption~\eqref{assumption:solution} (and notably the analiticity of~$\uz$),
the definition of the nonconforming enriched virtual element space in~\eqref{localEVEspace}, Remark~\ref{remark:well-posedness-dofs}, and the properties of orthogonal projectors entail
\begin{equation} \label{an:estimate}
\begin{split}
\int_\e \nabla u \cdot \llbracket \vn \rrbracket 	
& = \int_\e \nabla \uz \cdot \llbracket \vn \rrbracket  + \int_\e \nabla \SA \cdot \llbracket \vn \rrbracket = \int_\e \nabla \uz \cdot \llbracket \vn \rrbracket  \\
& = \int_\e ( \ne \cdot \nabla \uz   - \ne \cdot \Piboldze \nabla \uz  ) (\vn{}_{|\E^+}- \vn{}_{|\E^-} -  \Pi^{0,\e}_0 (\vn{}_{|\E^+}- \vn{}_{|\E^-}))   \\
& \le \Vert \nabla \uz   - \Piboldze \nabla \uz  \Vert_{0,\e} \Vert   \vn{}_{|\E^+}- \vn{}_{|\E^-} -  \Pi^{0,\e}_0 (\vn{}_{|\E^+}- \vn{}_{|\E^-})   \Vert_{0,\e}.
\end{split}
\end{equation}
We estimate the two terms on the right-hand side of~\eqref{an:estimate} separately. We begin with the first one:
using properties of orthogonal projectors, we get
\[
\Vert \nabla \uz   - \Piboldze \nabla \uz \Vert_{0,\e}  \le \Vert \nabla \uz - \Piboldnablapmo \nabla \uz \Vert_{0,\e}.
\]
Apply the trace inequality and the Poincar\'e-Wirtinger inequality~\cite[equation (1.2)]{brenner2003poincare},
which is valid due to the fact that each component of~$\nabla \uz - \Piboldnablapmo \nabla \uz$ has zero average on~$\partial \E$ by the definition of~$\Piboldnablapmo$,
in addition to assumption (\textbf{A2}), and get
\[
\Vert \nabla \uz - \Piboldze \nabla \uz \Vert_{0,\e} \lesssim \hE^{\frac{1}{2}} \vert \nabla \uz - \Piboldnablapmo \nabla \uz \vert_{1,\E}.
\]
Use the standard polynomial approximation theory~\cite{BrennerScott} and assumption (\textbf{A1}) to arrive at
\begin{equation} \label{bound:nc:1}
\Vert \nabla \uz - \Piboldze \nabla \uz \Vert_{0,\e} \lesssim  \hE^{\p - \frac{1}{2}} \vert \uz \vert_{\p,\E}
\end{equation}
Focus now on the second term on the right-hand side of~\eqref{an:estimate}. As proven in~\cite[Lemma~4.1]{nonconformingVEMbasic},
\begin{equation} \label{bound:nc:2}
\Vert    \vn{}_{|\E^+}- \vn{}_{|\E^-} -  \Pi^{0,\e}_0 (\vn{}_{|\E^+}- \vn{}_{|\E^-})   \Vert_{0,\e} \lesssim \h^{\frac{1}{2}} \vert \vn \vert_{1,\E^+ \cup \E^-}.
\end{equation}
The assertion follows combining~\eqref{an:estimate}, \eqref{bound:nc:1}, and~\eqref{bound:nc:2}, and summing over all the edges.
\end{proof}

%%%%%%%%%%%%%%%
\subsection{Stabilizations} \label{subsection:stab}
Here, we exhibit explicit choices of the stabilization~$\SE(\cdot, \cdot)$ introduced in~\eqref{property:stabilization} and discuss their properties.
More precisely, we exhibit a theoretical stabilization, for which we are able to prove the bounds in~\eqref{property:stabilization} explicitly,
assuming that~$\SA \in H^{\frac{3}{2}+\varepsilon}(\E)$ with~$\varepsilon>0$.
To the aim, we shall assume the validity of an inverse estimate for enriched polynomials on the boundary; see inequality~\eqref{equivalence-norms-boundary} below.
Eventually, we introduce a practical stabilization, which we shall widely employ in the numerical experiments in Section~\ref{section:NR} below.

\paragraph*{A theoretical stabilization.}
For every~$\E \in \taun$, define
\begin{equation} \label{theoretical:stabilization}
\SE_T (\un, \vn) = \hE^{-2} (\PizE \un, \PizE \vn)_{0,\E} + \hE^{-1} \sum_{\e \in \EE}(\Pize \un, \Pize \vn) _{0, \e} \quad \forall \un,\,\vn \in \VnE.
\end{equation}
Recall that the projector~$\PizE$ is defined in~\eqref{L2projection}, whereas the projector~$\Pize$ is defined in~\eqref{edge:projector}.

In the proof of Proposition~\ref{proposition:theoretical_stabilization} below, we assume the validity of the following inverse estimate: for all~$\E \in \taun$,
\begin{equation} \label{equivalence-norms-boundary}
\Vert \ne \cdot \nabla \vn \Vert_{0,\partial \E} \lesssim \hE^{-\frac{1}{2}} \Vert \ne \cdot \nabla \vn \Vert_{-\frac{1}{2},\partial \E} \qquad \forall \vn \in \VnE.
\end{equation}
The inverse inequality~\eqref{equivalence-norms-boundary} involves piecewise discontinuous enriched polynomials on the boundary of each element~$\E$.
Standard arguments imply such an inverse estimate, if only standard polynomial spaces are employed; yet, we are currently not able to provide a precise proof for the enriched case and postpone it to future investigations.
Notably, at the present stage, we are not able to claim that the hidden constant does not depend on the singular function.

\begin{prop} \label{proposition:theoretical_stabilization}
Let assumptions (\textbf{A0})-(\textbf{A2}) be valid and~$\SA \in H^{\frac{3}{2}+\varepsilon}(\E)$ with~$\varepsilon>0$.
Assuming the validity of the inverse estimate~\eqref{equivalence-norms-boundary},
the stabilization~$\SE_T(\cdot, \cdot)$ in~$\eqref{theoretical:stabilization}$ satisfies~\eqref{property:stabilization}.
\end{prop}
\begin{proof}
First, we show the lower bound in~\eqref{property:stabilization}.
For every~$\vn \in \ker(\Pinabla)$, using the definition of the local enriched spaces~$\VnE$, we write
\[
\vert \vn \vert^2_{1,\E} = \int_\E \nabla \vn \cdot \nabla \vn = \int_\E -\Delta \vn \PizE \vn + \int_{\partial \E} (\nE \cdot \nabla \vn) \, \Pize \vn.
\]
Use the Cauchy-Schwarz inequality to get
\begin{equation} \label{42.5}
\vert \vn \vert_{1, \E} \le \Vert \Delta \vn \Vert_{0,\E} \Vert \PizE \vn \Vert_{0,\E} + \Vert \nE \cdot \nabla \vn \Vert_{0 ,\partial \E} \Vert \Pize \vn \Vert_{0 , \partial \E}.
\end{equation}
Recall that we assume the validity of~\eqref{equivalence-norms-boundary}.
Recall also that the following inverse inequality is valid; see~\cite[Lemma 10]{cangianigeorgulispryersutton_VEMaposteriori} and~\cite[Theorem 2]{hpVEMcorner}:
\begin{equation} \label{inverse:VEM}
\Vert \Delta \vn \Vert_{0,\E} \lesssim \hE^{-1} \vert \vn \vert_{1,\E} .
\end{equation}
To see~\eqref{inverse:VEM}, we provide some details, which we can be found in two references above:
\[
\Vert \Delta \vn \Vert_{0,\E} 
\lesssim \hE^{-1} \Vert \Delta \vn \Vert_{-1,\E} 
:= \hE^{-1}\sup_{\Phi \in H^1_0(\E)} \frac{(\Delta \vn, \Phi)}{\vert \Phi \vert_{1,\E}}
=\hE^{-1} \sup_{\Phi \in H^1_0(\E)} \frac{(\nabla \vn, \nabla \Phi)_{0,\E}}{\vert \Phi \vert_{1,\E}}
\le \hE^{-1} \vert \vn \vert_{1,\E},
\]
where in the first inequality we used a standard polynomial inverse inequality on polygons,
whence the hidden constant in~\eqref{inverse:VEM} depends on the order of accuracy~$\p$ and on the shape of element~$\E$.

The Neumann trace inequality is valid as well:
\begin{equation} \label{Neumann:trace}
\Vert \nE \cdot \nabla \vn \Vert_{-\frac{1}{2}, \partial \E} \lesssim \vert \vn \vert_{1,\E} + \hE \Vert \Delta \vn \Vert_{0,\E} \overset{\eqref{inverse:VEM}}{\lesssim} \vert \vn \vert_{1,\E}.
\end{equation}
The Neumann trace inequality is valid not only for polynomials or functions in virtual element spaces, but for~$H^1(\E)$ functions with Laplacian in~$L^2(\E)$; see, e.g., \cite[Theorem A.33]{SchwabpandhpFEM},
and the hidden constant depends on the shape of the element~$\E$.

Collecting~\eqref{42.5}, \eqref{inverse:VEM}, and~\eqref{Neumann:trace} leads to
\[
\vert \vn \vert_{1,\E} \lesssim \hE^{-1} \Vert \PizE \vn \Vert_{0,\E} + \hE^{-\frac{1}{2}} \Vert \Pize \vn \Vert_{0, \partial \E},
\]
which is the lower bound in~\eqref{property:stabilization}.
\medskip

Next, we show the upper bound in~\eqref{property:stabilization}. We estimate from above the two terms on the right-hand side of the following identity:
\[
\SE_T (\vn, \vn)  =\hE^{-2} \Vert \PizE \vn\Vert^2_{0,\E} + \hE^{-1} \Vert \Pize \vn \Vert^2_{0, \partial \E}.
\]
As for the first term, we use the stability of orthogonal projections and the Poincar\'e-Wirtinger inequality~\cite[equation (1.2)]{brenner2003poincare}:
\[
\hE^{-2} \Vert \PizE \vn \Vert^2_{0,\E} \lesssim \vert \vn \vert^2_{1,\E}.
\]
As for the second term, use the stability of orthogonal projections, the trace inequality, the Poincar\'e-Wirtinger inequality~\cite[equation (1.2)]{brenner2003poincare} again, and assumption (\textbf{A2}):
\[
\hE^{-1} \Vert \Pize \vn \Vert_{0, \partial \E}^2 \lesssim \vert \vn \vert^2_{1,\E}.
\]
We can apply the Poincar\'e-Wirtinger inequality~\cite[equation (1.2)]{brenner2003poincare}, because~$\vn$ belongs to~$\ker(\Pinabla)$, whence~$\vn$ has zero average on~$\partial \E$.

This concludes the proof.
\end{proof}

\paragraph*{A practical stabilization.}
The stabilization~$\SE_T(\cdot, \cdot)$ introduced in~\eqref{theoretical:stabilization} is computable in terms of the degrees of freedom~\eqref{internal:moments} and~\eqref{edge:moments}.
Notwithstanding, it requires a certain amount of work to implement.

Thence, we suggest to use the following practical stabilization, which is defined on the local canonical basis~\eqref{canonical-basis} as follows: for all~$\E \in \taun$,
\begin{equation} \label{practical:stabilization}
\SE_P (\varphi_i, \varphi_j) = \max(1, \aE (\Pinabla \varphi_i,  \Pinabla \varphi_j)) \quad \quad \forall i,\,j=1,\dots,\dim (\VnE).
\end{equation}
Originally, such a stabilization was introduced for the nonenriched 3D VEM in~\cite{VEM3Dbasic}, and its performance was analyzed in the 2D case in~\cite{fetishVEM}.
To the best of our knowledge, such a stabilization is amongst the most robust from the numerical standpoint in the literature.
Roughly speaking, this stabilization keeps trace somehow of the true energy of the basis functions.

In Section~\ref{section:NR} below, we perform the numerical experiments using the stabilization~$\SE_P(\cdot, \cdot)$ in~\eqref{practical:stabilization} in almost all experiments, for it is easier to implement.
We shall compare the performance of the method employing the two stabilizations in Section~\ref{subsubsection:NR-stab} below.

\begin{remark}
The practical stabilization in~\eqref{practical:stabilization} is a weighted version of the original VEM stabilization in~\cite{VEMvolley}.
It can be checked that an orthonormalization of the nonenriched bulk polynomials in~\eqref{internal:moments} and the enriched edge polynomials in~\eqref{edge:moments} leads to the equivalence of the two stabilizations.
\end{remark}

%%%%%%%%%%%%%%%
\subsection{Convergence of the method} \label{subsection:convergence_method}
In this section, we collect all the a priori bounds hitherto proven and show a convergence result for the $\h$-version of method~\eqref{ncEVEM}.

\begin{thm} \label{theorem:h-version}
Let~$u$ and~$\un$ be the solutions to~\eqref{weak:form:simple} and~\eqref{ncEVEM}, $\uz$ be as in~\eqref{assumption:solution}, and assumptions (\textbf{A0})-(\textbf{A2}) be valid.
Then, we have the following a priori $\h$-convergence result:
\begin{equation} \label{final:estimate}
\vert u - \un \vert_{1,\taun} \le c \h^\p \left\{ \left( \sum_{\E \in \taun^1}\Vert \uz \Vert_{\p+1,\E}^2  + \sum_{\E \in \taun^2\cup \taun^3} \Vert u \Vert_{\p+1,\E}^2  \right)^{\frac{1}{2}}   + \Vert \f \Vert_{\p-1, \Omega}      \right\},
\end{equation}
where~$c$ is a positive constant independent of~$\h$ and~$u$, but which possibly depends on the order of accuracy~$\p$, the parameter~$\gamma$ introduced in (\textbf{A0})-(\textbf{A2}), and the singular function~$\SAE$.
\end{thm}
\begin{proof}
Combine Theorem~\ref{theorem:abstract}, Lemma~\ref{lemma:best_Epolynomial}, Proposition~\ref{proposition:bestInterpolation}, Lemma~\ref{lemma:RHS}, Lemma~\ref{lemma:bound_ncTerm}, and Proposition~\ref{proposition:theoretical_stabilization}.
\end{proof}

\begin{remark} \label{remark:computable-error}
In Theorem~\ref{theorem:h-version}, we proved an upper bound on the error~$\vert u - \un \vert_{1,\taun}$, which is not computable explicitly.
In Section~\ref{section:NR} below, we present a computable quantity based on enriched projections, and prove that it scales as the exact error.
\end{remark}

\begin{remark} \label{remark:first-layer}
Recall that elements~$\E$ in the first layer~$\tauno$ are such that
\begin{itemize}
\item[(i)] either~$\E$ abuts or contains~$\Abf$;
\item[(ii)] or~$\E$ is not too far apart from~$\Abf$.
\end{itemize}
In particular, $\text{dist}(\Abf,\xbfE) \le \gammatilde \diam(\Omega)$, where~$\widetilde{\gamma} \in (0,1)$ is a given parameter.
The reason why~$\tauno$ does not consist of elements abutting~$\Abf$ only is that a uniform $\h$-refinement implies
\[
\vert \Omega \vert -\sum_{\E \in \tauntw \cup \taunth} \vert \E \vert \rightarrow 0 \quad \quad \text{as } \h \rightarrow 0.
\]
This has to be avoided, for otherwise the second term on the right-hand side of~\eqref{final:estimate} would blow up.
\eremk
\end{remark}

%%%%%%%%%%%%%%%
\subsection{Generalizations} \label{subsection:generalizations}
Here, we discuss various generalizations of method~\eqref{ncEVEM}.

\paragraph*{Multiple singularities.}
The first generalization is when expansion~\eqref{assumption:solution} is substituted by the general case~\eqref{decomposition:solution}. We discuss two generalizations.

On the one hand, we can associate more singular functions at one singular vertex/tip.
The extension from the unique singular function is rather straightforward: it suffices to add more singular functions to the local spaces~$\VnE$, for all~$\E \in \taun$ close to~$\Abf$.

On the other hand, we can consider more than one singular vertex/tip. In this case, we define the first layer~$\tauno$ as the set of all elements close to one or more singular vertices-tips.
Thus, the local spaces are enriched with special functions having singularities at different points.

\paragraph*{The 3D case.}
In 3D, the singularities arising from the geometry of the domain are different from those in the 2D case. In particular, vertex and edge singularities have to be dealt with.

The advantage of employing nonconforming virtual element spaces over conforming ones is that the definition of the local spaces is the same as in the two dimensional case:
fiven~$\taun$ a polyhedral decomposition of the physical domain, for all~$\E \in \taun$, we set
\[
\VnE = \left\{ \vn \in H^1(\E)  \mid \Delta \vn \in \mathbb P_{\p-2}(\E),\; \n\cdot \nabla \vn {}_{|F} \in \Pbbtilde_{\p-1}(F) \; \, \forall\, F \text{ face of }\E    \right\}.
\]
Here, $\Pbbtilde_{\p-1}(\E)$ denotes the space of polynomials of degree at most~$\p-1$ over a face~$F$, possibly enriched of the normal derivative of the singular functions that we add to the local spaces.

The design of the method, as well as its analysis (with the exception of the stabilization bounds), follows exactly along the same lines as those of the two dimensional case.
Clearly, the big issue here is to find explicit corner and edge singularities, as well as to design proper quadrature formulas for singular functions over 2D faces.

\paragraph*{Other differential operators.}
It is possible to generalize our setting to PDEs with more general elliptic operators.

Let~$\Lcal$ be an elliptic differential operator of the second order, $\f$ a smooth datum, and~$\gamma_1$ a trace operator.
Consider the following problem: find~$u$ such that
\begin{equation} \label{general:problem}
\begin{cases}
\Lcal u =\f & \text{in } \Omega\\
\gamma_1 u =0 & \text{on } \partial \Omega.\\
\end{cases}
\end{equation}
If~$\Lcal$ is an elliptic operator, then the solution to problem~\eqref{general:problem} presents some singularities due to the geometry of the domain, and more specifically at corners, edges, tips of cracks, etc.
In case the singular behaviour is known explicitly, the solution can be decomposed as in~\eqref{decomposition:solution} into a combination of a smooth and a singular part.

For simplicity, consider the case of a unique singular function~$\S$. Such an~$\S$ belongs to the kernel of operator~$\Lcal$,
i.e., $\Lcal (\S) =0$. Thus, on every element~$\E \in \taun$ sufficiently close to the singular corner, tip, edge etc., we define the local space
\[
\VnE :=  \{ \vn \in \text{H} \mid  \Lcal \vn \in [\mathbb P_{\p_1}(\E)]^{\ell_1},\, \gamma_2(\vn)_{|\e} \in [\mathbb P_{\p_2}(\e)]^{\ell_2} \oplus \gamma_2(S) \, \forall \e \in \EE     \}
\quad \quad \text{$\ell_1$, $\ell_2\in \mathbb N$}.
\]
Here, $\text{H}$ and~$\gamma_2(\cdot)$ denote a suitable Sobolev space and a suitable trace operator.

\medskip

We postpone the generalizations presented in this section to future works.
In view of the ill-conditioning haunting extended methods, the above mentioned generalizations might be not straightforward
and require an orthonormalization procedure as illustrated in Section~\ref{subsection:NR-stab} and Appendix~\ref{appendix:implementation2}.

%%%%%%%%%%%%%%%%%%%%%%%%%%%%%%%%%%%%%%%%%%%%%%%%%%%%%%%%%%%%%%%%%%%%%%%%%%%
\section{Numerical results} \label{section:NR}
%%%%%%%%%%%%%%%%%%%%%%%%%%%%%%%%%%%%%%%%%%%%%%%%%%%%%%%%%%%%%%%%%%%%%%%%%%%
In this section, we present some numerical experiments validating the theoretical results discussed in Section~\ref{section:error_analysis}.

\paragraph*{Errors.}
In order to measure the convergence of the method, we cannot use the exact relative error
\begin{equation} \label{exact:error}
\frac{\vert u  - \un \vert_{1,\taun}}{\vert u \vert_{1,\Omega}}.
\end{equation}
In fact, $\un$ is not known in closed form, but only through its degrees of freedom.
Therefore, we measure the decay of the following \emph{computable} relative error: given~$\widetilde \Pi_\p^\nabla$ defined piecewise as~$\Pinabla$ over~$\taun$,
\begin{equation} \label{computable:error}
\frac{\vert u  - \widetilde \Pi_\p^\nabla \un \vert_{1,\taun}}{\vert u \vert_{1,\Omega}}.
\end{equation}
Errors~\eqref{exact:error} and~\eqref{computable:error} have the same convergence rate in terms of the mesh size~$\h$.
In order to see this, on the one hand, we use the stability of the~$H^1$ projection to get
\[
\vert u - \Pinablaglob \un \vert_{1 , \taun} \le \vert u - \Pinablaglob u \vert_{1,\taun} + \vert \Pinablaglob u - \Pinablaglob  \un \vert _{1,\taun} \le \vert u - \Pinablaglob u \vert_{1,\taun} + \vert u - \un \vert _{1,\taun}.
\]
The convergence of the first term on the right-hand side is provided by  Lemma~\ref{lemma:best_Epolynomial}.

On the other hand, we use Theorem~\ref{theorem:h-version} to get
\[
\begin{split}
\vert  u -\un \vert_{1,\taun}  \lesssim 	&  \max_{\E \in \taun} \alpha_*^{-1}(\E) \left\{  \sup_{\vn \in \Vn} \frac{\langle \f, \vn \rangle_n - (\f, \vn)_{0,\Omega} }{\vert \vn \vert_{1,\taun}}  + \sup _{\vn \in \Vn}  \frac{\Ncaln (u,\vn)}{\vert \vn \vert_{1,\taun}}        \right.\\
							&  \quad \quad \quad \quad \quad \quad \quad \left. + \left (1+ \max_{\E\in \taun } \alpha^*(\E) \right) \vert u - \Pinablaglob \un \vert_{1,\taun} \right\}. \\
\end{split}
\]
The convergence of the first two terms on the right-hand side is provided by Lemmas~\ref{lemma:RHS} and~\ref{lemma:bound_ncTerm}.

\paragraph*{Enriched polynomial basis functions.}
We introduce the basis functions for the nonenriched and enriched bulk and edge polynomial spaces.
Consider the natural bijection between $\mathbb N$ and $\mathbb N_0 \times \mathbb N_0$ given by
\begin{equation} \label{natural:bijection}
1 \leftrightarrow (0,0) ,\quad 2 \leftrightarrow (1,0) ,\quad 3 \leftrightarrow (0,1) , \quad 4 \leftrightarrow (2,0), \quad 5 \leftrightarrow (1,1),\quad 6 \leftrightarrow (0,2) \quad \dots
\end{equation}
Henceforth, given a positive integer scalar~$\alpha$, we denote the corresponding vector through the above bijection by~$\alphabold =(\alpha_1,\alpha_2)$.

Let~$\xbfE=(x_\E, y_\E)$ be the centroid of~$\E$, for all~$\E \in \taun$.
For all~$\E \in \taun$, as for the space~$\mathbb P_\p(\E)$, we consider the basis elements defined as scaled and centred monomials
\begin{equation} \label{monomials}
\malpha^\E(\xbold) := \left( \frac{x-x_\E}{\hE} \right)^{\alpha_1}  \left( \frac{y-y_\E}{\hE} \right)^{\alpha_2} \quad \quad \forall \alpha=1,\dots, \dim(\mathbb P_\p(\E)).
\end{equation}
As for the bulk-enriched polynomial space~$\widetilde{\mathbb P}_\p(\E)$, we consider the monomials in~\eqref{monomials} \emph{plus} the function~$\SAE$ defined in~\eqref{SAE} as basis elements.

Denote the Legendre polynomial of degree~$\alpha\in \mathbb N$ on~$[-1,1]$ by~$\mathbb L_\alpha(x)$.
Given~$\e\in\En$, denote the linear transformation mapping the interval~$[-1, 1]$ to edge~$\e$ by~$\Phi_\e$.
For all~$\e \in \En$, as for the space~$\mathbb P_{\p-1}(\e)$, we consider the basis elements defined as scaled and centred Legendre polynomials
\begin{equation} \label{Legendre}
\Lbbe_\alpha(\xbold) := \mathbb L_\alpha(\Phi_\e^{-1}(\xbold))    \quad\quad \forall \alpha=1,\dots, \dim (\mathbb P_{\p-1}(\e)).
\end{equation}
As for the enriched polynomial space~$\widetilde{\mathbb P}_{\p-1}(\e)$, we consider the Legendre polynomials in~\eqref{Legendre}
\emph{plus} the function~$\ne\cdot \nabla \SAe{}_{|\e}$, where~$\SAe$ is defined in~\eqref{SAe},  as basis elements.

For future convenience, the basis elements for the bulk-enriched polynomial spaces are denoted by
\[
\mtildealpha^\E \quad \quad \forall \alpha=1, \dots ,\dim(\widetilde{\mathbb P}_{\p}(\E)),
\]
whereas for the edge-enriched polynomial spaces
\[
\mtildealpha^\e \quad \quad \forall \alpha=1, \dots ,\dim(\widetilde{\mathbb P}_{\p-1}(\e)).
\]
\paragraph*{Stabilization.}
We employ the stabilization defined in~\eqref{practical:stabilization}
throughout and check the performance of the method using also the theoretical stabilization~\eqref{theoretical:stabilization} in Section~\ref{subsubsection:NR-stab} below.
\begin{remark} 
The nonconforming setting allows for an effective improvement of the condition number of the final system.
In order to get such an improvement, we suggest to orthonormalize the basis functions of the edge and bulk-enriched polynomial basis functions.
This is very much in the spirit of~\cite{fetishVEM};
see Section~\ref{subsection:NR-stab} and Appendix~\ref{appendix:implementation2}.
\end{remark}
\medskip

Next, we present the test cases we shall analyse numerically. In all of them, we consider solutions that are singular only at~$(0,0)$.

%%%%%%%%%%%%%%%%
\paragraph*{Test case 1.}
%%%%%%%%%%%%%%%%
The first test case is defined on the L-shaped domain
\begin{equation} \label{Omega1}
\Omega_1= (-1,1)^2 \setminus [0,1) \times(-1,0].
\end{equation}
Let~$(r,\theta)$ be the polar coordinates at the re-entrant corner~$(0,0)$.
We are interested in the approximation of the exact solution
\begin{equation} \label{u1}
u_1(x,y) = u_1(r,\theta) = \sin(\pi \, x) \sin(\pi \, y) + r^{\frac{2}{3}} \sin \left (\frac{2}{3} \theta \right).
\end{equation}
The primal formulation of the problem we are interested in is such that we have: zero Dirichlet boundary conditions on the edges generating the re-entrant corner; suitable Dirichlet boundary conditions on all the other edges;
right-hand side computed according to~\eqref{u1}.
The right-hand side of problem~\eqref{weak:form:simple} with exact solution~$u_1$ is smooth since the singular function is harmonic.

In Figure~\ref{figure:frame}, we plot the geometry with the reference frame and the adopted boundary conditions.
\begin{figure}  [h]
\begin{center}
\begin{tikzpicture}[scale=2] 
\draw[black, very thick, -] (0,0) -- (1,0) -- (1,1) -- (-1,1) -- (-1,-1) -- (0,-1) -- (0,0) ;
\draw[red, very thick, -] (0,0) -- (1,0); \draw[red, very thick, -] (0,0) -- (0,-1);
\draw[blue, very thick, -] (1,0) -- (1,1) -- (-1,1) -- (-1,-1) -- (0,-1);
\draw(.2,-.15) node[black] {{$(0,0)$}}; \draw[black, fill] (0,0) circle (0.5mm);
\draw(.8,.15) node[black] {{$(1,0)$}}; \draw[black, fill] (1,0) circle (0.5mm);
\end{tikzpicture}
\end{center}
\caption{Geometry with the reference frame and the adopted boundary conditions for test case~1, with exact solution~$u_1$ in~\eqref{u1}.
In red and blue, we highlight the edges where homogeneous and nonhomogeneous Dirichlet boundary conditions are imposed.}
\label{figure:frame}
\end{figure}
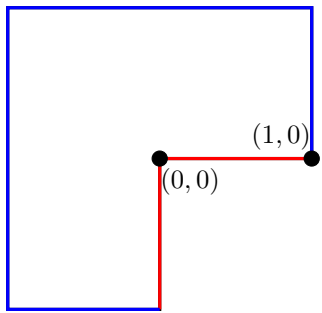

%%%%%%%%%%%%%%%%
\paragraph*{Test case 2.}
%%%%%%%%%%%%%%%%
The second test case is defined on the slit square domain
\[
\Omega_2 := (-1,1)^2 \setminus \{  (x,0) \in \mathbb R^2 \mid x \ge 0 \}.
\]
Let~$(r,\theta)$, $\theta \in [0 , 2 \pi)$, be the polar coordinates at~$(0,0)$. We are interested in the approximation of the exact solution
\begin{equation} \label{u2}
u_2(r,\theta) = r^{\frac{1}{2}} \sin \left(\frac{1}{2} \theta \right).
\end{equation}
The primal formulation of the problem we are interested in is such that we have: suitable Dirichlet boundary conditions on all edges; zero right-hand side, since~$u_2$ is harmonic.

%%%%%%%%%%%%%%%%
\paragraph*{Test case 3.}
%%%%%%%%%%%%%%%%
The third test case is defined on the L-shaped domain~\eqref{Omega1}.
Let~$(r,\theta)$ be the polar coordinates at the re-entrant corner~$(0,0)$.
We are interested in the approximation of the exact solution
\begin{equation} \label{u3}
u_3(x,y) = u_3(r,\theta) = r^{\frac{2}{3}} \sin \left (\frac{2}{3} \theta \right).
\end{equation}
The primal formulation of the problem we are interested in is such that we have: zero Dirichlet boundary conditions on the edges generating the re-entrant corner; suitable Dirichlet boundary conditions on all the other edges;
zero right-hand side.

\begin{remark} \label{remark:testcase-Dirichlet}
Solutions~$u_1$, $u_2$, and~$u_3$ have nonhomogeneous boundary conditions.
In order to cope with them, we refer to Remark~\ref{remark:Neumann}.
\end{remark}

%%%%%
\paragraph*{Meshes.}
%%%%%
In the forthcoming numerical experiments, we employ sequences of uniform Cartesian meshes; see Figure~\ref{figure:meshes:L-shaped} for examples of such meshes for the two test cases.
As for test case~2, we highlight the slit in colour: the couples of adjacent squares do \emph{not} share the same edge.
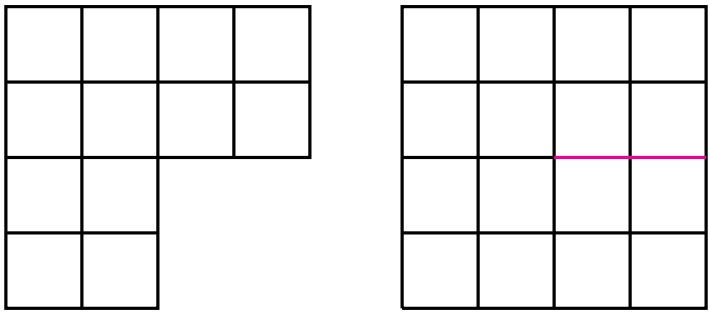
\begin{figure}  [h]
\begin{center}
\begin{tikzpicture}[scale=2] 
\draw[black, very thick, -] (0,0) -- (1,0) -- (1,1) -- (-1,1) -- (-1,-1) -- (0,-1) -- (0,0) ;
\draw[black, very thick, -] (0.5,0) -- (0.5,1); \draw[black, very thick, -] (-0.5,-1) -- (-0.5,1);  \draw[black, very thick, -] (-1,0) -- (0,0); 
\draw[black, very thick, -] (-1,0.5) -- (1,0.5); \draw[black, very thick, -] (-1,-.5) -- (0,-0.5);  \draw[black, very thick, -] (0,0) -- (0,1);  
\end{tikzpicture}
\quad \quad \quad 
%%%%%
\begin{tikzpicture}[scale=2] 
\draw[black, very thick, -] (-1,-1) -- (1,-1) -- (1,1) -- (-1,1) -- (-1,-1);
\draw[black, very thick, -] (-.5,-1) -- (-.5,1); \draw[black, very thick, -] (.5,-1) -- (.5,1);
\draw[black, very thick, -] (-1,-.5) -- (1,-.5); \draw[black, very thick, -] (-1,.5) -- (1,.5);
\draw[black, very thick, -] (0,-1) -- (0,1);  \draw[black, very thick, -] (-1,0) -- (0,0); 
\draw[magenta, very thick, -] (0,0) -- (1,0);
\end{tikzpicture}
\end{center}
\caption{Meshes that we employ in the numerical experiments. \emph{Left panel}: test case~1. \emph{Right panel}: test case~2.}
\label{figure:meshes:L-shaped}
\end{figure}

We depict the first mesh of the sequence of meshes for test case~$2$ in Figure~\ref{figure:first-mesh}: it consists of a single heptagonal element with two edges having endpoints sharing coordinates.
Thus, we show that the virtual element method works and is robust also on degenerate polygons.
In some cases, we shall also use Voronoi meshes for the sake of testing the method on general polygonal meshes; see Section~\ref{subsubsection:NR-stab}.
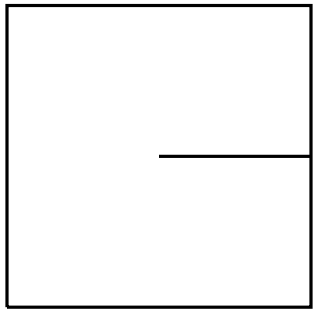
\begin{figure}  [h]
\begin{center}
\begin{tikzpicture}[scale=2] 
\draw[black, very thick, -] (-1,-1) -- (1,-1) -- (1,1) -- (-1,1) -- (-1,-1);
\draw[black, very thick, -] (0,0) -- (1,0);
\end{tikzpicture}
\end{center}
\caption{First mesh of the sequence of meshes employed for test case~$2$. It consists of a single heptagonal element with two edges having endpoints sharing coordinates.}
\label{figure:first-mesh}
\end{figure}

%%%%%
\paragraph*{Layers.}
%%%%%
We consider different distributions of the layers; see Section~\ref{subsection:EVES}.
We test the method assuming that all the elements of the mesh belong to the first layer defined in~\eqref{first-layer:definition}, i.e., we enrich all the local spaces.
In other words, we pick~$\gammatilde = +\infty$ in~\eqref{first-layer:definition}.
Furthermore, we consider the partially enriched scheme, i.e., we enrich only the elements in a neighbourhood of the singular point~$\Abf$.
In Figure~\ref{figure:L-shaped-layers}, we depict the two different layerisations for the same Cartesian mesh for test case~1. We pick~$\gammatilde=1/10$ in~\eqref{first-layer:definition}.
The number of enriched elements depend on the size of the mesh: the finer is the mesh, the more elements are enriched.
The partially enriched scheme corresponds to the geometric enrichement in the XFEM.
A similar layerisation is valid for test case~2 and is therefore not shown.
Further, we shall compare the new enriched method with the nonenriched one~\cite{nonconformingVEMbasic}. In this case, all the elements belong to the third layer defined in~\eqref{third:layer}.

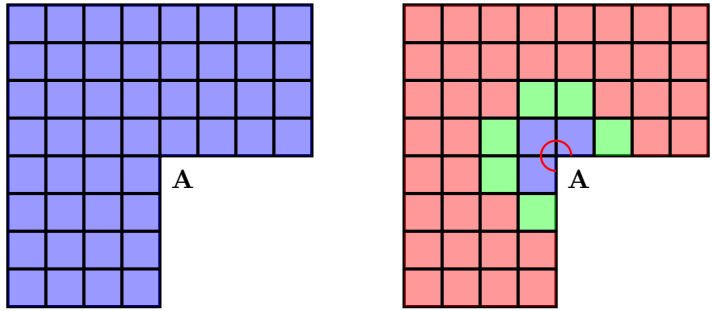
\begin{figure}  [h]
\begin{center}
\begin{tikzpicture}[scale=2] 
\draw[black, very thick, -] (0,0) -- (1,0) -- (1,1) -- (-1,1) -- (-1,-1) -- (0,-1) -- (0,0) ;
\fill[blue, opacity=0.4] (0,0) -- (1,0) -- (1,1) -- (-1,1) -- (-1,-1) -- (0,-1) -- (0,0) ;
\draw[black, very thick, -] (-.75,-1) -- (-.75,1); \draw[black, very thick, -] (-.5,-1) -- (-.5,1);  \draw[black, very thick, -] (-.25,-1) -- (-.25,1); 
\draw[black, very thick, -] (0,0) -- (0,1);  \draw[black, very thick, -] (0.25,0) -- (0.25,1);  \draw[black, very thick, -] (0.5,0) -- (0.5,1);  \draw[black, very thick, -] (0.75,0) -- (0.75,1); 
\draw[black, very thick, -] (-1,-.75) -- (0,-.75); \draw[black, very thick, -] (-1,-.5) -- (0,-.5);  \draw[black, very thick, -] (-1,-.25) -- (0,-.25);  \draw[black, very thick, -] (-1,-0) -- (0,-.0); 
\draw[black, very thick, -] (-1,.25) -- (1,.25); \draw[black, very thick, -] (-1,.5) -- (1,.5); \draw[black, very thick, -] (-1,.75) -- (1,.75);
\draw(.15,-.15) node[black] {{$\Abf$}};
\end{tikzpicture}
\quad \quad \quad 
\begin{tikzpicture}[scale=2] 
\draw[black, very thick, -] (0,0) -- (1,0) -- (1,1) -- (-1,1) -- (-1,-1) -- (0,-1) -- (0,0) ;
\fill[blue, opacity=0.4] (0.25,0) -- (0.25, 0.25) -- (-0.25, 0.25) -- (-0.25,-0.25) -- (0,-0.25) -- (0,0) -- (0.25,0);
\fill[green, opacity=0.4] (0.25, 0) -- (0.5,0) -- (.5,.5) -- (-.5,.5) -- (-.5,-.5) -- (0,-.5) -- (0,-.25) -- (-.25,-.25) -- (-.25,.25) -- (.25,.25) -- (.25,0);
\fill[red, opacity=0.4] (0.5, 0) -- (1,0) -- (1,1) -- (-1,1) -- (-1,-1) -- (0,-1) -- (0,-.5) -- (-.5,-.5) -- (-.5,.5) -- (.5,.5) -- (.5,0);
\fill[white] (0.25, 0.25) -- (0.5, 0.25) -- (0.5, 0.5) -- (0.25, 0.5) -- (0.25, 0.25);
\fill[red, opacity=0.4] (0.25, 0.25) -- (0.5, 0.25) -- (0.5, 0.5) -- (0.25, 0.5) -- (0.25, 0.25);
\fill[white] (-0.25, 0.25) -- (-0.5, 0.25) -- (-0.5, 0.5) -- (-0.25, 0.5) -- (-0.25, 0.25);
\fill[red, opacity=0.4] (-0.25, 0.25) -- (-0.5, 0.25) -- (-0.5, 0.5) -- (-0.25, 0.5) -- (-0.25, 0.25);
\fill[white] (-0.25, -0.25) -- (-0.5, -0.25) -- (-0.5, -0.5) -- (-0.25, -0.5) -- (-0.25, -0.25);
\fill[red, opacity=0.4] (-0.25, -0.25) -- (-0.5, -0.25) -- (-0.5, -0.5) -- (-0.25, -0.5) -- (-0.25, -0.25);
\draw[black, very thick, -] (-.75,-1) -- (-.75,1); \draw[black, very thick, -] (-.5,-1) -- (-.5,1);  \draw[black, very thick, -] (-.25,-1) -- (-.25,1); 
\draw[black, very thick, -] (0,0) -- (0,1);  \draw[black, very thick, -] (0.25,0) -- (0.25,1);  \draw[black, very thick, -] (0.5,0) -- (0.5,1);  \draw[black, very thick, -] (0.75,0) -- (0.75,1); 
\draw[black, very thick, -] (-1,-.75) -- (0,-.75); \draw[black, very thick, -] (-1,-.5) -- (0,-.5);  \draw[black, very thick, -] (-1,-.25) -- (0,-.25);  \draw[black, very thick, -] (-1,-0) -- (0,-.0); 
\draw[black, very thick, -] (-1,.25) -- (1,.25); \draw[black, very thick, -] (-1,.5) -- (1,.5); \draw[black, very thick, -] (-1,.75) -- (1,.75);
\draw[red,thick,domain=0:270] plot ({1/10*cos(\x)}, {1/10*sin(\x)});
\draw(.15,-.15) node[black] {{$\Abf$}};
\end{tikzpicture}
\end{center}
\caption{Two different layerisations for  the same Cartesian mesh on the L-shaped domain~$\Omega_1$ in~\eqref{Omega1}.
\emph{Left panel}: all the elements belong to the enriched layer~$\tauno$.
\emph{Right panel}: there are three layers; we pick~$\gammatilde=1/10$ in~\eqref{first-layer:definition}.
We depict: in blue, the elements in the first element layer~$\tauno$; in green, the elements in the second element layer~$\tauntw$; in red, the elements in the third layer~$\taunth$.}
\label{figure:L-shaped-layers}
\end{figure}

%%%%%%%%%%%%%%%%%%%%%
\subsection{Numerical experiments on the L-shaped domain} \label{subsection:nr-L-shaped}
%%%%%%%%%%%%%%%%%%%%%
In this section, we present several numerical experiments for test case~1, with exact solution~$u_1$ defined in~\eqref{u1}, using the fully enriched, the partially enriched, and the nonenriched~\cite{nonconformingVEMbasic} schemes.
Notably, we are interested in the performance of the~$\h$- and the $\p$-versions of the method, which are the topic of Sections~\ref{subsubsection:h-version-L-shaped} and~\ref{subsubsection:p-version-L-shaped}.
We investigate two additional computational aspects in Sections~\ref{subsubsection:optimizing-layerisation} and~\ref{subsubsection:no-full-layerisation}.
First, on a single mesh, we tune the parameter~$\gammatilde$ in~\eqref{first-layer:definition} in order to optimize the error of the method;
next, we provide a heuristic motivation as for why the fully enriched scheme turns out to be less stable than the partially enriched one.
Eventually, in Section~\ref{subsubsection:NR-stab}, we compare the performance of the method on Voronoi meshes and use the theoretical and practical stabilizations.

%%%%%%%
\subsubsection{The $\h$-version} \label{subsubsection:h-version-L-shaped}
%%%%%%%
In Figure~\ref{figure:h-version-Cartesian:L-shaped}, we present numerical results for the $\h$-version of the method.
We consider the solution~$u_1$ in~\eqref{u1} and degrees of accuracy~$\p=1$, $2$, and~$3$.
We use sequences of Cartesian meshes as those in Figure~\ref{figure:meshes:L-shaped}, and compare the performance of the fully enriched (left) and partially enriched methods (right).
Further, we plot the error of the nonenriched nonconforming VEM of~\cite{nonconformingVEMbasic}; see Figure~\ref{figure:h-version-Cartesian:L-shaped} (bottom).

\begin{figure}  [h]
\centering
\includegraphics [angle=0, width=0.45\textwidth]{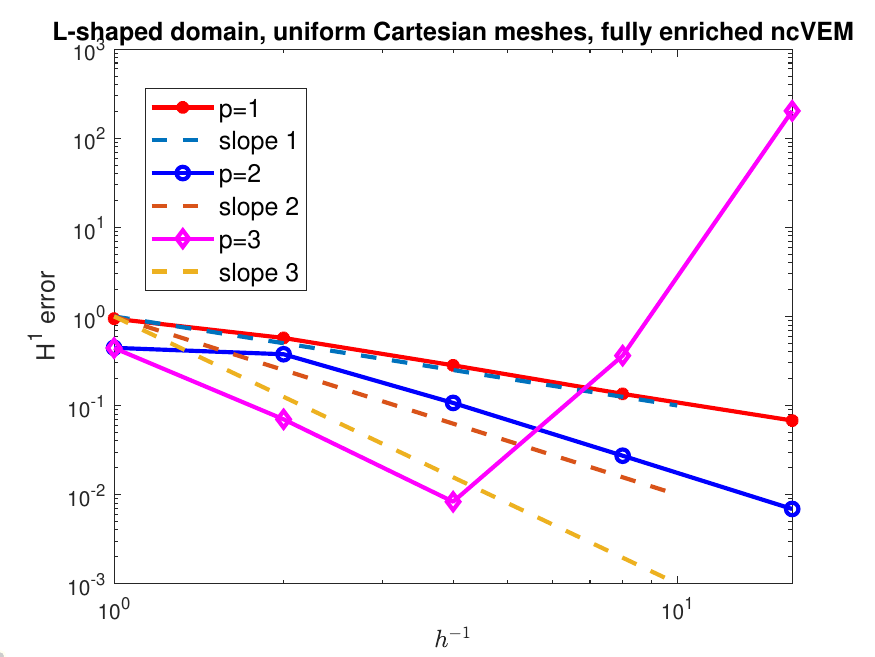}
\includegraphics [angle=0, width=0.45\textwidth]{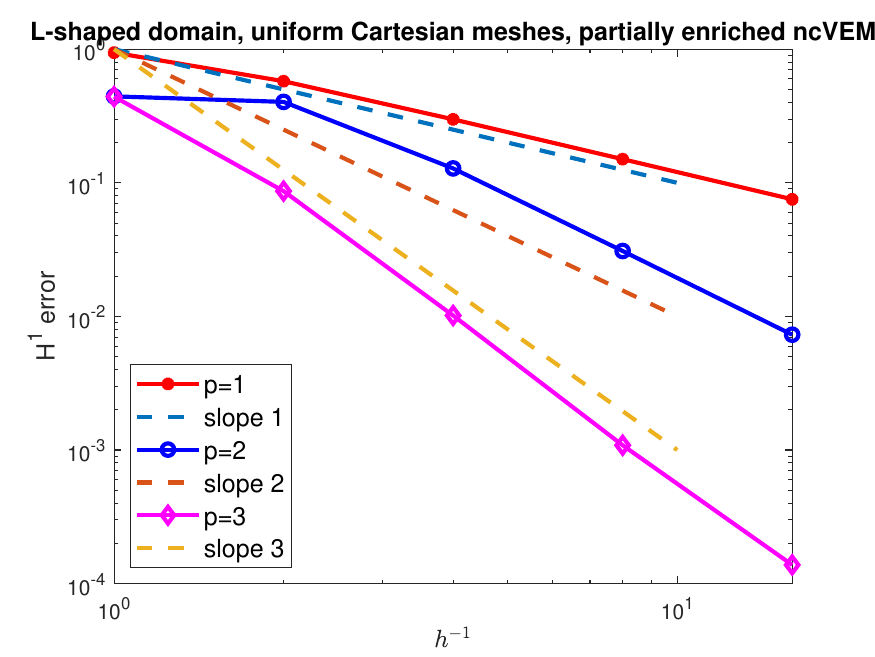}
\includegraphics [angle=0, width=0.45\textwidth]{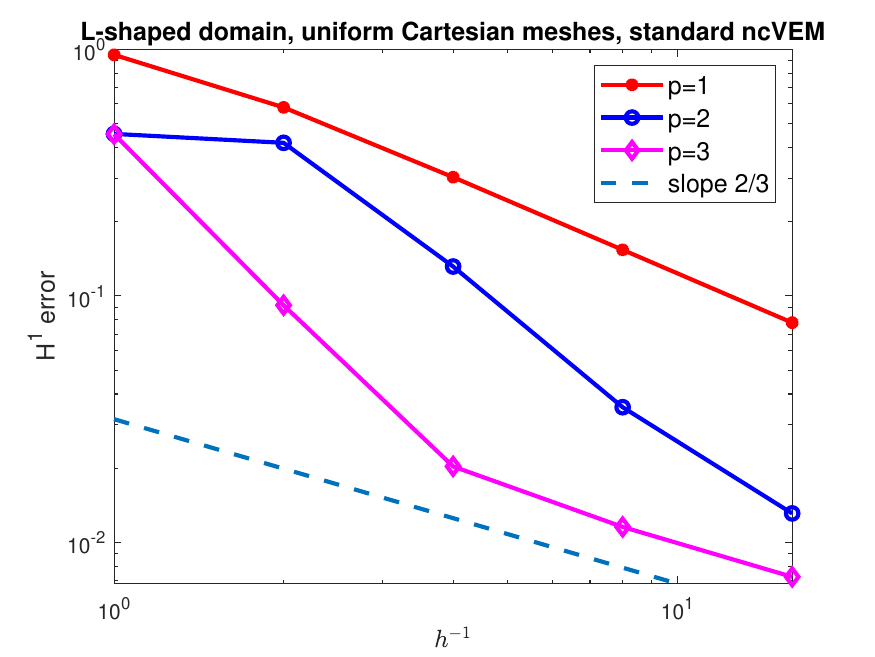}
\caption{$\h$-version of various versions of the method in order to approximate the solution~$u_1$ in~\eqref{u1}.
The polynomial order of accuracy of the method is~$\p=1$, $2$, and~$3$. We employ sequences of uniform Cartesian meshes as those in Figure~\ref{figure:meshes:L-shaped} (left).
\emph{Left panel}: fully enriched method.
\emph{Right panel}: partially enriched method ($\gammatilde=1/10$). Bottom panel: nonenriched method.}
\label{figure:h-version-Cartesian:L-shaped}
\end{figure}

In Figure~\ref{figure:h-version-Cartesian:L-shaped} (left) and (right), we observe an optimal rate of convergence.
In particular, for all the three degrees of accuracy~$\p$, the two enriched method converge with order~$\p$.
On the other hand, the nonenriched method converges suboptimally, which we can expect from the analysis of~\cite{nonconformingVEMbasic}; see Figure~\ref{figure:h-version-Cartesian:L-shaped} (bottom).

The fully enriched scheme suffers from a loss a convergence for fine meshes.
We investigate the reasons of this phenomenon in Section~\ref{subsubsection:no-full-layerisation} below.
Further, we shall provide a remedy for such a loss of accuracy in Section~\ref{subsection:NR-stab} and Appendix~\ref{appendix:implementation2} below.

%%%%%%%
\subsubsection{The $\p$-version} \label{subsubsection:p-version-L-shaped}
%%%%%%%
In Figure~\ref{figure:p-version-Cartesian:L-shaped}, we investigate the behaviour of the $\p$-version of the fully and partially enriched ($\gammatilde=1/10$) versions of the method.
Further, we plot the decay of the error employing the nonenriched nonconforming VEM of~\cite{nonconformingVEMbasic}.
We consider the solution~$u_1$ in~\eqref{u1} and fix a uniform Cartesian mesh consisting of~$48$ elements; see Figure~\ref{figure:L-shaped-layers}.
What we could expect combining the standard theory for the $\p$-version of Galerkin methods, see, e.g., \cite{babuskasurihpversionFEMwithquasiuniformmesh},
and the analysis developed in this paper, is that the enriched method converges exponentially in terms of~$\p$.
On the other hand, we expect an algebraic rate of convergence for the nonenriched method.

\begin{figure}  [h]
\centering
\includegraphics [angle=0, width=0.45\textwidth]{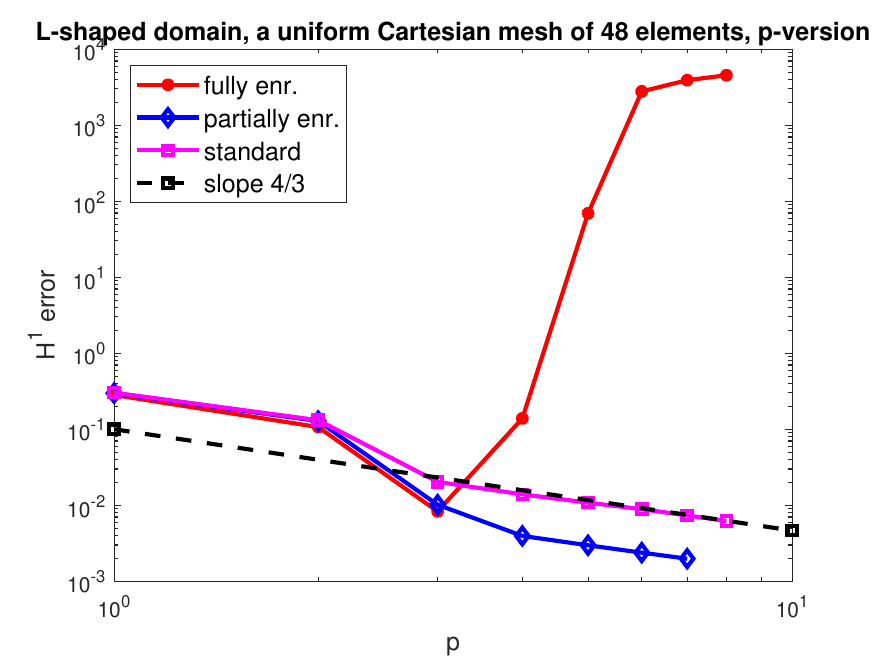}
\caption{$\p$-version of various versions of the method in order to approximate the solution~$u_1$ in~\eqref{u1}.
We take a uniform Cartesian mesh consisting of~$48$ elements
and employ the fully enriched, the partially enriched ($\gammatilde=1/10$), and the nonenriched methods.}
\label{figure:p-version-Cartesian:L-shaped}
\end{figure}

From Figure~\ref{figure:p-version-Cartesian:L-shaped}, we observe several facts.
The nonenriched version of the nonconforming VEM converges algebraically,
whereas the error computed with the fully enriched method blows up. This is due to the ill-conditioning of the final system; see Section~\ref{subsubsection:no-full-layerisation} below.

More surprisingly, the partially enriched method presents an exponential pre-asymptotic behaviour, up to~$\p=3$, but then the convergence turns to be algebraic. This is due again to the ill-conditioning.
When computing error~\eqref{computable:error}, round-off errors prevent us to have the correct coefficient in the expansion of~$\Pinabla$ so that we are not able to eliminate the singularity in the solution.

As a positive note, we observe that the $\p$-version of the partially enriched method performs one order of magnitude better than the nonenriched one.
We shall provide a remedy for such a loss of accuracy in Section~\ref{subsection:NR-stab} below.

%%%%%%%
\subsubsection{On the choice of the parameter~$\gammatilde$ in~\eqref{first-layer:definition}} \label{subsubsection:optimizing-layerisation}
%%%%%%%
In this section, we investigate how the choice of the parameter~$\gammatilde$ appearing in~\eqref{first-layer:definition} influences the performance of the method.
To this aim, consider the solution~$u_1$ in~\eqref{u1}, fix a uniform Cartesian mesh with~$192$ elements, consider~$14$ equispaced values of~$\gammatilde$ in~$(0.1, 1.4)$,
and depict the error of the method for each choice of~$\gammatilde$ in Figure~\ref{figure:tuning-parameter:L-shaped}.
\begin{figure}  [h]
\centering
\includegraphics [angle=0, width=0.45\textwidth]{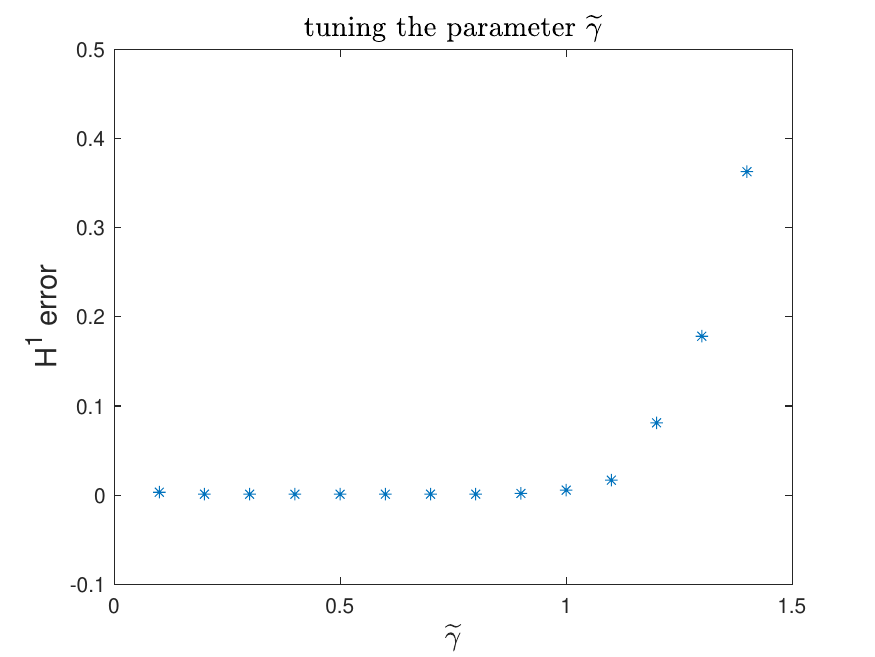}
\caption{Performance of the partially enriched method picking different choices of~$\gammatilde$ in~\eqref{first-layer:definition}.
In particular, we pick~$14$ equispaced values of~$\gammatilde$ in~$(0.1, 1.4)$
The exact solution is~$u_1$ in~\eqref{u1}. We employ a uniform Cartesian mesh with 192 elements.}
\label{figure:tuning-parameter:L-shaped}
\end{figure}

It turns out that the optimal choice of the parameter~$\gammatilde$ lies in the range $\gammatilde \in (0.1,1)$.
For larger choices of~$\gammatilde$, the error grows as the ill-conditioning of the system increases.

Needless to say, the above analysis of the best parameter is valid for the current test case and ought to be performed for every exact solution.

%%%%%%%
\subsubsection{On why the fully enriched scheme is more ill-conditioned than the partially enriched one} \label{subsubsection:no-full-layerisation}
%%%%%%%
In Sections~\ref{subsubsection:h-version-L-shaped} and~\ref{subsubsection:optimizing-layerisation}, we observed that the fully enriched scheme is more ill-conditioned than the partially enriched one; see Figure~\ref{figure:h-version-Cartesian:L-shaped}.
In this section, we give some heuristic motivations as for the reason why this happens.

The motivation behind the growth of the ill-conditioning for fine meshes is due to the behaviour of the singular enrichment function~$\SAE$.
On the elements that are close to the singular vertex~$\Abf$, $\SAE$ differs from all the scaled monomials, which span the nonenriched polynomial basis.
On the other hand, on the elements that are far from~$\Abf$, the singular function $\SAE$ becomes close to a constant function, especially for small elements.
In a sense, the basis functions of the virtual element space become close to be linearly dependent, as keeping on refining the mesh.

In Figure~\ref{figure:LD}, we depict the singular function~$\SAE$ with singular behaviour given by~$\alpha=2/3$ and a constant function along the radial component on two intervals of length~$0.03$.
The first one (left) is close to~$\Abf$, the second one (right) is slightly far from it.
\begin{figure}  [h]
\centering
\includegraphics [angle=0, width=0.45\textwidth]{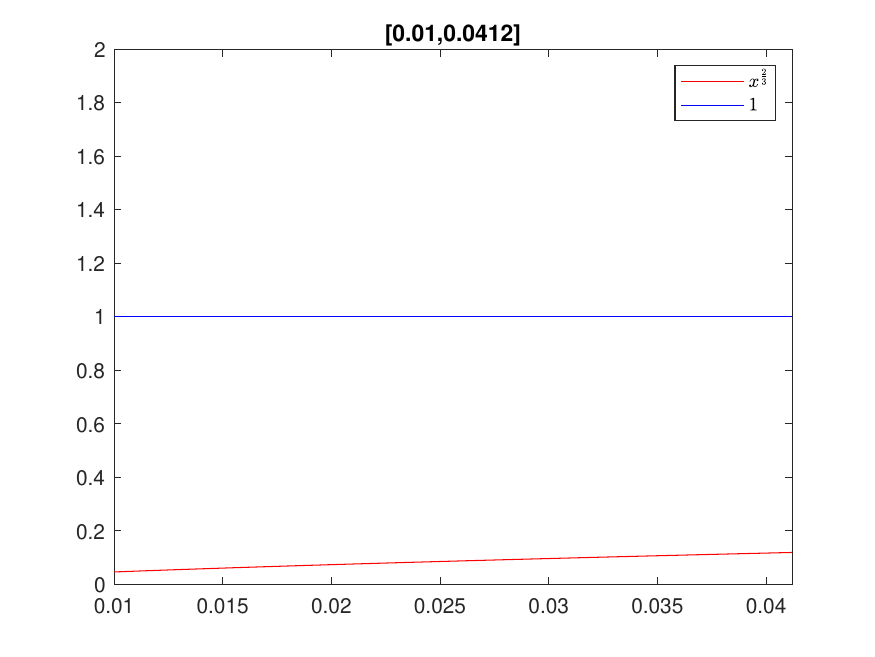}
\includegraphics [angle=0, width=0.45\textwidth]{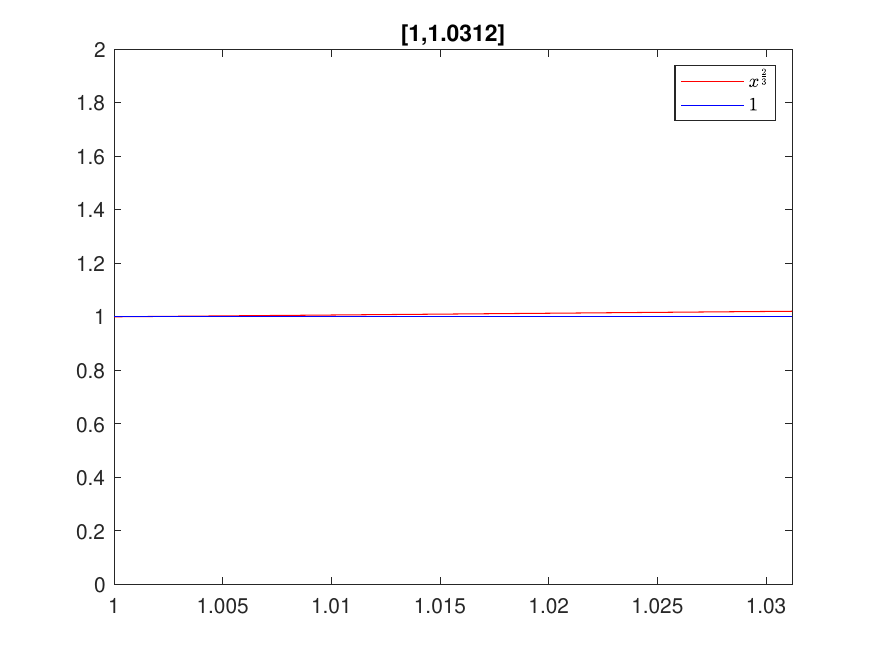}
\caption{Behaviour of~$\SA$ and a constant function along the radial component.
\emph{Left panel}: interval~$[0.01,0.04]$.
\emph{Right panel}: interval~$[0.01,0.04]$.}
\label{figure:LD}
\end{figure}

From Figure~\ref{figure:LD}, it is apparent that the singular functions get extremely close to a constant function, thus leading to an ill-conditioned system.
A possible remedy to this situation could be to use bulkwise, see~\cite{fetishVEM}, and edgewise orthonormalization techniques, see Section~\ref{subsection:NR-stab} and Appendix~\ref{appendix:implementation2}.

%%%%%%%
\subsubsection{Different stabilizations and general polygonal meshes} \label{subsubsection:NR-stab}
%%%%%%%
Here, we present a numerical comparison of the performance of the method employing the theoretical~\eqref{theoretical:stabilization} and practical~\eqref{practical:stabilization} stabilizations.
Moreover, we consider Voronoi meshes, in order to show the robustness of the proposed method on more general meshes, albeit our method is new on quadrilateral and triangular meshes as well.

For the test case with exact solution~$u_1$ in~\eqref{u1},
we run an $\h$-version of the partially enriched method with~$\gammatilde = 1/10$ in~\eqref{first-layer:definition} and~$\p=1$, $2$, and~$3$; see Figure~\ref{figure:comparison-stab}.
\begin{figure}  [h]
\centering
\includegraphics [angle=0, width=0.45\textwidth]{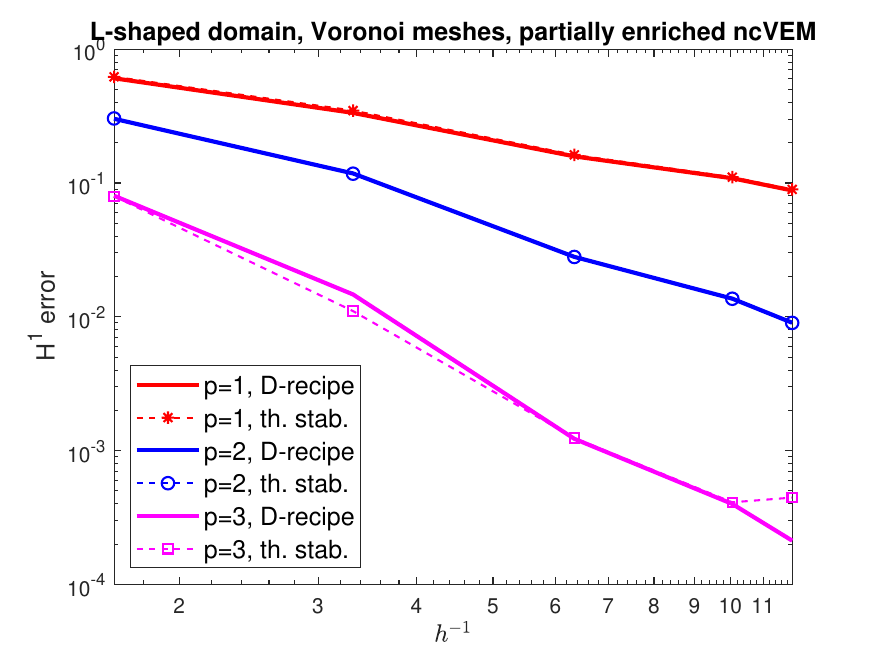}
\caption{$\h$-version of the partially enriched method with~$\gammatilde = 1/10$ in~\eqref{first-layer:definition} and~$\p=1$, $2$, and~$3$,
employing the theoretical~\eqref{theoretical:stabilization} and the practical~\eqref{practical:stabilization} stabilizations.}
\label{figure:comparison-stab}
\end{figure}

The two stabilizations lead to analogous performance of the method. Yet, the practical one is slighlty more robust for the higher order case, as could have been expected~\cite{fetishVEM}.

%%%%%%%%%%%%%%%%%%%%%
\subsection{Numerical experiments on the slit domain: the extended patch test} \label{subsection:nr-slit-cracks}
%%%%%%%%%%%%%%%%%%%%%
In this section, we verify that the method works also on the test case with exact solution~$u_2$ in~\eqref{u2}.
We investigate the performance of the $\h$-version of the method only, employing sequences of uniform Cartesian meshes as in Figure~\ref{figure:meshes:L-shaped} (right).

Test case~$2$ can be regarded as an extended patch test: the exact solution is equal, up to constants, to the enrichment function~$\SAE$.
Thence, if we consider the fully enriched version of the method then the error is zero up to machine precision, thanks to the enriched consistency of the discrete bilinear forms in~\eqref{consistency}.
In Figure~\ref{figure:other1}, we depict the decay of the error of the method for~$\p=1$, $2$, and~$3$.
We consider the fully enriched version of the method.

\begin{figure}  [h]
\centering
\includegraphics [angle=0, width=0.45\textwidth]{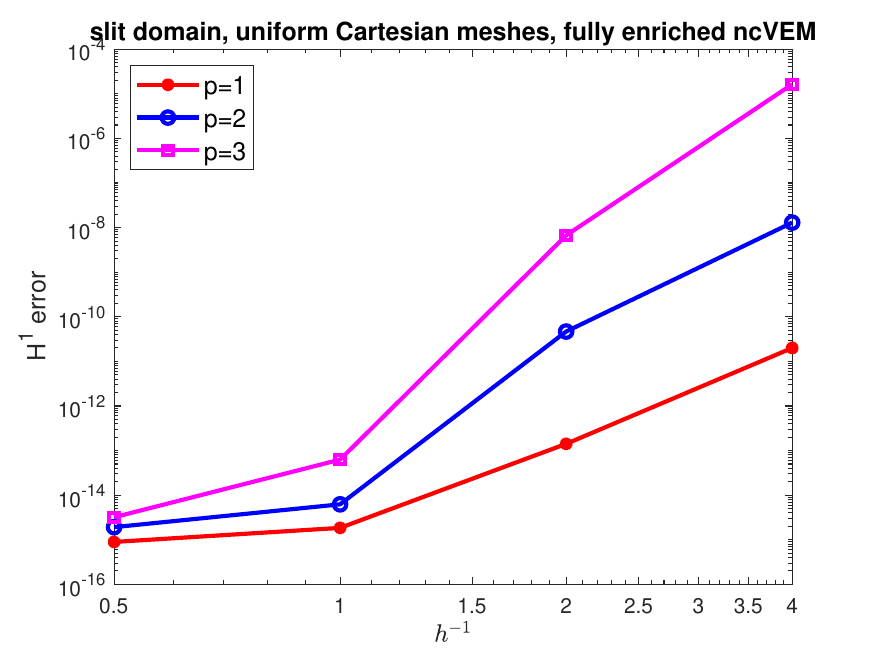}
\caption{$\h$-version of the fully enriched version of the method in order to approximate the solution~$u_2$ in~\eqref{u2}.
The polynomial order of accuracy of the method is~$\p=1$, $2$, and~$3$. We employ sequences of uniform Cartesian meshes as those in Figure~\ref{figure:meshes:L-shaped} (right).}
\label{figure:other1}
\end{figure}

From Figure~\ref{figure:other1}, we realize that the fully enriched method returns an error, which is zero up to machine precision as expected from the enriched consistency~\eqref{consistency}.
The growth of the error for this enriched patch test is an excellent indicator for the growth of the ill-conditioning of the system.
Finally, we have evidence that the VEM works also on elements with internal cracks:
the first element of the sequence of meshes is the heptagon with two overlapping edges depicted in Figure~\ref{figure:first-mesh}.

%%%%%%%
\subsection{Orthonormalization of the enriched edge polynomial basis functions} \label{subsection:NR-stab}
%%%%%%%
In the foregoing Sections~\ref{subsection:nr-L-shaped} and~\ref{subsection:nr-slit-cracks}, we observed that  the method suffer of ill-conditioning, notably employing fine meshes and the high order case.

In this section, we provide numerical evidence that such an ill-conditioning can be drastically reduced by changing the definition of the edge degrees of freedom.
We postpone to Appendix~\ref{appendix:implementation2} the design of the new degrees of freedom and the implementation details, and focus here on the comparison between the performance of the two methods.
We refer to such a modification as the \emph{orthonormal} enriched method, whereas that investigated so far goes under the name of \emph{standard} enriched method.
In few words, the former approach is based on orthonormalization of the basis of enriched edge polynomials to avoid situations as those described in Section~\ref{subsubsection:no-full-layerisation}.

For the extended patch test~$u_3$ in~\eqref{u3}, we consider both the $\h$- and $\p$-version of the  \emph{orthonormal} and \emph{standard} fully enriched methods.
In the former case, we employ sequences of uniform Cartesian meshes, in the latter a fixed Voronoi mesh;
see Figure~\ref{figure:orthoVSstandard} (left) and (right).
Since this is an extended patch test, the errors should be zero up to machine precision and their growth is the \emph{real} indicator of the ill-conditioning of the system.

\begin{figure}  [h]
\centering
\includegraphics [angle=0, width=0.45\textwidth]{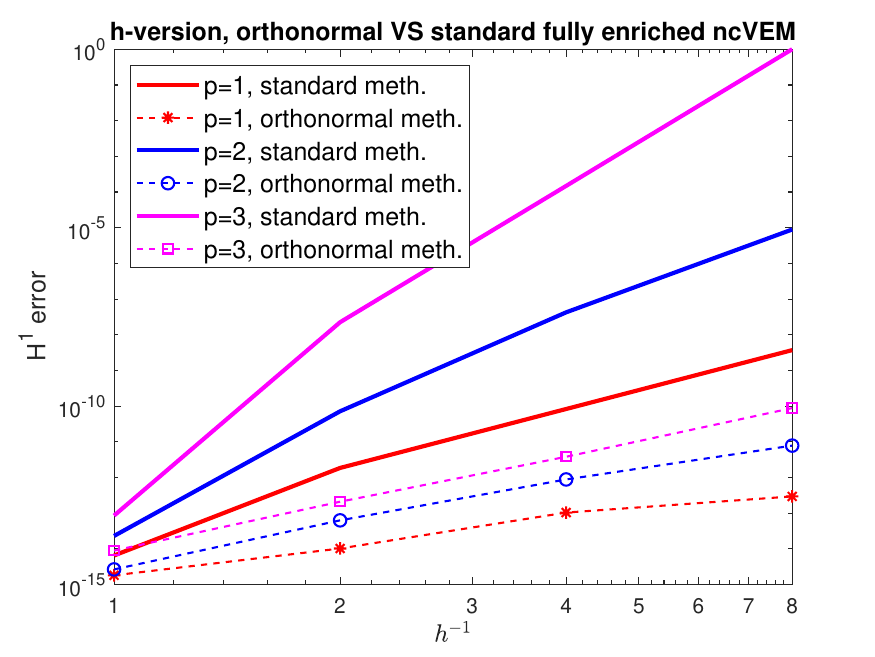}
\includegraphics [angle=0, width=0.45\textwidth]{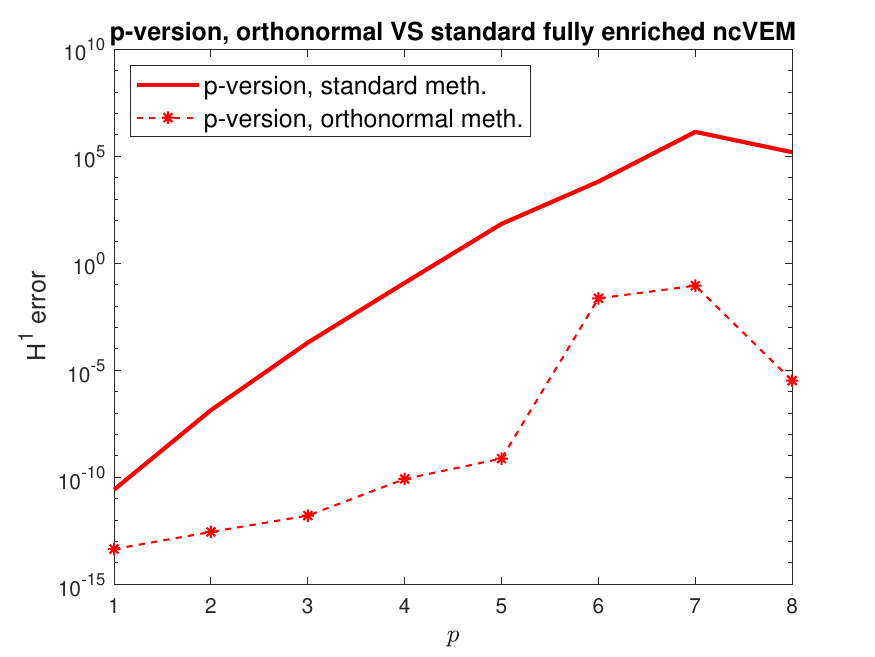}
\caption{Behaviour of the error for the solution~$u_3$ in~\eqref{u3}.
\emph{Left panel:} $\h$-version of the \emph{orthonormal} and \emph{standard} fully enriched method on sequenced of uniform Cartesian meshes.
\emph{Right panel:} $\p$-version of the \emph{orthonormal} and \emph{standard} fully enriched method on a fixed Voronoi mesh.}
\label{figure:orthoVSstandard}
\end{figure}

From Figure~\ref{figure:orthoVSstandard}, it is apparent that the \emph{orthonormal} enriched method drastically outperforms the \emph{standard} one.
For fine meshes and high degrees of accuracy, the former produces a relatively small error, which is even ten orders smaller than that produced by the latter.

For the sake of completeness, we also investigate the behaviour of the condition number employing the \emph{orthonormal} and the \emph{standard} enriched methods.
We report them in Table~\ref{table:condition-h} and~\ref{table:condition-p} for the $\h$- and $\p$-versions.

\begin{table}[H] 
\centering
\begin{tabular}{|c|c|c|c|c|c|c|}
\hline
 				& $\p=1$ - \text{std.}  		& $\p=1$ - \text{orth.}  			& $\p=2$ - \text{std.} 		& $\p=2$ - \text{orth.} 			& $\p=3$ - \text{std.}  		& $\p=3$ - \text{orth.}   \\
\hline
mesh~1 		& 8.07e+02						& 2.09e+01							& 3.32e+03						& 6.00e+01 							& 7.90e+03						& 1.57e+03    	\\
mesh~2 		&5.18e+05						& 2.74e+01							& 3.22e+07						& 8.27e+01  						& 6.73e+09   					& 2.03e+03		\\
mesh~3 		& 1.06e+08 						& 5.86e+01							& 5.15e+10						& 2.04e+02  						& 6.21e+14  					& 4.99e+03		\\
mesh~4 		& 2.09e+10						& 2.05e+02 							& 5.14e+13						& 7.58e+02 							& 9.01e+19  					& 1.84e+04		\\
\hline
\end{tabular}
\caption{Condition number of the system resulting from the $\h$-version of the \emph{orthonormal} and \emph{standard} fully enriched method on sequenced of uniform Cartesian meshes.}
\label{table:condition-h}
\end{table}

\begin{table}[H] 
\centering
\begin{tabular}{|c|c|c|}
\hline
 				& \text{standard}  				& \text{orthonormal}  		\\
\hline
$\p=1$ 		& 1.74e+07						& 1.41e+01	\\
$\p=2$ 		& 6.20e+11						& 1.01e+02	\\
$\p=3$ 		& 4.31e+15						& 9.39e+03	\\
$\p=4$ 		& 1.57e+19						& 7.53e+05	\\
$\p=5$ 		& 2.47e+24						& 4.90e+07	\\
$\p=6$ 		& 1.49e+25						& 7.00e+14	\\
$\p=7$ 		& 1.66e+28						& 2.72e+11	\\
$\p=8$ 		& 3.91e+30						& 7.63e+12	\\
\hline
\end{tabular}
\caption{Condition number of the system resulting from the $\p$-version of the \emph{orthonormal} and \emph{standard} fully enriched method on sequenced of a fixed Voronoi mesh.}
\label{table:condition-p}
\end{table}
Also in Tables~\ref{table:condition-h} and~\ref{table:condition-p}, the \emph{orthonormal} enriched method results in much smaller condition numbers.
The orthonormalization procedure detailed in Appendix~\ref{appendix:implementation2} below
is particularly effective in the nonconforming setting, and in general in the context of skeletal methods.
An analogous procedure in partition of unity based methods would result in orthonormal basis functions with increasing support due to the presence of the partition of unity functions.

Finally, in Figure~\ref{figure:p-version-orthoVSstandard}, we compare the  $\p$-version of the  \emph{orthonormal} and \emph{standard} fully enriched methods with the exact solution~$u_1$ in~\eqref{u1}
in terms of~$\p$ and the square root of the number of the degrees of freedom.
We employ sequences of uniform Cartesian meshes, in the latter a fixed Voronoi mesh.

\begin{figure}  [h]
\centering
\includegraphics [angle=0, width=0.45\textwidth]{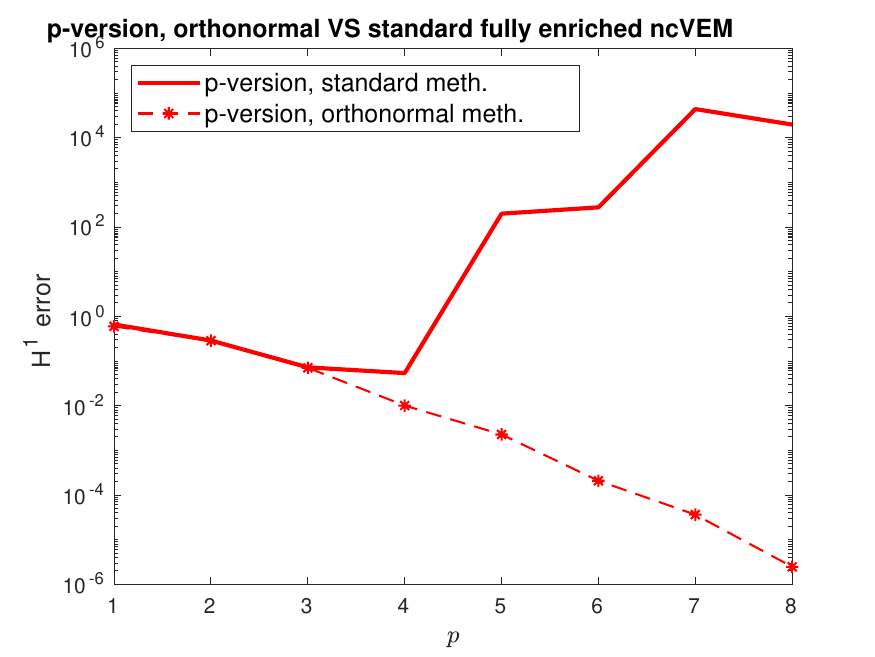}
\includegraphics [angle=0, width=0.45\textwidth]{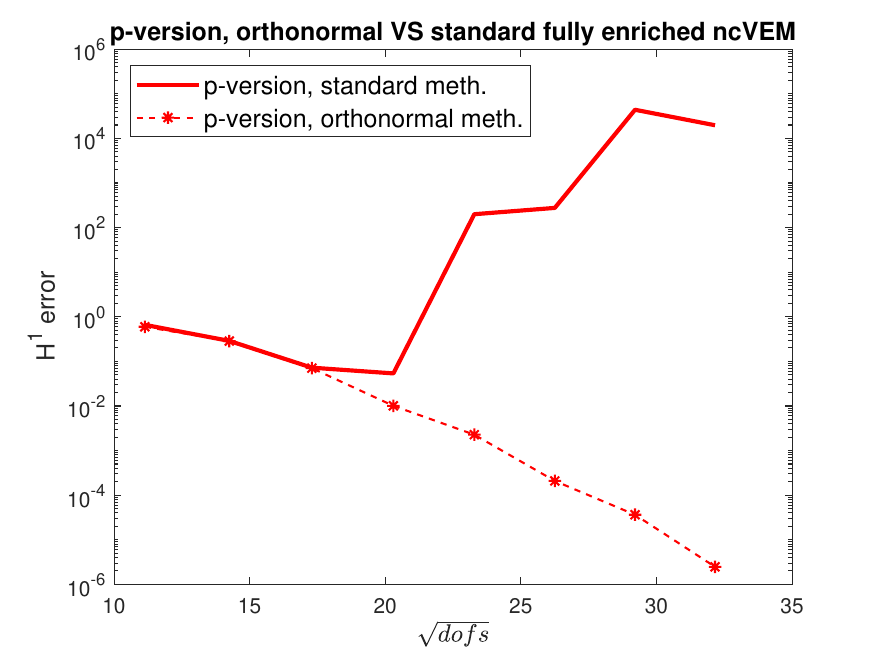}
\caption{Behaviour of the error for the solution~$u_1$ in~\eqref{u1}.
$\p$-version of the \emph{orthonormal} and \emph{standard} fully enriched method on a fixed Voronoi mesh.
\emph{Left panel}: error versus the polynomial degree of accuracy~$\p$.
\emph{Right panel}: error versus the square root of the number of degrees of freedom.}
\label{figure:p-version-orthoVSstandard}
\end{figure}

The orthonormalization allows us to recover exponential convergence of the error in terms of the degree of accuracy~$\p$.
The convergence is clearly exponential in terms of the square root of the number of degrees of freedom.

%%%%%%%%%%%%%%%%%%%%%%%%%%%%%%%%%%%%%%%%%%%%%%%%%%%%%%%%%%%%%%%%%%%%%%%%%%%
\section{Conclusions} \label{section:conclusions}
%%%%%%%%%%%%%%%%%%%%%%%%%%%%%%%%%%%%%%%%%%%%%%%%%%%%%%%%%%%%%%%%%%%%%%%%%%%
We introduced a novel enriched nonconforming virtual element method for the approximation of solutions to the Laplace problem on polygonal domain.
This has been done in the spirit of the extended Galerkin methods, but with a novel twist.
The nonenriched virtual element spaces are endowed with special singular functions arising from asymptotic singular expansions at the corners of the domain.
We analyzed the method and presented several numerical results, including the high-order version of the method, which validate the theoretical predictions.
In Appendix~\ref{appendix:implementation2} below, we discuss the implementation details. Importantly, both the theoretical and practical aspects are extensions of what is done in the nonenriched nonconforming virtual element methods.
\medskip

In future works, we plan to investigate some generalizations of this approach:
\begin{itemize}
\item the full analysis of a stabilization, which works regardless of the strength of the singularity;
\item a full analysis of the stabilizing term without resorting to inverse estimates in enriched polynomial spaces; this could be done analyzing errors as those in~\cite{yemm2021design},
i.e., stabilization dependent norms and errors involving the projected discrete solution;
\item multiple singularities;
\item the 3D version of the method;
\item enriched virtual elements for more general elliptic operators.
\end{itemize}
Notably, our approach seems to be applicable in several branches of computational mechanics including modeling of cracking phenomena, discontinuous media, and highly nonlinear complex materials behaviours.

%%%%%%%%%%%%%%%%%%%%%%%%%%%%%%%%%%%%%%%%%%%%%%%%%%%%%%%%%%%%%%%%%%%%%%%%%%%
\paragraph*{Acknowledgements}
%%%%%%%%%%%%%%%%%%%%%%%%%%%%%%%%%%%%%%%%%%%%%%%%%%%%%%%%%%%%%%%%%%%%%%%%%%%
We would like to thank the reviewers for their insightful comments and remarks.
L. Mascotto acknowledges the support of the Austrian Science Fund (FWF) through the project~P~$33477$.

%%%%%%%%%%%%%%%%%%%%%%%%%%%%%%%%%%%%%%%%%%%%%%%%%%%%%%%%%%%%%%%%%%%%%%%%%%%
\begin{appendices}

%%%%%%%%%%%%%%%%%%%%%%%%%%%%%%%%%%%%%%%%%%%%%%%%%%%%%%%%%%%%%%%%%%%%%%%%%%%
\section{\textnormal{Implementation details}} \label{appendix:implementation}
%%%%%%%%%%%%%%%%%%%%%%%%%%%%%%%%%%%%%%%%%%%%%%%%%%%%%%%%%%%%%%%%%%%%%%%%%%%
Here, we discuss the implementation details of the method. We employ the same notation as in~\cite{hitchhikersguideVEM}.
As in nonenriched nonconforming finite and virtual elements, the global stiffness matrix is obtained by assembling the local ones.
Therefore, we show the computation of the local stiffness matrices only.

We focus on the elements~$\E$ in the first layer~$\tauno$ only:
the local stiffness matrices on the elements~$\E \in \taunth$ are computed as in~\cite{nonconformingVEMbasic},
whereas it suffices to combine the tools employed for the other two layers on the elements~$\E \in \tauntw$.

Recall that we are assuming~\eqref{assumption:solution}, i.e., we enrich the approximation space with one singular function only.
The implementation details are utterly similar in the case of multiple singularities.

\medskip

We fix the following notation:
\[
\begin{split}
& \nPE = \dim (\mathbb P_\p(\E)), \quad \nPe=\dim(\mathbb P_{\p-1}(\e)), \\
& \ntildeE = \dim (\widetilde{\mathbb P}_\p(\E)), \quad \ntildee=\dim(\widetilde {\mathbb P}_{\p-1}(\e)), \quad \NtildeE = \dim(\VnE). \\
\end{split}
\]
Moreover, we set
\[
\widetilde {\mathbb P} _\p (\E) =\text{span}_{\alpha=1}^{\ntildeE} \{ \mtildealpha^\E  \} , \quad \quad \widetilde {\mathbb P}_{\p-1}(\e)  =\text{span}_{\alpha=1}^{\ntildee} \{ \mtildealpha^\e  \},
\]
where the functions~$\mtildealpha^\E$ and~$\mtildealpha^\e$ are defined in Section~\ref{section:NR}.

\medskip

Following~\cite{hitchhikersguideVEM}, the local matrix on element~$\E \in \taun$ is given by
\begin{equation} \label{matrix:form}
\mathbf A _n^\E = \Pibfstar{}^T \Gbf \Pibfstar + (\Ibf - \Pibf)^T \Sbf (\Ibf - \Pibf).
\end{equation}
We define the various matrices appearing in~\eqref{matrix:form}.
We begin with
\[
\Gbf _{\alpha,\beta} =
\begin{cases}
(\nabla \mtildebeta^\E, \nabla \mtildealpha^\E)_{0,\E} & \forall \alpha,\,\beta =2,\dots,\dim(\Pbbtilde_\p(\E))\\
0 & \text{otherwise}.\\
\end{cases}
\]
The matrix~$\Pibfstar$ is the matrix representation of the expansion of the projector~$\Pinabla$ in terms of the basis functions of~$\Pbbtilde_\p(\E)$:
\[
\Pibfstar = \Gbftilde^{-1} \Bbf,
\]
where
\begin{equation} \label{matrix:Gtilde}
\Gbftilde_{\alpha,\beta} =
\begin{cases}
\frac{1}{\vert \partial \E \vert} \int_{\partial \E} \mtildebeta^\E 	& \text{if } \alpha=1,\, \forall \beta=1,\dots,\dim(\Pbbtilde_\p(\E)) \\
(\nabla \mtildebeta^\E, \nabla \mtildealpha^\E)_{0,\E} 					& \text{otherwise}, \\
\end{cases}
\end{equation}
and
\begin{equation} \label{matrix:B}
\Bbf_{\alpha,i} =
\begin{cases}
\frac{1}{\vert \partial \E \vert} \int_{\partial \E} \varphi_i 					& \text{if } \alpha=1,\, \forall \beta=1,\dots,\dim(\Pbbtilde_\p(\E)) \\
(\nabla \varphi_i, \nabla \mtildealpha^\E) 										& \text{otherwise}. \\
\end{cases}
\end{equation}
The matrix~$\Pibf$ is the matrix representation of the expansion of the projector~$\Pinabla$ in terms of the basis functions of~$\VnE$:
\[
\Pibf = \Dbf \Pibfstar,
\]
where
\[
\Dbf_{i,\alpha} = \dof_i (\mtildealpha^\E) \quad \forall i =1,\dots, \dim(\VnE),\, \forall \alpha=1,\dots , \dim (\Pbbtilde_\p(\E)).
\]
Finally, $\Sbf$ is the matrix representation of the stabilization, i.e.,
\[
\Sbf_{i,j} = \SE_P(\varphi_j, \varphi _i) \quad \forall i,\,j =1,\dots, \dim(\VnE).
\]
If we employ the stabilization introduced in~\eqref{practical:stabilization}, then the matrix~$\Sbf$ is diagonal,
with entries given by the maximum between~$1$ and the corresponding diagonal entries of the consistency matrix~$ \Pibfstar{}^T \Gbf \Pibfstar$.
\medskip

We devote the remainder of this section to show how to compute the matrices~$\Gbf$, $\Gbftilde$, $\Bbf$, and~$\Dbf$.

%%%%%%%%%%%%%%%%%%%%%%%%%%%%%%%%%%%%%%%%%%%%%%%%%%%%%%%%%%%%%%%%%
\paragraph*{The matrices~$\Gbf$ and~$\Gbftilde$.}
It suffices to show how to compute the matrix~$\Gbftilde$ in~\eqref{matrix:Gtilde}.
Define
\[
\GbftildeA \in \mathbb R^{1\times \ntildeE},\quad \quad \GbftildeB \in \mathbb R^{(\ntildeE-1)\times \ntildeE},
\]
as follows. Begin with~$\GbftildeA$:
\[
\GbftildeA_{1,\beta} = \frac{1}{\vert \partial \E \vert} \int_{\partial \E} \mtildebeta ^\E                  \quad \forall \beta = 1, \dots , \ntildeE.
\]
As for~$\GbftildeB$, we set
\[
\GbftildeB_{\alpha,\beta} = \int_{\E} \nabla \mtildealpha^\E \cdot \nabla \mtildebeta^\E \quad \forall \alpha=2,\dots,\ntildeE,\quad \beta = 1, \dots , \ntildeE.
\]
Each entry of~$\GbftildeA$ and~$\GbftildeB$ can be approximated at any precision employing a sufficiently accurate quadrature formula.

The matrix~$\Gbftilde \in \mathbb R^{\ntildeE \times \ntildeE}$ is given by
\[
\begin{bmatrix}
\mbox{ $\GbftildeA$} \\ 
\mbox{$\GbftildeB$} \\
\end{bmatrix}.
\]

%%%%%
\paragraph*{The matrix~$\Bbf$.}
Denote the number of edges and vertices of~$\E$ by~$\NV$ and define
\[
\begin{split}
&\BbfA \in \mathbb R^{1 \times \NV  \p},\quad \quad \quad \quad \;\, \BbfB \in \mathbb R^{1 \times \NV},\quad \quad \quad \quad \;\, \BbfC \in \mathbb R^{1 \times \nPEmt},\\
&\BbfD \in \mathbb R^{(\nPE-1) \times \NV \p},\quad \quad  \BbfE \in \mathbb R^{(\nPE-1) \times \NV},\quad \quad  \BbfF \in \mathbb R^{(\nPE-1) \times \nPEmt},\\
&\BbfG \in \mathbb R^{1 \times \NV  \p},\quad \quad \quad \quad \;\, \BbfH \in \mathbb R^{1 \times \NV},\quad \quad \quad \quad \;\, \BbfI \in \mathbb R^{1 \times \nPEmt}.\\
\end{split}
\]
The matrix~$\Bbf \in \mathbb R^{\ntildeE \times \NtildeE}$ is given by
\[
\Bbf = 
\begin{bmatrix}
\mbox{ $\BbfA$} & \mbox{ $\BbfB$} & \mbox{ $\BbfC$} \\ 
\mbox{ $\BbfD$} & \mbox{ $\BbfE$} & \mbox{ $\BbfF$} \\ 
\mbox{ $\BbfG$} & \mbox{ $\BbfH$} & \mbox{ $\BbfI$} \\ 
\end{bmatrix}.
\]
In the matrix~$\Bbf$, the first column represents the contributions due to the basis elements associated with the nonenriched edge polynomials;
the second with special boundary functions; the third with the nonenriched bulk polynomials.
On the other hand, the first row represents the zero average constraint;
the second the contributions of the nonenriched bulk polynomials;
the third the contributions of the singular bulk function.

Owing to~\eqref{matrix:B} and the definition of the edge degrees of freedom in~\eqref{edge:moments}, we set
\[
\begin{split}
& \BbfB_{1,i} =0 \quad \forall i=1,\dots,\NV,\quad \quad  \BbfC_{1,i} =0 \quad \forall i=1,\dots, \nPEmt.
\end{split}
\]
Denote the $j$-th edge in the local ordering of~$\E$ by~$\e(j)$.
Vector~$\BbfA$ has the entries equal to~$\frac{\vert \e(j) \vert}{\vert \partial \E \vert}$ in the~$j\,\p$-th column, for~$j=1,\dots,\NV$.
Otherwise, it has zero entries.

Next, observe that an integration by parts yields
\begin{equation} \label{an:integration-by-parts}
(\nabla \mtildealpha^\E, \nabla \varphi_i)_{0,\E} = -(\Delta \mtildealpha^\E, \varphi_i)_{0,\E} + \sum_{\e\in\EE}(\nE \cdot \nabla \mtildealpha^\E, \varphi_i)_{0, \e}.
\end{equation}
If~$\varphi_i$ is an edge basis function, then the first term on the right-hand side of~\eqref{an:integration-by-parts} vanishes.
For all~$\e\in \EE$, if~$\mtildealpha^\E$ is a monomial, then we expand~$\nE \cdot \nabla \mtildealpha^\E{}_{|\e}$ into a linear combination of scaled Legendre polynomials defined in~\eqref{Legendre}:
\begin{equation} \label{expansion:Neumann}
\nE \cdot \nabla \mtildealpha^\E {}_{|\e}= \sum_{\beta=0}^{\p-1} \lambda_\beta \Lbbe_\beta.
\end{equation}
We identify the coefficients~$\lambda_\alpha$ in expansion~\eqref{expansion:Neumann} as follows: test~\eqref{expansion:Neumann} with any scaled Legendre polynomial of degree at most~$\p-1$ and use the orthogonality property
\begin{equation} \label{Legendre:ortho}
(\Lbbe_\alpha, \Lbbe_\beta)_{0,\e} = \frac{\he}{2\beta+1} \delta_{\alpha,\beta} \quad  \forall \alpha,\, \beta=0,\dots,\p-1,
\end{equation}
to get
\[
\lambda_\beta  = \frac{2\beta+1}{\he} (\nE \cdot \nabla \mtildealpha^\E, \Lbbe_\beta)_{0,\e}  \quad \forall \beta=0,\dots,\p-1.
\]
The integral on the right-hand side is computable exactly. Thus, expansion~\eqref{expansion:Neumann} becomes
\[
\nE \cdot \nabla \mtildealpha^\E {}_{|\e}= \sum_{\beta=0}^{\p-1} \frac{2 \beta+1}{\he} (\nE \cdot \nabla \mtildealpha^\E, \Lbbe_\beta)_{0,\e}  \Lbbe_\beta.
\]
Given~$\e(i)$ the edge where the edge basis function~$\varphi_i$ has a nonzero edge moment, we write
\[
\begin{split}
(\nabla \mtildealpha^\E, \nabla \varphi_i)_{0,\E} 	& = (\nE \cdot \nabla \mtildealpha^\E {}_{|\e(i)},\varphi_i)_{0,\e(i)} \\
																& = \sum_{\beta=0}^{\p-1} \frac{2\beta+1}{\hei}  (\nE \cdot \nabla \mtildealpha^\E {}_{|\e(i)}, \Lbbei_{\beta})_{0,\e(i)} (\Lbbei_\beta, \varphi_i)_{0,\e(i)}.
\end{split}
\]
Let~$j_{\e(i)}$ denote the numbering of~$\varphi_i$ as a basis function on edge~$\e(i).$
Using the definition of the edge degrees of freedom~\eqref{edge:moments}, we set
\[
\begin{split}
& \BbfD_{\alpha,i} = (2(j_{\e(i)}-1)+1) (\nE \cdot \nabla \mtildealpha^\E {}_{|\e(i)}, \Lbbei_{i})_{0,\e(i)}  \quad \forall \alpha=2, \dots, \nPE-1,\, \forall i=1,\dots,\p\NV,\\
& \BbfE_{\alpha,i}  = 0 \quad \forall \alpha=2, \dots, \nPE-1,\, \forall i=1,\dots, \NV.\\
\end{split}
\]
Next, consider the case of~$\varphi_i$ being an edge basis function and~$\mtildealpha^\E$ being the special function~$\SAE$ defined in~\eqref{SAE}.
From~\eqref{relation:special} and~\eqref{an:integration-by-parts}, we deduce
\[
(\nabla \mtildealpha^\E, \nabla \varphi_i)_{0,\E} = (\nE \cdot \nabla \SAE , \varphi_i)_{0,\e(i)} = \nE \cdot \nei \left( \frac{\hei}{\hE}  \right)^\alpha (\nei \cdot \nabla\SAei,\varphi_i)_{0,\ei}  .
\]
Using the definition of the edge degrees of freedom~\eqref{edge:moments}, we set
\[
\begin{split}
& \BbfG_{1,i}=0 \quad \forall i=1,\dots, \NV \p,\quad \quad \BbfH_{1,i}=\nE \cdot \nei \left( \frac{\hei}{\hE}  \right)^\alpha  \quad \forall i=1,\dots, \NV.\\
\end{split}
\]
Finally, focus on the case given by~$\varphi_i$ being a bulk basis function.
Firstly, assume that~$\mtildealpha^\E=\malpha^\E$ is a nonenriched polynomial.
Given~$(x_\E,y_\E)$ the centroid of~$\E$, the following splitting is valid:
\small{\[
\begin{split}
\Delta \malpha^\E 	& = \Delta \left( \left (\frac{x-x_\E}{\hE}\right)^{\alpha_1} \left (\frac{y-y_\E}{\hE}\right)^{\alpha_2} \right) \\
						& = \frac{1}{\hE^2} \left(  \alpha_1(\alpha_1-1) \left (\frac{x-x_\E}{\hE}\right)^{\alpha_1-2} \left (\frac{y-y_\E}{\hE}\right)^{\alpha_2}   + \alpha_2(\alpha_2-1)  \left (\frac{x-x_\E}{\hE}\right)^{\alpha_1} \left (\frac{y-y_\E}{\hE}\right)^{\alpha_2-2}     \right).
\end{split}
\]}\normalsize{}
Let~$\alpha$ be associated to~$(\alpha_1,\alpha_2)$ via bijection~\eqref{natural:bijection}.
Whenever it makes sense, set~$\alphatilde_1$ and~$\alphatilde_2$ the natural numbers associated with~$(\alpha_1-2,\alpha_2)$ and~$(\alpha_1-2, \alpha_2)$, via the same bijection~\eqref{natural:bijection}.

We have
\[
\begin{split}
& (\nabla \malpha^\E, \nabla \varphi_i)_{0,\E} = -(\Delta \malpha^\E, \varphi_i)_{0,\E}  \\
& = \begin{cases}
0 & \text{if } \alpha=1,2,3\\
-\frac{\vert \E \vert}{\hE^2} \left( (\alpha_1-1)\alpha_1 \frac{1}{\vert \E \vert}(m_{\alphatilde_1}^\E,\varphi_i)_{0,\E} + (\alpha_2-1)\alpha_2\frac{1}{\vert \E \vert} (m_{\alphatilde_2}^\E,\varphi_i)_{0,\E}    \right) & \text{otherwise}.
\end{cases}
\end{split}
\]
In other words, we get
\[
\begin{split}
& \BbfF_{\alpha,i} =    
\begin{cases}
0  & \text{if } \alpha=1,2,3, \; \forall i=2,\dots, \nPEpmt\\
-\frac{\vert \E \vert}{\hE^2} \left( (\alpha_1-1)\alpha_1 \delta_{\alphatilde_1,i} + (\alpha_2-1)\alpha_2 \delta_{\alphatilde_2,i} \right) & \forall \alpha=4,\dots,\nPE,\,\forall i=1,\dots, \nPEpmt ,\\
\end{cases}\\
& \BbfI_{1,i} = 0 \quad \forall i=1,\dots, \nPEpmt. \\
\end{split}
\]

%%%%%
\paragraph*{The matrix~$\Dbf$.}
We introduce
\[
\DbfA \in \mathbb R^{\p\NV \times \ntildeE},\quad \quad  \DbfB \in \mathbb R^{\NV \times \ntildeE},\quad \quad  \DbfC \in \mathbb R^{\nPEpmt \times \ntildeE},
\]
so that the matrix~$\Dbf \in \mathbb R^{\NtildeE \times \ntildeE}$ is given by
\[
\begin{bmatrix}
\mbox{ $\DbfA$} \\ 
\mbox{ $\DbfB$} \\ 
\mbox{ $\DbfC$} \\ 
\end{bmatrix}.
\]
The matrix~$\DbfA$ represents the contributions of the nonenriched edge basis functions; $\DbfB$ the contributions of the special edge functions; $\DbfC$ the contributions of the bulk functions.

Given~$\varphi_i$ an edge basis function, let~$\ei$ be the edge, where~$\varphi_i$ has a nonzero moment.
If~$\varphi_i$ is associated with the nonenriched polynomial moments, then denote the nonzero order moment by~$\beta(i)$.
Recalling the definition of the edge moments~\eqref{edge:moments}, we set
\[
\begin{split}
& \DbfA_{i,\alpha}  = \frac{1}{\vert \ei \vert} (\mathbb L_{\beta(i)}^{\ei} , \mtildealpha^\E{}_{|\ei})_{0,\ei}      \quad \forall i=1,\dots, \p \NV,\, \forall \alpha=1,\dots,\ntildeE, \\
& \DbfB_{i,\alpha}  = (\ne  \cdot \nabla \S_\Abf^{\ei} , \mtildealpha^\E{}_{|\ei}) _{0,\ei}         \quad \forall i=1,\dots, \NV,\, \forall \alpha=1,\dots,\ntildeE. \\
\end{split}
\]
As for the matrix~$\DbfC$, we simply write
\begin{equation} \label{DC}
\DbfC_{i,\alpha} = \frac{1}{\vert \E \vert} \int_\E m^\E_i \mtildealpha^\E       \quad \forall i=1,\dots, \nPEpmt ,\, \forall \alpha=1,\dots,\ntildeE.
\end{equation}
All the entries of the three matrices above can be computed exactly or approximated at any precision with a sufficiently accurate quadrature formula.

\begin{remark} \label{remark:Laplace-stuff}
As for the computation of~$\DbfC$ in~\eqref{DC} in the case~$\mtildealpha^\E = \SAE$, we suggest to use the following strategy. 
Given~$m^\E_i \in \mathbb P_{\p-2}(\E)$, it is possible to write
\[
m^\E_i  = \Delta m^\E_{\p,i}
\]
for some~$m^\E_{\p,i} \in \mathbb P_\p(\E)$. We provide an explicit representation of~$m^\E_{\p,i}$ in Appendix~\ref{appendix:Laplacian-polynomial}.

Then, the integral in~\eqref{DC} can be rewritten using an integration by parts twice and the fact that~$\Delta \SA=0$ as
\[
\frac{1}{\vert \E \vert} \int_\E m^\E_{i} \SAE = \frac{1}{\vert \E \vert} \int_\E \Delta m^\E_{\p,i} \SAE = \frac{1}{\vert \E \vert} \left[  \int_{\partial \E} \n \cdot \nabla m^\E_{\p,i} \SAE - \int_{\partial \E} m^\E_{\p,i} \n \cdot \nabla \SAE    \right].
\]
\eremk
\end{remark}

\begin{remark}\label{remark:Jacobi}
In the computation of the matrix~$\Gbf^B$, if at least one of the two terms, say~$\mtildealpha^\E$, is the singular enrichment function~$\SAE$, we reduce the computation of the bulk integral to the computation of the boundary integral
\[
\int_\E \nabla \mtildealpha^\E \cdot\nabla \mtildebeta^\E = \int_\E \nabla \SAE \cdot\nabla \mtildebeta^\E = \int_{\partial \E} \n \cdot \nabla \SAE \, \mtildebeta^\E.
\]
In the light of this fact,
in the computation of the matrices~$\Bbf$, $\Dbf$, $\Gbf$, and the boundary conditions, the integrals involving singular functions are always boundary integrals. This fact is extremely relevant.
Indeed, in order to compute integrals involving singular functions, we resort to Gau\ss-Jacobi quadrature formulas; see, e.g., \cite[Section~4.8-1]{Rabinowitz2001first}.
By doing so, the singular integrals can be computed up to machine precision with relatively few quadrature knots,
whereas, in order to achieve the same precision with the standard Gau\ss{} integration rule, we would need to require a disproportionate number of quadrature knots.
\eremk
\end{remark}

\begin{remark} \label{GBD}
The ``$G=BD$'' test of~\cite[Remark~3.3]{hitchhikersguideVEM} is valid also in the enriched framework. This is an excellent test to check the correctness of the implementation of the method.
In order to fulfil this test correct, the integrals must be computed up to machine precision.
Notably, we suggest to use suitable quadrature formulas; see Remark~\ref{remark:Jacobi}.
\eremk
\end{remark}

\begin{remark} \label{remark:stress}
In view of possible extensions to linear elasticity, it might be of interest to discuss the approximation of the gradient of the discrete solution in the elements and on faces.
In the bulk of the elements, we can consider~$\nabla \Pinabla \un$, whereas, on an edge~$\e$, we can consider, e.g., the average of the energy projection on the neighbouring elements~$\E^+$ and~$\E^-$:
\[
\frac{1}{2} (\nabla \Pinabla \un{}_{|\E^+} \cdot \n_{\E^+} + \nabla \Pinabla \un{}_{|\E^-} \cdot \n_{\E^-}).
\]
\end{remark}

%%%%%
\paragraph*{Computation of the right-hand side.} Proceed as in~\cite{hitchhikersguideVEM}: no enrichment affects the right-hand side.

%%%%%
\subsection{Nonhomogeneous Dirichlet boundary conditions}
%%%%%
As for the treatment of nonhomogenous Dirichlet boundary conditions,
we identify the boundary edge degrees of freedom of the discrete and exact solutions~$\un$ and~$u$. In other words, for all~$\e \in \EnB$ such that~$\e \subset \GammaD$, we impose the following condition:
\[
\int_\e (\un - u) \mtildealphae =0  \quad \quad \forall \alpha=1, \dots, \dim (\widetilde {\mathbb P}_{\p-1} (\e)).
\]

%%%%%
\subsection{Nonhomogeneous Neumann boundary conditions} \label{subsection:nh-Neumann}
%%%%%
Here, we address the implementation aspects for the computation of the Neumann boundary conditions term~\eqref{Neumann:approximation}.
In particular, given a canonical basis function~$\varphi _ i$ associated with a nonzero moment on a Neumann edge~$\e \subset \GammaN$,
we describe how to compute
\begin{equation} \label{Neumann:bis}
\int_\e \gN \Pize \varphi _i .
\end{equation}
As highlighted in Remark~\ref{remark:computability-edge-proj}, for all~$\e\in \En$, the projector~$\Pize$ can be computed
only under assumption~$\SA \in H^{\frac{3}{2} +\varepsilon}(\E)$ with~$\varepsilon>0$, where~$\E \in\taun$ is such that~$\e \in \EE$.
At the end of this section, we show that, under suitable assumptions on~$\gN$, we can indeed compute nonhomogenous Neumann boundary conditions for~$\SA \not \in H^{\frac{3}{2}+\varepsilon}(\E)$, $\varepsilon>0$,
as well, with no need whatsoever of resorting to the projector~$\Pize$.

In order to compute enriched edge projections of the basis functions, consider the expansion
\begin{equation} \label{coefficients:Neumann}
\Pize \varphi_i = \sum_{\alpha = 0} ^{\p-1} \lambdaz_{\alpha} \malphae + \lambdaz _\p \ne \cdot \nabla \SAe.
\end{equation}
Once we know the coefficients~$\lambdaz_\alpha$, $\alpha=0,\dots,\p$, we are able to approximate the integral in~\eqref{Neumann:bis} at any precision.

Define the matrix~$\Gbfze \in \mathbb R^{(\p+1)\times(\p+1)}$ and vector~$\bbfze \in \mathbb R^{\p+1,1}$ as follows:
\[
\Gbfze_{\alpha, \beta} :=
\begin{cases}
\frac{\he}{2} \frac{2}{2(\beta-1)+1}=\frac{\he}{2(\beta-1)+1}	& \text{if }  \alpha=\beta, \; \beta=1,\dots,\p\\
0                                                     							& \text{if }  \alpha \text{ is not equal to } \beta,\; \alpha,\beta=1,\dots,\p\\
(\ne \cdot \nabla \SAe, \mbetae)_{0,\e}                        			& \text{if } \alpha=\p+1 \text{ and } \beta=1,\dots,\p;\;  \beta=\p+1 \text{ and } \alpha=1,\dots,\p \\
(\ne \cdot \nabla  \SAe, \ne \cdot \nabla  \SAe) _{0,\e} 		& \text{if } \alpha=\beta=\p+1.
\end{cases}
\]
Moreover, define vector~$\bbfze \in \mathbb R^{\p+1,1}$ as the $i$-th column of the diagonal matrix~$\Bbfze \in \mathbb R^{(\p+1)\times(\p+1)}$, which is given by
\[
\Bbfze_{\alpha,i} = 
\begin{cases}
\he 	& \text{if } \alpha=i \le \p \\
1 		& \text{if } \alpha=i = \p+1. \\
\end{cases}
\]
The matrix~$\Lambdabfz$ of the coefficients in~\eqref{coefficients:Neumann} for the expansion of the basis function element~$\varphi$ is computed solving the system
\[
\Gbfze \Lambdabfz = \bbfze.
\]
In order to see this, it suffices to test~\eqref{coefficients:Neumann} with the elements in a basis of~$\widetilde{\mathbb P}_{\p-1}(\e)$ and use the orthogonality property of the Legendre polynomials~\eqref{Legendre:ortho}.

The computation of nonhomogenous Neumann boundary conditions can be simplified and extended to the case of general singular functions~$\SAE \not \in H^{\frac{3}{2}+\varepsilon}  (\Omega)$ with~$\varepsilon >0$.
In particular, assume that, given~$c \in \mathbb R$,
\begin{equation} \label{particular-Neumann}
\gN = \nOmega \cdot \nabla u = c\, \nOmega \cdot \nabla \SA.
\end{equation}
More generally, we can assume that~$\gN{}_{|\e} \in \widetilde{\mathbb P}_{\p-1}(\e)$ for all~$\e \subset \GammaN$.
We employ the following discretization of the Neumann datum contribution:
\[
\sum_{\e \in \EnB, \, \e \subset \GammaN} \int_\e \gN \varphi_i.
\]
Fix~$\e \in \EnB$ with~$\e \subset \GammaN$ and let~$\varphi_i$ be a basis function, with a nonzero moment on edge~$\e$.
Denote the power of the singularity of~$\SA$ by~$\alpha$.
Thanks to assumption~\eqref{particular-Neumann}, we can write
\[
\begin{split}
\int_\e \gN \varphi_i 	& = c \int_\e \nOmega \cdot \nabla \SA \, \varphi_i = c \he^{\alpha}  \int_\e \nOmega \cdot \nabla \SAe \, \varphi_i = c \he^{\alpha} \int_\e \nOmega \cdot \nabla \SAe \, \varphi_i.
\end{split}
\]
Due to the definition of the enriched edge degrees of freedom~\eqref{edge:moments},
this quantity is equal to zero for all basis functions dual to the scaled Legendre polynomials, whereas it is equal to~$\he^{\alpha}$ if~$\varphi_i$ is dual to the singular edge function.

%%%%%%%%%%%%%%%%%%%%%%%%%%%%%%%%%%%%%%%%%%%%%%%%%%%%%%%%%%%%%%%%%%%%%%%%%%%
\section{\textnormal{Design and implementation of a robust variant of~\eqref{ncEVEM}}} \label{appendix:implementation2}
%%%%%%%%%%%%%%%%%%%%%%%%%%%%%%%%%%%%%%%%%%%%%%%%%%%%%%%%%%%%%%%%%%%%%%%%%%%
In order to mitigate the ill-conditioning observed in Section~\ref{section:NR},
we discuss method~\eqref{ncEVEM} changing the definition of the edge degrees of freedom on enriched edges;
see Section~\ref{subsection:new-method-design}.
The numerical results with this version of the method are provided in Section~\ref{subsubsection:NR-stab}.
We also provide some implementation details for such a version of the method; see Section~\ref{subsection:new-method-implementation}.

Importantly, the orthonormalization procedure is based on a ``diagonalization'' process of the matrix~$\Gbfz$ in Appendix~\ref{subsection:nh-Neumann}.
Thence, we need to assume that the singular function~$\SA \in H^{\frac{3}{2} + \varepsilon}(\Omega)$ with~$\varepsilon >0$.

To the best of our understanding, this procedure is not possible to use in other contexts, e.g., in the setting of partition of unity methods.
Applying there any orthonormalization whatsoever, in fact, would result in a dramatic loss of localization of the basis functions.

%%%%%%%%%%%%%%%%%%%%%%%%%%%%%%%
\subsection{\textnormal{Orthonormalization of enriched edge polynomials}} \label{subsection:new-method-design}
%%%%%%%%%%%%%%%%%%%%%%%%%%%%%%%
Consider method~\eqref{ncEVEM} with a modified definition of the degrees of freedom on the enriched edges~$\e \in \Eno$.
More precisely, recall that $\{\mtildealphae\}_{\alpha=0}^{\p-1}$ denote the basis of the enriched edge polynomial space~$\widetilde{\mathbb P}_{\p-1}(\e)$ consisting of the first~$\p$ scaled Legendre polynomials and~$\n \cdot \nabla \SAe$.
The degrees of freedom on enriched edges~\eqref{edge:moments} have been defined with respect to such a basis.

As discussed in Section~\ref{subsubsection:no-full-layerisation}, for small elements and high polynomial degrees, those basis elements become close to linear dependent.
Therefore, we $L^2(\e)$-orthonormalize the elements~$\mtildealphae$.
For instance, we can use a stable Gram-Schmidt orthonormalization as that presented in~\cite[Section~2]{BassiBottiColomboDipietroTesini}.
Denote the new $L^2(\e)$-orthonormal basis elements by~$\{\mbaralphae\}_{\alpha=0}^{\p-1}$.

The modified method is based on the same local and global virtual element spaces, the same bulk degrees of freedom, the same edge degrees of freedom on nonenriched edges~$\e \in \Entw$,
and edge degrees of freedom on enriched edges~$\e \in \Eno$ with respect to the new basis of~$\widetilde{\mathbb P}_{\p-1}(\e)$.
Due to the normalization of the basis functions, the scaling of the edge degrees of freedom is given by
\[
\frac{1}{\he^{\frac{1}{2}}} \int_\e \vn \mtildealphae.
\]
Clearly, the analysis of the method is the same as for the original one.

%%%%%%%%%%%%%%%%%%%%%%%%%%%%%%%
\subsection{\textnormal{Implementation details}} \label{subsection:new-method-implementation}
%%%%%%%%%%%%%%%%%%%%%%%%%%%%%%%
Here, we provide some implementation details for the new setting discussed in Section~\ref{subsection:new-method-design}.
In particular, we explain how to compute the various matrices needed in the implementation of the method.
We denote the matrices computed with the new method adding a bar on top of their counterparts in Appendix~\ref{appendix:implementation}.
Note that we only modify the matrices associated with enriched edges and elements.
Moreover, for the sake of conciseness, we avoid to discuss the details of the matrices associated with the elements~$\E \in \tauntw$ and rather focus on those associated with~$\E \in \tauno$.
\medskip

Fix~$\e \in \Eno$ and let~$\GSbf^\e \in \mathbb R^{(\p+1) \times (\p+1)}$ be the lower triangular matrix containing the coefficients obtained via the orthonormalization such that
\[
\mbaralphae = \sum_{\beta=1}^{\alpha} \GSbf^\e_{\alpha,\beta} \mtildebetae    \quad \quad \forall \alpha=0,\dots,\p-1.
\]
The matrix~$\GSbf^\e$ can be computed, e.g., as in~\cite[Section~2]{BassiBottiColomboDipietroTesini}. 
\begin{remark} \label{remark:structure-GS}
The structure of the Gram-Schmidt orthonormalizing matrix is partially known a priori.
As already mentioned, $\GSbf^\e$ is lower triangular. Moreover, the upper-left block in $\mathbb R^{\p \times \p}$ is diagonal, with diagonal entry given by~$\Vert \malphae \Vert_{0,\e}^{-1}$ for all~$\alpha=1,\dots,\p$.
This follows from the fact that the scaled Legendre basis we employ is already orthogonal, albeit it needs to be normalized.
Thus, the only ``full'' row of~$\GSbf^\e$ is the last one, because the singular function has no orthogonality property whatsoever with respect to the scaled Legendre polynomials.
Knowing a priori the structure of~$\GSbf^\e$ is of extreme help in the computation of the local matrices below.
\eremk
\end{remark}
In what follows, we employ the same notation as that in Appendix~\ref{appendix:implementation}.

%%%
\paragraph*{The matrix~$\Gbarbfze$.}
%%%
For all~$\e \in \Eno$, the matrix~$\Gbarbfze \in \mathbb R^{(\p+1) \times (\p+1)}$ is defined as the identity matrix.
This follows from the orthonormality of the elements in the new basis of the enriched edge polynomial spaces.

%%%
\paragraph*{The matrix~$\Bbarbfze$.}
%%%
For all~$\e \in \Eno$, the matrix~$\Bbarbfze \in \mathbb R^{(\p+1) \times (\p+1)}$ is defined as the diagonal matrix with diagonal entries given by~$\he^{\frac{1}{2}}$.
This follows from the fact that the elements of the new basis of the enriched edge polynomial spaces are normalized in~$L^2(\e)$ and therefore the moments need the proper scaling.

%%%
\paragraph*{The matrix~$\Gbarbf$.}
%%%
The matrix~$\Gbarbf \in \mathbb R^{\ntildeE \times \ntildeE}$ is equal to the matrix~$\Gbf$.
In fact, it involves only the product of basis elements of the enriched bulk polynomial space, which has not been modified in the new setting.

%%%
\paragraph*{The matrix~$\Bbarbf$.}
%%%
The matrix~$\Bbarbf \in \mathbb R^{\ntildeE \times \NtildeE}$ can be split into the four submatrices
\[
\begin{split}
&\BbarbfAB \in \mathbb R^{1 \times \NV  (\p+1)},\quad \quad \quad \quad \; \BbarbfC \in \mathbb R^{1 \times \nPEmt},\\
&\BbarbfDEGH \in \mathbb R^{\nPE \times \NV (\p+1)},\quad \quad  \BbarbfFI \in \mathbb R^{\nPE \times \nPEmt}.\\
\end{split}
\]
In particular, we write
\[
\Bbarbf = 
\begin{bmatrix}
\mbox{ $\BbarbfAB$} & \mbox{ $\BbarbfC$} \\ 
\mbox{ $\BbarbfDEGH$} & \mbox{ $\BbarbfFI$} \\ 
\end{bmatrix}.
\]
Since there is no modification of the bulk degrees of freedom, we have
\[
\BbarbfC = \BbfC, \quad\quad
\BbarbfFI = 
\begin{bmatrix}
\mbox{ $\BbfF$} \\ 
\mbox{ $\BbfI$} \\
\end{bmatrix}.
\]
Next, focus on the matrix~$\BbarbfAB$. Thanks to Remark~\ref{remark:structure-GS} and the fact that~$\mtilde_1^\e=1$, we have
\[
\frac{1}{\vert \partial \E \vert} \int_{\partial \E} \phibar_i = \frac{(\GSbf^\e_{1,1})^{-1}}{\vert \partial \E \vert} \int_{\partial \E} \GSbf^\e_{1,1} \phibar_i = \frac{\he^{\frac{1}{2}}}{\GSbf^\e_{1,1}\, \vert \partial \E \vert} \, \frac{1}{\he^{\frac{1}{2}}} \int_{\partial \E} \mbare_1\, \phibar_i.
\]
We deduce that~$\BbarbfAB$ has the entries equal to~$\he^{\frac{1}{2}} / (\GSbf^\e_{1,1} \, \vert \partial \E \vert )$ in the $j\, \p$-th column for all~$j=1,\dots,\NV$, where~$\e(j)$ denotes the $j$-th edge in the local ordering of~$\E$.

Eventually, we deal with the matrix~$\BbarbfDEGH$,
i.e., on the case of~$\mtildealphaE$ being a bulk enriched basis function and~$\varphi_i$ a basis function of the virtual element space, which is dual to an edge moment.
We write
\begin{equation} \label{ortho-IBPs}
\int_{\partial \E} \nabla \mtildealphaE \cdot \nabla \phibar_i = \sum_{\e \in \EE} \int_{\e} (\nE \cdot \nabla \mtildealphaE)_{|\e} \phibar_{i}.
\end{equation}
For all~$\e \in \EE$, we need to expand~$(\nE \cdot \nabla \mtildealphaE)_{|\e}$ into the orthonormalized basis elements:
\[
(\nE \cdot \nabla \mtildealphaE)_{|\e} = \sum_{\beta = 1} ^{\p+1} \lambdabar_\beta \mbarbetae.
\]
Testing the above identity with~$\mbargammae$, for all~$\gamma=1,\dots, \p+1$, gives
\[
(\nE \cdot \nabla \mtildealphaE, \mbargammae)_{0,\e} = \lambdabar_\gamma .
\]
Thus, the following decomposition is valid:
\[
(\nE \cdot \nabla \mtildealphaE)_{|\e} = \sum_{\beta = 1} ^{\p+1} (\nE \cdot \nabla \mtildealphaE, \mbarbetae)_{0,\e} \mbarbetae.
\]
Inserting this into~\eqref{ortho-IBPs} and using the definition of the new degrees of freedom yield
\[
\int_{\partial \E} \nabla \mtildealphaE \cdot \nabla \phibar _i = \sum_{\e \in \EE} \sum_{\beta=1}^{p+1} \he^{\frac{1}{2}} (\nE \cdot \nabla \mtildealphaE, \mbarbetae)_{0,\e} \frac{1}{\he^{\frac{1}{2}}} \int_\e \mbarbetae \phibar_i.
\]
Let~$j_{\e(i)}$ denote the numbering of~$\varphi_i$ as a basis function on edge~$\e(i)$.
Then, we can write, thanks to the definition of the enriched edge degrees of freedom~\eqref{edge:moments},
\[
\begin{split}
\BbarbfDEGH _{\alpha,i } = \int_{\partial \E} \nabla \mtildealphaE \cdot \nabla \phibar _i 	& = \sum_{\beta=1}^{\p+1}  \h_{\e(i)}^{\frac{1}{2}} (\nE \cdot \nabla \mtildealphaE, \mbar_{j_{\e(i)}}^{\e(i)})_{0,\e(i)} \frac{1}{\h_{\e(i)}^{\frac{1}{2}}} \int_{\e(i)} \mbar_{j_{\e(i)}}^{\e(i)} \phibar_i \\
																	& = \h_{\e(i)}^{\frac{1}{2}} (\nE \cdot \nabla \mtildealphaE, \mbar_{j_{\e(i)}}^{\e(i)})_{0,\e(i)}  .\\
\end{split}
\]

%%%
\paragraph*{The matrix~$\Dbarbf$.}
%%%
The matrix~$\Dbarbf \in \mathbb R^{\NtildeE \times \ntildeE}$ can be split into two submatrices
\[
\DbarbfAB \in \mathbb R^{(\p+1) \NV  \times  \ntildeE},\quad \quad \quad \quad \; \DbarbfC \in \mathbb R^{\nPEmt \times \ntildeE}.
\]
In particular, we write
\[
\Dbarbf = 
\begin{bmatrix}
\mbox{ $\DbarbfAB$} \\
\mbox{ $\DbarbfC$} \\ 
\end{bmatrix}.
\]
Since there is no modification of the bulk degrees of freedom, we have
\[
\DbarbfC = \DbfC.
\]
Let~$j_{\e(i)}$ denote the numbering of~$\varphi_i$ as a basis function on edge~$\e(i)$.
As for the matrix~$\DbarbfAB$, we apply the definition of the new enriched edge degrees of freedom and get
\[
\begin{split}
\DbarbfAB_{i, \alpha} = \frac{1}{\h_{\e(i)}^{\frac{1}{2}}} \int_{\e(i)} \mtildealphaE \mbar_{j_{\e(i)}}^{\e(i)} = \sum_{\ell=1} ^{j_{\e(i)}}  \GSbf^\e_{j_{\e(i)}, \ell} \frac{1}{\h_{\e(i)}^{\frac{1}{2}}} \int_{\e(i)} \mtildealphaE m_{\ell}^{\e(i)}.
\end{split}
\]
All the integrals appearing on the right-hand side can be computed up to machine precision as detailed in Appendix~\ref{appendix:implementation}.

%%%%%%%%%%%%%%%%%%%%%%%%%%%%%%%%%%%%%%%%%%%%%%%%%%%%%%%%%%%%%%
\section{\textnormal{Given a polynomial of degree~$\p$ in two dimensions, how can we write it as the Laplacian of a polynomial of degree~$\p+2$?}} \label{appendix:Laplacian-polynomial}
%%%%%%%%%%%%%%%%%%%%%%%%%%%%%%%%%%%%%%%%%%%%%%%%%%%%%%%%%%%%%%
For all~$\p \in \mathbb N$, denote the set of the polynomials of degree~$\p$ in two dimensions by~$\Pbb_\p(\Rbb^2)$.
When no confusion occurs, we replace~$\Pbb_\p(\Rbb^2)$ with~$\Pbb_\p$.

In this appendix, we address the following question.
\\
\fbox{\parbox{\textwidth}{
\textbf{Question~1.}
Given~$\qp \in \Pbb_\p$, is it possible to find~$\qppt \in \Pbb_{\p+2}$ in closed form such that~$\Delta \qppt = \qp$?}}
\\
The answer to this question is crucial in the implementation of the method; see Remark~\ref{remark:Laplace-stuff}.

It suffices to answer Question~1 for all~$\qp$ being the elements of a basis of~$\Pbb_\p$.
In fact, given~$\{m_{\alpha}\}_{\alpha=1}^{\dim (\Pbb_\p)}$ a basis of~$\Pbb_\p$, we can write
\[
\qp = \sum_{\alpha=1}^{\dim(\Pbb_\p)} \lambda_\alpha m_\alpha, \quad \quad \text{where~$\lambda_\alpha \in \Rbb \quad \forall \alpha=1,\dots, \dim(\Pbb_\p)$} .
\]
Assume to know how to compute~$\mtilde_\alpha \in \Pbb_{\p+2}$ such that~$\Delta \mtilde_\alpha = m_\alpha$. Then, we have
\[
\Delta \qppt:= \Delta \left(\sum_{\alpha=1}^{\dim(\Pbb_\p)} \lambda_\alpha \mtilde_\alpha \right) =\sum_{\alpha=1}^{\dim(\Pbb_\p)} \lambda_\alpha m_{\alpha} =  \qp.
\]
As a basis of~$\Pbb_\p$, we shall consider the basis of monomials.
More precisely, given~$\p \in \mathbb N$, define the monomials of degree exactly equal to~$\p$ as
\[
\mp_s = x^{\p+1-s} y^{s-1} \quad \quad \forall s=1,\dots, \p+1,
\]
and define the space of monomials of degree exactly equal to~$\p$ as
\begin{equation} \label{monomial:spaces}
\Mbb_\p := \text{span} \left\{ \mp_s \text{ with } s=1,\dots, \p+1       \right\}.
\end{equation}
The main result reads as follows.
\begin{thm} \label{thm:main}
({i}) Let~$\p \in \mathbb N$, $m_s^{[\p]} = x^{\p+1-s} y^{s-1} \in \Mbb _\p$ be such that~$\p+1 \ge 2s$,
and~$\Ks \in \mathbb N$ be the largest integer such that~$s-2\Ks \ge 1$.
Then, the following identity is valid:
\begin{equation} \label{main:splitting}
\mp_s = \sum_{k=0}^{\Ks} \lambda^{[\p], s}_{s-2k} \Delta m^{[\p+2]}_{s-2k} \quad \quad \quad \forall s=1,\dots, \left\lceil \frac{\p+1}{2} \right\rceil.
\end{equation}
The coefficients~$\lambda^{[\p], s}_t$ appearing in~\eqref{main:splitting} are defined as follows: for all $s=1,\dots, \left\lceil \frac{\p+1}{2} \right \rceil$,
\[
\lambda^{[\p], s}_{s-2k} = (-1)^k \frac{1}{(p - s +3) (p - s +2)} \prod_{j=1}^{k} \frac{(s -1-2(j-1))(s-2-2(j-1))}{(\p-s+3+2j) (\p-s+2+2j)} \quad \quad \forall k=1,\dots, \Ks.
\]
We use the notation~$\prod_{j=1} ^0 \mu_j =1$.\\
({ii}) The following identity is valid:
\begin{equation} \label{main:splitting-2}
\mp_{\p+2-s} = \sum_{k=0}^{\Ks} \lambda^{[\p], s}_{s-2k} \Delta m^{[\p+2]}_{\p + 4 - s + 2k} \quad \quad \quad \forall s=1,\dots, \left\lfloor \frac{\p+1}{2} \right\rfloor.
\end{equation}
\end{thm}
\begin{proof}
We only prove~(\emph{i}), for~(\emph{ii}) follows likewise.

The first two coefficients~$\lambda^{[\p],s}_1$ and~$\lambda^{[\p],s}_2$ can  be computed by hand easily.
As for the others, we proceed by induction.
More precisely, given~$s \ge 3$, assume that~\eqref{main:splitting} is valid for monomials of the form~$m^{[\p]}_{_{\widetilde s}}=x^{\p+1-\widetilde s} y^{\widetilde s-1}$ with~$\widetilde s \le s$.
We prove the assertion for the monomial~$m^{[\p]}_{_s}=x^{\p+1-s} y^{s-1}$.

Let~$\Ksmt \in \mathbb N$ be the largest integer such that~$s-2-2\Ksmt\ge1$.
Thanks to the induction hypothesis
\[
m^{[\p]}_{s-2} = \sum_{k=0}^{\Ksmt} \lambda^{[\p], s-2-2k}_t \Delta m^{[\p-2]}_{s-2-2k},
\]
we write
\begin{equation} \label{an:expansion}
\begin{split}
& m^{[\p]}_{s} = x^{\p+1-s} y^{s-1} 	\\
&  = \frac{1}{(\p+1-s+2)(\p+1-s+1)} \left[ \Delta (x^{\p+1-s} y^{s-1}) - (s-1)(s-2) x^{\p+1-s+2} y^{s-3}  \right]\\
& = \frac{1}{(\p-s+3)(\p-s+2)} \left[ \Delta (x^{\p+1-s+2} y^s) - (s-1)(s-2) m^{[\p]}_{s-2}   \right] \\
& =: \frac{1}{(\p-s+3)(\p-s+2)}\Delta(x^{\p+1-s+2} y^s) -  \frac{(s-1)(s-2)}{(\p-s+3)(\p-s+2)}\sum_{k=0}^{\Ksmt} \lambda^{[\p], s-2}_{s-2-2k} \Delta m^{[\p+2]}_{s-2k}.
\end{split}
\end{equation}
We prove that the coefficients on the right-hand side of~\eqref{an:expansion} are those given in the assertion of the theorem:
they must be equal to~$\lambda^{[\p],s}_{s-2k}$ for all $k=0,\dots,\Ks$.

On the one hand, we have that the coefficient associated with~$\Delta(x^{\p+1-s+2} y^{s-1})$ is
\[
\frac{1}{(\p-s+3)(\p-s+2)},
\]
which is nothing but~$\lambda^{[\p], s}_{s-2k}$.

As for the other coefficients, we have to show that
\[
- \frac{(s-1)(s-2)}{(\p-s+3)(\p-s+2)} \lambda^{[\p], s-2}_{s-2-2k}= \lambda^{[\p], s}_{s-2(k+1)} \quad \quad \forall k =0,\dots, \Ksmt.
\]
To this purpose, for all~$k =0,\dots, \Ksmt$, we observe that
\[
\begin{split}
& - \frac{(s-1)(s-2)}{(\p-s+3)(\p-s+2)} \lambda^{[\p], s-2}_{s-2-2k} \\
& = -  \frac{(s-1)(s-2)}{(\p-s+3)(\p-s+2)} (-1)^k \frac{1}{(\p-(s-2)+3) (\p - (s-2) +2)} \times \\
& \quad \quad \quad \quad \quad \quad \quad \quad \quad \quad \quad \quad \quad \quad  \times \prod_{j=1} ^{k} \frac{ ((s-2) - 1 -2(j-1) )  ((s-2)-2-2(j-1))}{(\p-(s-2)+3+2j) (\p-(s-2)+2+2j)}\\
& = (-1)^{k+1}  \frac{1}{(\p-s+3)(\p-s+2)}  \prod_{j=1} ^{k+1} \frac{ (s - 1 -2(j-1) )  (s-2-2(j-1))}{(\p- s +3+2j) (\p s+2+2j)} = \lambda^{[\p], s}_{s-2(k+1)} , \\
\end{split}
\]
whence the assertion follows.
\end{proof}

An immediate consequence of Theorem~\ref{thm:main} is the following well known result.
\begin{cor} \label{corollary:surjectivity}
For all~$\p\in \Nbb$, the Laplace operator~$\Delta$ is surjective from~$: \Pbb_{\p+2} (\Rbb^2)$ into~$\Pbb_\p(\Rbb^2)$.
\end{cor}

Next, we present a MatLab script that, given the degree of~$\Mbb_\p$ defined in~\eqref{monomial:spaces}, allows for the computations of the coefficients in expansion~\eqref{main:splitting}.
The script can be found in Algorithm~\ref{algo:script}.

\begin{algorithm}[b] 
\begin{verbatim}
function [M_x,M_y]=polynomial_laplacian(p)
%%
if ceil(p/2)==p/2
    dim_M=ceil(p/2)+1;
else
    dim_M=ceil(p/2);
end
%%
M=zeros(dim_M,dim_M);
%%
for n=1:dim_M    
    M(n,n) = 1/((p+2-(n-1))*(p+1-(n-1)));    
    for m=n-2:-2:1
        M(n,m)=-((n-1)*(n-2))/((p+2-(n-1))*(p+1-(n-1))) * M(n-2,m);
    end
end
%%
M_x = M;
if ceil(p/2)==p/2
    M_y = fliplr(flipud(M(1:end-1,1:end-1)));
else
    M_y = fliplr(flipud(M));
end
%%
return
\end{verbatim} 
\caption[]{Computing the coefficients in the expansions~\eqref{main:splitting} and~\eqref{main:splitting-2}.}
\label{algo:script}
\end{algorithm}
%}}

The output of the script consists of two matrices.
The first matrix contains the coefficients of the expansion of~$x^n y^m$, $n\ge m$, in terms of the Laplacian of monomials of the form~$x^{\ntilde+2} y^{\mtilde}$, with~$\ntilde \ge \mtilde$;
the second matrix contains the coefficients of the expansion of~$x^n y^m$, $m > n$, in terms of the Laplacian of monomials of the form~$x^{\ntilde+2} y^{\mtilde}$, with~$\mtilde=\p+2-\ntilde, \dots,\p+2$.

\begin{exmp}
Consider the case~$\p=8$. The output of the script is provided by the two following matrices:
\begin{verbatim}
1/90   	           0             0              0          0
0                  1/72          0 			            0  		       0
-2/(56 * 90)       0             1/56           0  	       0
0                  -6/(42*72)    0              1/42      	0
(12*2)/(30*56*90)  0             -12/(30*56)   	0          1/30
\end{verbatim}
and
\begin{verbatim}
1/42    0      -6/(42*72)    0
0       1/56   0             -2/(56 * 90)  
0       0      1/72          0
0       0      0             1/90   	           
\end{verbatim}
Indeed, it can be checked that this is the correct output. In fact, we have
\[
\begin{split}
& x^8 = \frac{1}{90} \Delta x^{10}, \\
& x^7 y   = \frac{1}{72} \Delta x^9 y,\\
& x^6 y^2 =  \frac{1}{56} \Delta x^8y^2 - \frac{1}{56*90} \Delta x^8 ,\\
& x^5 y^3 =  \frac{1}{42} \Delta x^7 y^3 - \frac{6}{42* 72} \Delta x^7 y ,\\
& x^4 y^4 =  \frac{1}{30} \Delta x^6 y^4 - \frac{12}{30*56} \Delta x^8y^2 + \frac{12*2}{30*56*90} \Delta x^8 ,\\
& x^3 y^5 = \frac{1}{42} \Delta x^3 y^7 - \frac{6}{42*72} \Delta y^8 ,\\
& x^2 y^6 =  \frac{1}{56} \Delta x^2 y^8 - \frac{2}{56*90} \Delta x y^7,\\
& x y^7 = \frac{1}{72}  \Delta x y^9 , \\
& y^8 = \frac{1}{90} \Delta y^{10} .  \\
\end{split}
\]

\end{exmp}

\end{appendices}

%%%%%%%%%%%%%%%%%%%%%%%%%%%%%%%%%%%%%%%%%%%%%%%%%%%%%%%%%%%%%%%%%%%%%%%%%%%

%%%%%%%%%%%%%%%%%%%%%%%%%%%%%%%%%%%%%%%%%%%%%%%%%%%%%%%%%%%%%%%%%%%%%%%%%%%
{\footnotesize
\bibliography{bibliogr}
}
\bibliographystyle{plain}
%%%%%%%%%%%%%%%%%%%%%%%%%%%%%%%%%%%%%%%%%%%%%%%%%%%%%%%%%%%%%%%%%%%%%%%%%%%

\end{document}